\documentclass[11pt,oneside,reqno]{amsart}

\usepackage{amsfonts,amsmath,amssymb} 
\usepackage[pdftex]{graphicx} 
\usepackage[utf8]{inputenc}
\usepackage[english]{babel}
\usepackage{color}   
\usepackage[dvipsnames]{xcolor}
\usepackage{mathtools}
\usepackage{mathrsfs} 
\usepackage[T1]{fontenc}
\usepackage{lmodern}
\usepackage[hmarginratio=3:2]{geometry} 
\usepackage[ruled, lined, longend, linesnumbered]{algorithm2e}
\SetKwProg{testwhile}{while}{ do}{~~~~~~end while}
\SetKwProg{testif}{if}{ then}{~~~~~end if}
\SetKwProg{testift}{if}{ then}{ }
\usepackage{longtable} 
\usepackage[margin=1cm]{caption} 
\usepackage[rightcaption,raggedright]{sidecap}
\usepackage{framed} 
\usepackage{soul} 
\usepackage{float} 
\usepackage[section]{placeins}
\usepackage{comment}
\usepackage{booktabs,array} 

\usepackage{chngcntr} 
\usepackage{tikz}

\usepackage{fancyhdr}

\usepackage[shortlabels]{enumitem}

\usepackage{framed} 

\usepackage{bm}

\usepackage[acronym, sort=use, nopostdot, nonumberlist, hyperfirst=false]{glossaries}

\usepackage{pgfplotstable}

\DeclareMathAlphabet\mathbfcal{OMS}{cmsy}{b}{n}
\usepackage{lipsum}
\usepackage{upgreek}
\usepackage{tikz,pgf}
\usepackage{lscape}
\usepackage{algorithm2e}
\usepackage{listings}
\usepackage{xcolor}
\usepackage{multicol}
\usepackage{physics}
\usepackage{textcomp,mathcomp}
\usepackage{pgfplots}
\usepgfplotslibrary[groupplots]	 
\definecolor{bown}{rgb}{0.0, 0.55, 0.55}
\definecolor{bref}{rgb}{0.28, 0.24, 0.55}
\definecolor{error}{rgb}{0.82, 0.41, 0.12}
\definecolor{maroon}{rgb}{0.69, 0.19, 0.38}
\definecolor{olive}{rgb}{0.5, 0.5, 0.0}
\definecolor{cadmium}{rgb}{0.0, 0.42, 0.24}

\newcommand{\Ham}{\mathcal{H}}
\newcommand{\eff}{\tilde{e}}
\newcommand{\Bflo}{f_\mathrm{B}}
\newcommand{\Beff}{e_\mathrm{B}}
\newcommand{\Sflo}{f_\mathrm{S}}
\newcommand{\Seff}{e_\mathrm{S}}
\newcommand{\Rflo}{f_\mathrm{R}}
\newcommand{\Reff}{e_\mathrm{R}}
\newcommand{\Pflo}{f_\mathrm{P}}
\newcommand{\Peff}{e_\mathrm{P}}

\newcommand{\tBflo}{\tilde{f}_\mathrm{B}}
\newcommand{\tBeff}{\tilde{e}_\mathrm{B}}
\newcommand{\tSflo}{\tilde{f}_\mathrm{S}}
\newcommand{\tSeff}{\tilde{e}_\mathrm{S}}
\newcommand{\tRflo}{\tilde{f}_\mathrm{R}}
\newcommand{\tReff}{\tilde{e}_\mathrm{R}}
\newcommand{\tPflo}{\tilde{f}_\mathrm{P}}
\newcommand{\tPeff}{\tilde{e}_\mathrm{P}}

\newcommand{\Seffa}{e_{\mathrm{S}_1}}

\newcommand{\Seffb}{e_{\mathrm{S}_2}}

\newcommand{\Seffc}{e_{\mathrm{S}_3}}

\newcommand{\tSfloa}{\tilde{f}_{\mathrm{S}_1}}
\newcommand{\tSeffa}{\tilde{e}_{\mathrm{S}_1}}
\newcommand{\tSflob}{\tilde{f}_{\mathrm{S}_2}}
\newcommand{\tSeffb}{\tilde{e}_{\mathrm{S}_2}}
\newcommand{\tSfloc}{\tilde{f}_{\mathrm{S}_3}}
\newcommand{\tSeffc}{\tilde{e}_{\mathrm{S}_3}}

\newcommand{\tReffa}{\tilde{e}_{\mathrm{R}_1}}
\newcommand{\tReffb}{\tilde{e}_{\mathrm{R}_2}}
\newcommand{\tReffc}{\tilde{e}_{\mathrm{R}_3}}

\newcommand{\BfloL}{f_\mathrm{B}|_L}
\newcommand{\BeffL}{e_\mathrm{B}|_L}
\newcommand{\BfloZ}{f_\mathrm{B}|_0}
\newcommand{\BeffZ}{e_\mathrm{B}|_0}
\newcommand{\Bflonub}{f_\mathrm{B}^\omega |_{\bar{\nu}}}
\newcommand{\Beffnub}{e_\mathrm{B}^\omega |_{\bar{\nu}}}
\newcommand{\Bflonu}{f_\mathrm{B}^\omega |_\nu}
\newcommand{\Beffnu}{e_\mathrm{B}^\omega |_\nu}

\newcommand{\SeffL}{e_\mathrm{S}|_L}
\newcommand{\SeffZ}{e_\mathrm{S}|_0}
\newcommand{\SeffaL}{e_{\mathrm{S}_1}|_L}
\newcommand{\SeffbL}{e_{\mathrm{S}_2}|_L}
\newcommand{\SeffcL}{e_{\mathrm{S}_3}|_L}
\newcommand{\SeffaZ}{e_{\mathrm{S}_1}|_0}
\newcommand{\SeffbZ}{e_{\mathrm{S}_2}|_0}
\newcommand{\SeffcZ}{e_{\mathrm{S}_3}|_0}
\newcommand{\eZ}{e|_0}
\newcommand{\enu}{e^\omega |_\nu}
\newcommand{\phiZ}{\phi|_0}
\newcommand{\phiL}{\phi|_L}

\newcommand{\mnu}{m^\omega|_\nu}
\newcommand{\hnu}{h^\omega|_\nu}
\newcommand{\snu}{s^\omega|_\nu}

\newcommand{\vect}{\mathbf}
\newcommand{\brho}{\boldsymbol{\rho}}
\renewcommand{\bm}{\mathbf{m}}
\newcommand{\be}{\mathbf{e}}
\newcommand{\blam}{\boldsymbol{\lambda}}

\newcommand{\mat}{\mathbf}
\newcommand{\V}{\mathbf{V}}
\newcommand{\JRM}{\tilde{\mathbf{J}}_{\upm,\upe}}

\newcommand{\sz}{\mathbf{O}}
\renewcommand{\so}{\mathbf{I}}

\newcommand{\spa}{\mathcal}

\newcommand{\D}{d}

\newcommand{\upm}{\mathrm{m}}
\newcommand{\upe}{\mathrm{e}}
\newcommand{\uph}{\mathrm{h}}
\newcommand{\upv}{\mathrm{v}}

\newcommand{\upc}{\mathrm{c}}

\newcommand{\upr}{\mathrm{r}}
\newcommand{\upH}{\mathrm{H}}

\newcommand{\upM}{\mathrm{M}}
\newcommand{\upS}{\mathrm{S}}

  {\bgroup\color{#1}}{\egroup\ignorespacesafterend}

 \newcommand{\Net}{\mathcal{E}}
\usepackage{ifthen}
\usepackage[T1]{fontenc}
\usepackage{ae,aecompl}
\usepackage{amsmath,amsfonts,amssymb}
\usepackage{color}
\usepackage[utf8]{inputenc}
\usepackage{graphicx}
\usepackage{setspace}
\usepackage{tikz}
\usepackage{pgfplots}
\usetikzlibrary{calc}
\usepackage{fancyhdr}
\usepackage{footnote}
\usepackage{todonotes}
\usepackage{nicefrac} 
\usepackage{bm}
\usepackage{comment} 
\usepackage{csquotes}
\usetikzlibrary{plotmarks}
\usepackage{setspace}
\usepackage{xcolor}
\usepackage{enumitem}
\pgfplotsset{compat=1.7}
\usetikzlibrary{pgfplots.groupplots}
\usetikzlibrary{spy}
\usetikzlibrary{plotmarks}
\newtagform{simple}{}{}{}

\newtheorem{theorem}{Theorem}

\newtheorem{corollary}[theorem]{Corollary}

\newtheorem{definition}[theorem]{Definition}
\newtheorem{remark}[theorem]{Remark}
\newtheorem{assum}[theorem]{Assumption}
\newtheorem{coordinate}[theorem]{Coordinate Representation}
\newtheorem{weak}[theorem]{Weak Formulation}
\newtheorem{close}[theorem]{Closure}
\newtheorem{algo}[theorem]{Algorithm}

\def\thname{name}
\newtheorem{mysystem}{System \thname{}}
\newenvironment{msystem}[1][name]{\def\thname{#1}\mysystem}{\endmysystem}

\definecolor{mp2}{rgb}{0.90, 0.0, 0.48}
\definecolor{bg}{rgb}{0.0, 0.87, 0.87}
\definecolor{cin}{rgb}{0.89, 0.26, 0.2}
\definecolor{cy}{rgb}{1.0, 0.65, 0.0}
\definecolor{mauve}{rgb}{0.88, 0.69, 1.0}
\definecolor{jon}{rgb}{0.98, 0.85, 0.37}

\definecolor{b1}{rgb}{0.28, 0.24, 0.55}
\definecolor{b2}{rgb}{0.08, 0.38, 0.74}
\definecolor{b3}{rgb}{0.0, 0.72, 0.92}
\definecolor{b4}{rgb}{0.0, 0.90, 1.0}
\definecolor{b5}{rgb}{0.61, 0.87, 1.0}

\definecolor{ar1}{rgb}{1.0, 0, 0.2}
\definecolor{ar2}{rgb}{0.8, 0, 0.4}
\definecolor{ar3}{rgb}{0.6, 0, 0.6}
\definecolor{ar4}{rgb}{0.4, 0, 0.8}
\definecolor{ar5}{rgb}{0.2, 0, 1.00}

\definecolor{amber}{rgb}{0.91, 0.41, 0.17}
\definecolor{applegreen}{rgb}{0.55, 0.71, 0.0}
\definecolor{awesome}{rgb}{1.0, 0.13, 0.32}
\definecolor{capri}{rgb}{0.0, 0.75, 1.0}
\definecolor{darkmagenta}{rgb}{0.55, 0.0, 0.55}
\definecolor{etonblue}{rgb}{0.59, 0.78, 0.64}

\begin{document}
\title[Thermodynamic pH-system]{On a port-Hamiltonian formulation and structure-preserving numerical approximations for thermodynamic compressible fluid flow}
\author[Hauschild]{Sarah-Alexa Hauschild$^1$}
\author[Marheineke]{Nicole Marheineke$^{1,\star}$}
\date{\today\\
$^1$ Universit\"at Trier, Arbeitsgruppe Modellierung und Numerik, Universit\"atsring 15, D-54296 Trier, Germany\\
$^\star$ corresponding author, email: marheineke@uni-trier.de, orcid: 0000-0002-5912-3465
}

\begin{abstract}
The high volatility of renewable energies calls for more energy efficiency. Thus, different physical systems need to be coupled efficiently although they run on various time scales. Here, the port-Hamiltonian (pH) modeling framework comes into play as it has several advantages, e.g., physical properties are encoded in the system structure and systems running on different time scales can be coupled easily. Additionally, pH systems coupled by energy-preserving conditions are still pH. Furthermore, in the energy transition hydrogen becomes an important player and unlike in natural gas, its temperature-dependence is of importance. Thus, we introduce an infinite dimensional pH formulation of the compressible non-isothermal Euler equations to model flow with temperature-dependence. We set up the underlying Stokes-Dirac structure and deduce the boundary port variables.
We introduce coupling conditions into our pH formulation, such that the whole network system is pH itself. This is achieved by using energy-preserving coupling conditions, i.e., mass conservation and equality of total enthalpy, at the coupling nodes. Furthermore, to close the system a third coupling condition is needed. Here, equality of the outgoing entropy at coupling nodes is used and included into our systems in a structure-preserving way. Following that, we adapt the structure-preserving aproximation methods from the isothermal to the non-isothermal case. Academic numerical examples will support our analytical findings.
\end{abstract}



\maketitle

\section{Introduction}
\noindent A very important part of a successful energy transition is an increasing supply of renewable energies. However, the power supply through such energies is highly volatile. That is why a balancing of this volatility and more energy efficiency is needed. To store such superfluous energy Power-to-X technologies come to mind, where electrical power is converted into heat or gas, which are then stored in the corresponding facilities, e.g., district heating or gas networks. Translating this into mathematics, in order to simulate and optimize in advance, a crucial point here is the coupling of the various physical systems, which run on different time scales. Furthermore, within Power-to-Gas hydrogen plays a growing role. Here, electrical power is converted into hydrogen by electrolysis, which can then be injected into and stored in networks built for natural gas. In contrast to natural gas in pipelines the temperature of the gas becomes more important for the modeling, simulation and optimization when working with mixtures of hydrogen, see \cite{Cle22}. Therefore, the gas flow in the network needs to be modeled with, e.g., the non-isothermal Euler equations, which is rarely done in literature, see \cite{BerLV16} as one of the few publications. Our main goal is to tackle both of these mathematical difficulties, i.e., the efficient coupling and the need for temperature-dependence. The port-Hamiltonian (pH) modeling framework has lately been widely used in the modeling of energy networks, as it has various advantages. As energy is used as a lingua franca, it brings the different scales on a single level, e.g., gas, power and district heating. This makes the coupling of these individual systems easier. Furthermore, when using energy preserving coupling the pH character is inherited during the coupling of the individual systems. Additionally, physical principles like passivity, conservation of energy and mass are ideally encoded in the algebraic and geometric structures of the model. Introductory work on primarily finite dimensional and linear pH differential(-algebraic) systems can be found in  \cite{BeaMXZ18}, \cite{SchJ14}, and \cite{MehS23}. A survey on literature on pH differential-algebraic systems is given in \cite{MehU23}. We will work with the definition  of a finite-dimensional pH system found in \cite{MehM19}, which can be extended to the case of weak solutions and state spaces with infinite dimension. Finite dimensional differential equations often arise from a space discretization of partial differential equations. If so, the pH structure of the underlying partial differential equations needs to be preserved in order to fully utilize the advantages of this framework. Thus, establishing a pH formulation of partial differential equations and the structure-preserving approximation are current fields of research in different application areas. Fundamental results on linear infinite dimensional pH systems are given in \cite{JacZ12}, \cite{LeGZM05}, \cite{SchM02} and \cite{PhiRS23}. Nevertheless,  there is plenty of literature also on the modeling of non-linear infinite dimensional pH systems in various application areas, e.g., fluid flow \cite{AltS17}, \cite{BanSAZISW} and \cite{BanZISW21}, heat transfer \cite{JaesEGJ22} and \cite{SerHM19}, flexible multibody dynamics \cite{BruPM21} and poroelastic networks \cite{AltMU21}. For structure-preserving space discretization there are different approaches in the pH community. There is discrete exterior and finite element exterior calculus, when the modeling is based on differential forms, see \cite{KotML18}, \cite{SchTG02} and \cite{SchSS13}. Furthermore, there are also structure-preserving finite volume methods, \cite{Kot16} and \cite{SerMH18}. As the finite element method is often the method of choice for space discretization, variants of it have been developed in the pH framework, e.g., the mixed finite element method, see \cite{BanWISW20}, \cite{KotT21}, or the partitioned finite element method, see \cite{CardML18} and \cite{SerMH19}. In some applications the finite element spaces have to be chosen problem-dependent, i.e., they have to fulfill some compatibility conditions in order to preserve the pH structure, e.g., as for the isothermal Euler equations \cite{LilSM20}. If the applications lead to large finite dimensional systems, model order reduction needs to be applied to simulate the system. Here, we do not only want a good approximation of the dynamics of the full order model, but also preserve the pH structure. For linear pH differential(-algebraic) systems there is a vast amount of methods based on different ansatzes available, e.g., balancing based \cite{BreS21} and \cite{PolS10}, interpolation based \cite{BeaGM22} and \cite{GugPBS12}, Krylov method based \cite{HauMM19b}, \cite{PolS11} and \cite{LohWEK10}, based on spectral factorization \cite{BreU22}, and based on the effort and flow structure \cite{HauMM19b} and \cite{PolS12}. There are also special methods available for differential-algebraic systems of index one \cite{SchMMV22} and index two \cite{MosSMV22}. Systems with a non-linearity only in the Hamiltonian are subject to symplectic model order reduction, see \cite{AfkH19} and references therein, and Petrov Galerkin model order reduction \cite{ChaBG16}. The literature for model order reduction methods which are applicable to general fully non-linear (port-)Hamiltonian system is unfortunately less rich. The paper \cite{HesP20} deals with non-linear Hamiltonian systems without dissipation, whereas \cite{IonA13} and \cite{Sch23} are concerned with non-linear port-Hamiltonian systems with dissipation. There are also problem-dependent model reduction ansatzes, e.g., for isothermal gas networks \cite{EggKLSMM18} and \cite{LilSM21}, the non-linear Schrödinger equation \cite{KarU18} and the shallow water equations \cite{KarYU21}. Last but not least, structure-preserving time integration is also a wide field of ongoing research. There are ansatzes based on symplectic time integrators \cite{LefM17}, collocation methods \cite{MehM19} and \cite{KotL18} and operator splitting \cite{FroGLSM23}. Recently, in the context of modeling thermodynamic pipe flow the GENERIC framework has been used, see \cite{BadMBM18} and \cite{BadZ18}. Here, a generalized (non-linear) pH system in operator form is set up, i.e.,
\begin{align*}
\frac{\D z}{\D t}&=\left(J(z)-R(z)\right)\frac{\delta\mathcal{E}(z)}{\delta z}+B(z)u(z)\quad &\text{in }  \mathcal{D}_z^*,\\
y(z)&=B^*(z)\frac{\delta\mathcal{E}(z)}{\delta z} \quad &\text{in } \mathcal{D}_u^*,
\end{align*}
with state $z$. Here, $\frac{\delta\mathcal{E}(z)}{\delta z}$ denotes the Frech\`et derivative of the operator $\mathcal{E}(z)$ with respect to $z$, $\mathcal{D}_z^*$ and $\mathcal{D}_u^*$ are the dual spaces to the state space $\spa{D}_z$ and the input space $\spa{D}_u$. Furthermore, $B^*(z)$ denotes the adjoint of the input operator $B(z)$, see \cite{HauMMMMBRS19} for more details on the system structure and function spaces in the case of the non-isothermal Euler equations.
Accounting for the thermodynamic behavior of the pipe flow, this system is composed of a Hamiltonian and a generalized gradient system. This is reflected an exergy-like functional consisting of a Hamiltonian $\Ham(z)$ and an entropy part $\mathcal{S}(z)$, i.e.,
\begin{align*}
\mathcal{E}(z)=\Ham(z)-\theta\mathcal{S}(z),
\end{align*}
where $\theta$ is a constant, chosen such that the units of $\Ham$ and $\theta\mathcal{S}$ match. These functionals have to fulfill the non-interacting conditions
\begin{align*}
J(z)\frac{\delta\mathcal{S}(z)}{\delta z}=0, \quad \quad R(z)\frac{\delta\Ham(z)}{\delta z}=0.
\end{align*}
These conditions ensure that the flows of the Hamiltonian and the gradient system do not overlap. 
This framework is closely related to the use of Poisson and dissipation brackets within modeling thermodynamics, see \cite{OetG97} and \cite{OetG97b}. It has been successfully used in combination with structure-preserving approximation for the Vlasov-Maxwell system in \cite{KraH17} and \cite{KraKMS17}. The exergetic modeling framework is still a topic of ongoing research, see \cite{LohKL21}. Nevertheless, this framework has disadvantages regarding the structure-preserving approximation with respect to our application, e.g., the non-interacting conditions. Thus, we refrain from working with this and rather take advantage of the well studied pH framework.  Thus, we first introduce a new infinite dimensional pH formulation for the non-isothermal Euler equations in Section \ref{sec:pHForm}, which depends only on the Hamiltonian. In Section \ref{Sec:Ports} we deduce a boundary port and prove that the underlying structure is indeed a Stokes-Dirac structure. The boundary port allows for structure-preserving incorporation of boundary and coupling conditions, which not only are energy-preserving, but also suit the thermodynamics of the system, see Sections \ref{Sec:WeakForm} and \ref{Sec:NetPH}. See \cite{LanM18} for the recently found well-defined coupling conditions for the non-isothermal Euler equations. Based on the literature for isothermal gas flow, e.g., \cite{LilSM20}, \cite{LilSM21}, we set up numerical methods for space-discretization, model order and complexity reduction in Section \ref{Sec:SPApprox}, which not only preserve the pH structure, but also suit the thermodynamics of the system. The latter two procedures become extremely important when simulating large networks. This overall leads to a novel pH formulation of compressible non-isothermal gas flow in networks of pipes with its own tailored structure-preserving numerical methods, which can now be easily coupled to existing pH formulations of power and isothermal gas networks. Numerical examples in Section \ref{Sec:Num} show the interesting influence of the interconnection operator $J(z)$ on model order reduction. We further compare pipe-wise and network-wise reduction of the gas networks. The paper is concluded with a short summary in Section \ref{Sec:Con}. Several coordinate representations are given in Appendix \ref{Ap:CoRep}. The derivation of the boundary port and the proof for the Stokes-Dirac structure are moved to Appendix \ref{Appendix2} for better readability.

\section{Port-Hamiltonian Framework}\label{sec:pHForm}
\subsection{Port-Hamiltonian Formulation of Compressible Non-Isothermal Pipe Flow}
We set up infinite dimensional pH formulations for the compressible non-isothermal Euler-type equations. This includes the underlying Stokes-Dirac structure and the boundary port with its flow and effort variables. Lastly, we add boundary conditions in a structure-preserving way to our weak pH formulation, which is essential for structure-preserving space discretization, model order reduction and even coupling. For the beginning we consider compressible fluid flow through a single pipe $\omega=(0,L)$ with $L>0$ being the length of the pipe. We make use of the ideal gas laws and therefore, make the following assumption.
\begin{assum}\label{as:IdealGas}
We assume that the gas in our network is an ideal gas, such that we have,
\begin{align*}
p=RT\rho,\quad \epsilon=c_vT,\quad s=c_v\operatorname{ln}(\frac{p}{\rho^\gamma}),\quad E=\rho(\epsilon+\frac{v^2}{2}) ,\quad  h=\frac{E+p}{\rho}.
\end{align*}
Here, $p$, $T$, $\rho$, $\epsilon$ and $E$ denote the pressure, temperature, mass density, the internal energy, and the total energy, respectively. Furthermore, $s$ is the specific entropy and $h$ the total specific enthalpy. The specific gas constant $R$ and the heat capacities $c_p$ and $c_v$ of the considered ideal gas are related to each other, i.e., $c_p=R+c_v$. Additionally, the adiabatic index is given by $\gamma=\frac{c_p}{c_v}$.
\end{assum}
\begin{msystem}[$E^v$]
\noindent For $t>0$, $x\in\omega$ the compressible non-isothermal Euler equations are of the form,
\begin{align}\label{eq:DCv} 
    0&=\partial_t\rho+\partial_x(\rho v),\nonumber\\
    0&=\partial_tv+\partial_x(\frac{v^2}{2})+\frac{1}{\rho}\partial_x p+\frac{\lambda}{2d}|v|v,\nonumber\tag*{$E^v$}\\
    0&=\partial_t e+\partial_x(ev)+p\partial_xv-\frac{\lambda}{2d}\rho|v|v^2+\frac{k_\omega}{d}(T-T_\infty).\nonumber
\end{align}
\end{msystem}
\noindent Here, $v$  and $e=\rho \epsilon$ denote the velocity and the internal energy density, respectively. Pressure $p$ and temperature $T$ are prescribed in Assumption \ref{as:IdealGas}. Furthermore, the specifics of the pipe are given by its diameter $d>0$, the friction factor $\lambda\geq 0$ and the thermal conductivity coefficient $k_\omega\geq 0$. The ambient temperature of the pipe is denoted by $T_\infty\geq0$. Naturally, there are different choices for the state variables. For the following analytical considerations we choose to work with the mass density $\rho$, the velocity $v$ and the internal energy density $e$, as this simplifies the procedure. Nevertheless, other state variables will prove advtangeous for the numerical treatment, see Section \ref{Sec:WeakForm}, but even after a change of variables the following theoretical findings still hold. To close \ref{eq:DCv} we not only need state equations for the pressure and temperature and initial conditions, but also boundary conditions. As the compressible non-isothermal Euler equations are a hyperbolic system of three partial differential equations, we need to take special care of the boundary conditions. In the following, we make use of the assumption 
\begin{assum}\label{as:SubSonic}
We assume subsonic flow, i.e., $|v|<c=\sqrt{\frac{\gamma p}{\rho}}$. Here, $c$ denotes the speed of sound.
\end{assum} 
\noindent Under this assumption, we have to state two boundary conditions at the inflow and one at the outflow of the pipe, see \cite{Egg16}, i.e., three boundary conditions in total. In the following, we assume that three consistent boundary conditions are given. In Section \ref{sec:BoundaryConditions} we show how a certain kind of boundary conditions is incorporated into the pH formulation via Lagrange multipliers.
\begin{remark}
The state variables $\rho$, $v$ and $e$ are obviously time- and space-dependent. We suppress these dependencies to enhance readability. The same applies for other quantities. If it is not clear through the context, there will be remarks on the time- and space-dependencies. 
\end{remark}
\noindent Following the ideas of \cite{MehM19} for finite dimensional pH systems, we embed \ref{eq:DCv} into the pH formalism. We seek a system of the form
\begin{align*}
E(z)\partial_tz&=(J(z)-R(z))\eff(z)+Bu,\\
y&=B^T\eff(z),\nonumber
\end{align*}
with $J(z)$ skew-adjoint and $R(z)$ self-adoint semi-elliptic in the $\spa{L}^2$ inner-product, fulfilling the condition 
\begin{align}\label{Cond:Volker}
E(z)^T\eff(z)=\frac{\delta\Ham}{\delta z}(z).
\end{align}
\begin{msystem}[$E^v_{pH}$]
A port-Hamiltonian formulation of \ref{eq:DCv} for the  state $z=[\rho\quad v\quad e]^T$ is given by
\begin{align*}
E(z)\partial_tz&=(J(z)-R(z))\eff(z)+Bu,\\
y&=B^T\eff(z),
\end{align*}
with effort function $\eff(z)=[\frac{v^2}{2}\quad \rho v\quad 1]^T$ and system operators
\usetagform{simple}
\begin{align}\label{Sys:pHDCv}\tag{$E^v_{pH}$}
\begin{split}
E(z)&=\begin{bmatrix} 1&0&0\\ 0 & 1 &0\\0&0&1\end{bmatrix},\quad R(z)=\begin{bmatrix}0&0&0\\0&0&0\\0&0&\frac{k_\omega}{d}T\end{bmatrix},\quad B=\begin{bmatrix}0\\0\\\frac{k_\omega}{d}\end{bmatrix},\\
J(z)&=\begin{bmatrix}0&-D_x&0\\-D_x&0&-\frac{\lambda}{2d} v|v|-\frac{e}{\rho}D_x-\frac{1}{\rho}D_xp\\0&\frac{\lambda}{2d}v|v|-D_x\frac{e}{\rho}-pD_x\frac{1}{\rho}&0
\end{bmatrix}.
\end{split}
\end{align}
The input is given as $u=T_\infty$. The corresponding Hamiltonian is $$\Ham(z)=\int_\omega\rho\frac{v^2}{2}+e\,\D x.$$ Furthermore, the operator $D_x$ is defined for smooth enough functions $g_1,\,g_2,\,g_3$ by $(g_1D_xg_2)g_3:=g_1\partial_x(g_2 g_3)$.
\end{msystem}
\noindent The above system is derived as follows. Starting from \ref{eq:DCv} with Hamiltonian $\Ham(z)=\int_\omega\rho\frac{v^2}{2}+e\D x$, we can deduce the corresponding effort variables, which play an important role in the pH framework. They are given by the Frech\`et derivative of the Hamiltonian,
\begin{align*}
\frac{\delta\Ham}{\delta (\rho,v,e)}(\rho,v,e)=\left[\frac{\delta\Ham}{\delta \rho}\quad \frac{\delta\Ham}{\delta v}\quad \frac{\delta\Ham}{\delta e}\right]^T=\left[\Ham_\uprho\quad\Ham_\upv\quad\Ham_\upe\right]^T=\left[\frac{v^2}{2}\quad \rho v \quad 1\right]^T. 
\end{align*}
The first step towards a pH formulation is to create a dependency on the effort variables. This can be easily done, as we only need division by $\rho$ or multiplication with $\Ham_\upe$ in certain places, i.e.,
\begin{align*}
\partial_t\rho&=-\partial_x\Ham_\upv,\\
\partial_tv&=-\partial_x\Ham_\uprho-\frac{1}{\rho}\partial_x(\Ham_\upe p)-\frac{\lambda}{2d}v|v|\Ham_\upe,\\
\partial_te&=-\partial_x(\frac{e}{\rho}\Ham_\upv)-p\partial_x(\frac{1}{\rho}\Ham_\upv)+\frac{\lambda}{2d}v|v|\Ham_\upv-\frac{k_\omega}{d}T\Ham_\upe+\frac{k_\omega}{d}T_\infty.
\end{align*} 
We split the cooling term into a state- and effort-dependent part, which is incorporated into $R(z)$, and a state-independent part, such that $T_\infty$ acts as an input from the environment.  Taking a closer look, we notice that the counterpart to $-\partial_x(\frac{e}{\rho}\Ham_\upv)$ in the energy density equation is absent in the velocity equation, but is needed to create skew-adjointness. As this new additional term needs to be dependent on $\Ham_\upe=1$, we add a nurturing zero, i.e., $\frac{e}{\rho}\partial_x\Ham_\upe=\frac{e}{\rho}\partial_x1=0$, to the velocity equation. This yields,
\begin{align*}
\partial_t\rho&=-\partial_x\Ham_\upv,\\
\partial_tv&=-\partial_x\Ham_\uprho-\frac{e}{\rho}\partial_x\Ham_\upe-\frac{1}{\rho}\partial_x(\Ham_\upe p)-\frac{\lambda}{2d}v|v|\Ham_\upe,\\
\partial_te&=-\partial_x(\frac{e}{\rho}\Ham_\upv)-p\partial_x(\frac{1}{\rho}\Ham_\upv)+\frac{\lambda}{2d}v|v|\Ham_\upv-\frac{k_\omega}{d}T\Ham_\upe+\frac{k_\omega}{d}T_\infty,
\end{align*}
i.e., System \ref{Sys:pHDCv}. Note that we incorporate the derivative terms and the friction terms into the skew-adjoint operator of the pH system, the latter is possible due to the opposite signs of the friction terms.
Thus, $E(z)=\so_3$, such that condition \eqref{Cond:Volker} is trivially fulfilled as $\eff(z)=\frac{\delta}{\delta z}\Ham(z)$. 
\begin{corollary}\label{Cor:JRProp}
Assume that boundary terms vanish. For $\psi,\phi\in\spa{H}^1(\omega)^3$ and $\tilde{\psi},\tilde{\phi}\in\spa{L}^2(\omega)^3$ the operators
\begin{align*}
J(z)[\cdot]:\spa{H}^1(\omega)^3\rightarrow\spa{H}^1(\omega)^3\quad \text{and}\quad R(z)[\cdot]:\spa{L}^2(\omega)^3\rightarrow\spa{L}^2(\omega)^3
\end{align*}
defined in System \ref{Sys:pHDCv} are skew-adjoint and self-adjoint semi-elliptic in the $\spa{L}^2$ inner-product, respectively.
\end{corollary}
\begin{proof}
For $\tilde{\phi},\,\tilde{\psi}\in \spa{L}^2(\omega)^3$ with $\tilde{\phi}=[\tilde{\phi}_1,\,\tilde{\phi}_2,\tilde{\phi}_3]^T$ and $\tilde{\psi}=[\tilde{\psi}_1,\,\tilde{\psi}_2,\,\tilde{\psi}_3]^T$ we have that 
\begin{align*}
(\tilde{\phi},R(z)\tilde{\psi})=(\tilde{\psi},R(z)\tilde{\phi}) \quad \text{and}\quad (\tilde{\phi},R(z)\tilde{\phi})=\int_\omega \tilde{\phi}_3 \frac{k_\omega}{d}T \tilde{\phi}_3\,\D x\geq 0,
\end{align*}
since $k_\omega,\, T\geq 0$ and $d>0$, such that $R(z)[\cdot]:\spa{L}^2(\omega)^3\mapsto\spa{L}^2(\omega)^3$ is self-adjoint and semi-elliptic.\\
\noindent Assuming that the boundary terms vanish, we can show that $J(z)[\cdot]:\spa{H}^1(\omega)^3\mapsto\spa{H}^1(\omega)^3$ is skew-adjoint in the $\spa{L}^2$ inner-product. For $\phi,\,\psi\in\spa{H}^1(\omega)^3$ with $\phi=[\phi_1,\,\phi_2,\phi_3]^T$ and $\psi=[\psi_1,\,\psi_2,\,\psi_3]^T$ we show that
$(\phi,J(z)\psi)=-(J(z)\phi,\psi)$. Using integration by parts and the vanishing boundary terms, we have that
\begin{align*}
(\phi,J(z)\psi)&=-(\phi_1,\partial_x\psi_2)-(\phi_2,\partial_x\psi_1)-(\phi_2,\frac{\lambda}{2d}v|v|\psi_3)-(\phi_2,\frac{e}{\rho}\partial_x\psi_3)-(\phi_2,\frac{1}{\rho}\partial_x(p\psi_3))\\&\quad +(\phi_3,\frac{\lambda}{2d}v|v|\psi_2)-(\phi_3,\partial_x(\frac{e}{\rho}\psi_2))-(\phi_3,p\partial_x(\frac{1}{\rho}\psi_2))\\
&=(\partial_x\phi_1,\psi_2)+(\partial_x\phi_2,\psi_1)-(\frac{\lambda}{2d}v|v|\phi_2,\psi_3)+(\partial_x(\frac{e}{\rho}\phi_2),\psi_3)+(p\partial_x(\frac{1}{\rho}\phi_2),\psi_3)\\&\quad 
+(\frac{\lambda}{2d}v|v|\phi_3,\psi_2)+(\frac{e}{\rho}\partial_x\phi_3,\psi_2)+(\frac{1}{\rho}\partial_x(p\phi_3),\psi_2)=-(J(z)\phi,\psi).
\end{align*}
\end{proof}

\subsubsection{Ports and Stokes-Dirac Structure}\label{Sec:Ports}
In this section, we take a closer look at the Stokes-Dirac structure which the pH System \ref{Sys:pHDCv} is based upon and introduce a boundary port in order to pose boundary conditions in a structre-preserving way. A Stokes-Dirac structure is defined as follows, see \cite{BanSAZISW}.
\begin{figure}
\centering
\begin{tikzpicture}
\draw (2,2) ellipse (1.5cm and 0.75cm) node{$\mathcal{D}$};
\draw (0,2.1)--(0.5,2.1);
\draw (0,1.9)--(0.5,1.9);
\draw (3.5,2.1)--(4,2.1);
\draw (3.5,1.9)--(4,1.9);
\draw (2.1,2.75)--(2.1,3.25);
\draw (1.9,2.75)--(1.9,3.25);
\draw (2.1,1.25)--(2.1,0.75);
\draw (1.9,1.25)--(1.9,0.75);
Text
\node at (0.25,2.3) {$\Sflo$};
\node at (0.25,1.7) {$\Seff$};
\node at (3.75,2.3) {$\Pflo$};
\node at (3.75,1.7) {$\Peff$};
\node at (2.35,3) {$\Reff$};
\node at (1.65,3) {$\Rflo$};
\node at (1.65,1) {$\vect{\Bflo}$};
\node at (2.35,1) {$\vect{\Beff}$};
\node at (2,3.95) {energy-dissipating port}; 
\node at (2,3.55) {$\spa{F}_\mathrm{R}\times\spa{E}_\mathrm{R}$};
\node at (-1.5,2.4) {energy-saving port};
\node at (-1.5,2) {$\spa{F}_\mathrm{S}\times\spa{E}_\mathrm{S}$};
\node at (-1.5,1.6) {$\Ham(z):\spa{F}_\mathrm{S}\rightarrow\mathbb{R}$};
\node at (5.5,2.4) {external port};
\node at (5.5,2) {$\spa{F}_\mathrm{P}\times\spa{E}_\mathrm{P}$};
\node at (2,0.2) {boundary port};
\node at (2,-0.2) {$\spa{F}_\mathrm{B}\times\spa{E}_\mathrm{B}$};
\end{tikzpicture}
\caption{Stokes-Dirac structure}
\label{fig:Stokes-Dirac}
\end{figure}
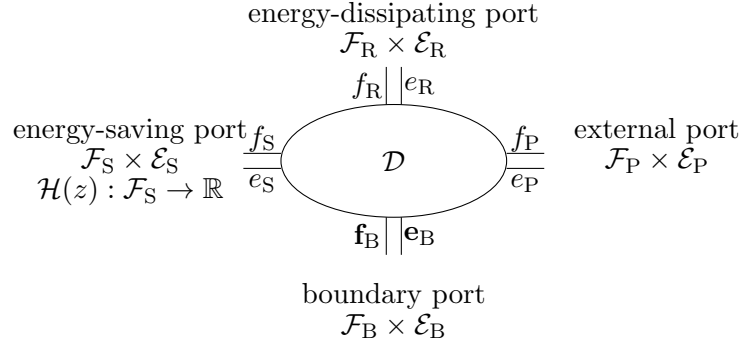
\begin{definition}\label{def:StokesDirac}
Let $\spa{F}_\omega$ and $\spa{E}_\omega$ denote the real Hilbert spaces of flow variables and effort variables defined on the domain $\omega$, respectively. Also let $\spa{F}_\mathrm{B}$ and $\spa{E}_\mathrm{B}$ denote the real Hilbert spaces of flow and effort variables defined on the boundary $\partial\omega$, respectively. Furthermore, consider the spaces of flow variables $\spa{F}$ and effort variables $\spa{E}$, where $\spa{F}=\spa{F}_\omega\times\spa{F}_\mathrm{B}$ and $\spa{E}=\spa{E}_\omega\times\spa{E}_\mathrm{B}$, respectively. Moreover, let $\spa{E}$ be the dual of $\spa{F}$. Finally, we endow the product space $\spa{B}=\spa{F}\times\spa{E}$ with the pairing:
\begin{align}\label{eq:Bilin}
\langle\langle[f_\omega,\vect{\Bflo},e_\omega,\vect{\Beff}],[\tilde{f}_\omega,\vect{\tBflo},\tilde{e}_\omega,\vect{\tBeff}]\rangle\rangle=\langle [f_\omega,\vect{\Bflo}],[\tilde{e}_\omega,\vect{\tBeff}]\rangle+\langle [e_\omega,\vect{\Beff}],[\tilde{f}_\omega,\vect{\tBflo}]\rangle.
\end{align}
Here, $\langle\cdot,\cdot\rangle$ defines a power-product on the product space $\spa{B}$.
Then, the subset $\spa{D}$ of $\spa{B}$ is a Stokes-Dirac structure with respect to the non-degenerate bilinear form \eqref{eq:Bilin}, if $\spa{D}=\spa{D}^\perp$, where $\spa{D}^\perp$ denotes the orthogonal complement of $\spa{D}$ and is defined as $\spa{D}^\perp\coloneqq\{[\tilde{f},\tilde{e}]\in\spa{F}\times\spa{E}|\langle\langle[\tilde{f},\tilde{e}],[f,e]\rangle\rangle=0\, \forall[f,e]\in\spa{D}\}.$ Here,\linebreak $[f,e]=[f_\omega,\vect{\Bflo},e_\omega,\vect{\Beff}]$ and analogously for $[\tilde{f},\tilde{e}]$.
\end{definition}
\noindent Considering flow through a pipe, we have that $\omega=(0,L)$ and $\partial\omega=\{0,L\}$. Furthermore, the flow and effort spaces $\spa{F}_\omega$ and $\spa{E}_\omega$ are given by $\spa{F}_\omega=\spa{F}_\upS\times\spa{F}_\mathrm{R}\times\spa{F}_\mathrm{P}$ and $\spa{E}_\omega=\spa{E}_\upS\times\spa{E}_\mathrm{R}\times\spa{E}_\mathrm{P}$ with $\spa{F}_\upS$, $\spa{E}_\upS$ being the flow and effort spaces of the energy-saving port, $\spa{F}_\mathrm{R}$, $\spa{E}_\mathrm{R}$ being the flow and effort spaces of the resistive port and $\spa{F}_\mathrm{P}$, $\spa{E}_\mathrm{P}$ being the flow and effort spaces of the external port. A visualization of the Stokes-Dirac structure and the different ports is given in Figure \ref{fig:Stokes-Dirac}. Considering System \ref{Sys:pHDCv}, the energy-dissipating port describes the state-dependent part of the cooling term, whereas the external port takes the ambient temperature of the pipe as an input. Furthermore, the energy-saving port is endowed with the structural operator $J(z)$ and the Hamiltonian $\Ham(z)$. The boundary port gives us the possibility to model energy flow over the boundary and to couple pipes in an energy-preserving manner, see Section \ref{Sec:pHNet}. The power-product needed in Definition \ref{def:StokesDirac} is then given by,
\begin{align}\label{eq:PowerProduct}
\langle[f_\omega,\vect{\Bflo}],[e_\omega,\vect{\Beff}]\rangle=\int_\omega \Sflo\cdot\Seff+\Rflo\cdot\Reff+\Pflo\cdot\Peff\,\D x +\vect{\Bflo}\cdot\vect{\Beff},
\end{align}
with $f_\omega=[\Sflo\quad\Rflo\quad\Pflo]^T\in\spa{F}_\omega$ and $e_\omega=[\Seff\quad\Reff\quad\Peff]^T\in\spa{E}_\omega$.
The dot $\cdot$ denotes the standard inner-product. Now, we can propose the flow and effort variables for the System \ref{Sys:pHDCv}.
\begin{close}
For the port-Hamiltonian formulation of System \ref{Sys:pHDCv} the boundary effort and flow variables can be stated in the following way
\begin{align*}
\vect{\Bflo}=[\BfloL\quad \BfloZ]^T, \quad  \vect{\Beff}=[\BeffL\quad \BeffZ]^T
\end{align*}
with
\begin{align}\label{eq:BoundaryPort}
\begin{bmatrix}\BfloL\\\BeffL\end{bmatrix}=\begin{bmatrix}0 & -1 &0\\ 1&0& \frac{e+p}{\rho}|_L\end{bmatrix}\left.\begin{bmatrix}
\frac{v^2}{2}\\\rho v\\1\end{bmatrix}\right|_L, \quad \begin{bmatrix}\BfloZ\\\BeffZ\end{bmatrix}=\begin{bmatrix}0 &1 &0\\ 1&0& \frac{e+p}{\rho}|_0\end{bmatrix}\left.\begin{bmatrix}
\frac{v^2}{2}\\\rho v\\1\end{bmatrix}\right|_0.
\end{align}
\end{close}
\noindent As the derivation of these boundary port variables is rather elaborate, we refer the interested reader to Appendix \ref{Appendix2}. When taking a closer look at the boundary variables given in \eqref{eq:BoundaryPort}, we notice that the boundary flows are the mass flows at both ends of the pipe. The boundary efforts are a sum of the storage efforts with respect to $\rho$ and $e$ evaluated at the boundary points of the pipe. This resembles the total specific enthalpy of the system at both pipe ends. This representation is advantageous for system coupling.
Now, we are able to set up the Stokes-Dirac structure underlying the pH System \ref{Sys:pHDCv}.
\begin{theorem}\label{Theo:DiracpHDC}
Let $\spa{F}=\spa{H}^1(\omega)^3\times\spa{L}^2(\omega)^3\times\spa{L}^2(\omega) \times\mathbb{R}^2$ and $\spa{E}=\spa{H}^1(\omega)^3\times\spa{L}^2(\omega)^3\times\spa{L}^2(\omega)\times\mathbb{R}^2$. Then the underlying Stokes-Dirac structure is given by the linear subset $\spa{D}\subset\spa{F}\times\spa{E}$, 
\begin{align}\label{eq:StokesDirac}
\begin{split}
\spa{D}=&\Bigl\{[[f_\omega,\vect{\Bflo}],[e_\omega,\vect{\Beff}]]\in\spa{F}\times\spa{E}|\,
[p\Seffc,\,\frac{e}{\rho}\Seffb,\,\frac{1}{\rho}\Seffb]\in\spa{H}^1(\omega)^3\\
&\quad\begin{bmatrix}\Sflo\\\Rflo\\\Pflo\end{bmatrix}+\begin{bmatrix}
J(z)&\so&B\\-\so&\sz&\vect{0}\\-B^T&\vect{0}&0
\end{bmatrix}\begin{bmatrix}\Seff\\\Reff\\\Peff\end{bmatrix}=\begin{bmatrix}
\vect{0}\\\vect{0}\\0\end{bmatrix},\\
&\quad\begin{bmatrix}\BfloL\\\BeffL\end{bmatrix}=\begin{bmatrix}0 & -1 &0\\ 1&0& \frac{e+p}{\rho}|_L\end{bmatrix}\SeffL, \quad  \begin{bmatrix}\BfloZ\\\BeffZ\end{bmatrix}=\begin{bmatrix}0 &1 &0\\ 1&0& \frac{e+p}{\rho}|_0\end{bmatrix}\SeffZ\Bigr\}.
\end{split}
\end{align}
Furthermore, the system of equations
\begin{align}\label{eq:ports}
\begin{split}
\Sflo&=-E(z)\partial_tz, \quad \Seff=\frac{\delta\Ham}{\delta z}(z),\quad \Reff=-R(z)\Rflo, \quad \Peff=T_\infty,\\&\quad [[f_\omega,\vect{\Bflo}],[e_\omega,\vect{\Beff}]]\in\spa{D},
\end{split}
\end{align}
is equivalent to the original System \ref{Sys:pHDCv}, and 
$\langle[f_\omega,\vect{\Bflo}],[e_\omega,\vect{\Beff}]\rangle=0$ represents the power-balance equation.
\end{theorem}
\noindent The proof is given in Appendix \ref{Appendix2}. In Theorem \ref{Theo:DiracpHDC} we choose $\spa{F}=\spa{E}$ instead of $\spa{E}=\spa{F}^*$, i.e., the dual of the space of flow variables. This is possible since $\spa{F}$ has an inner-product structure, such that the inner-product can be used as the duality product, see \cite{SchJ14}.
\subsubsection{Weak Formulation and Incorporation of Boundary Conditions}\label{Sec:WeakForm}
In this section we prepare our infinite dimensional systems for structure-preserving space discretization with the finite element method. For this, we introduce a variable transformation, incorporate boundary conditions in a structure-preserving manner and set up a weak formulation. Furthermore, we prove energy dissipation and mass conservation.
\subsubsection*{Change of State Variables}
We introduce another strong pH formulation for \ref{eq:DCv}, which is more suitable for structure-preserving approximation. The variable change from the velocity $v$ to the mass flow $m=\rho v$ leads to $m$ being state and effort at the same time. This simplifies the proof of the energy dissipation of a pH system and the choice of finite dimensional state and test spaces, see Section \ref{Sec:SpaceDis}. Choosing $\rho$, $m$ and $e$ as states for the non-isothermal Euler equations already in the beginning of this section leads to a different pH formulation. After the transformation the pH formulation of \ref{eq:DCv} takes the following form.
\begin{msystem}[$E_{pH}$]
Another port-Hamiltonian formulation of \ref{eq:DCv} with $z=[\rho\quad m\quad e]^T$ is given by
\begin{align*}
E(z)\partial_tz&=(J(z)-R(z))\eff(z)+Bu,\\
y&=B^T\eff(z),
\end{align*}
with effort function $\eff(z)=[\frac{m^2}{2\rho^2}\quad m\quad 1]^T$ and system operators
\usetagform{simple}
\begin{align}\label{Sys:pHDCm}\tag{$E_{pH}$}
\begin{split}
E(z)&=\begin{bmatrix} 1&0&0\\ -\frac{m}{\rho^2} & \frac{1}{\rho} &0\\0&0&1\end{bmatrix},\quad R(z)=\begin{bmatrix}0&0&0\\0&0&0\\0&0&\frac{k_\omega}{d}T\end{bmatrix},\quad B=\begin{bmatrix}0\\0\\\frac{k_\omega}{d}\end{bmatrix},\\
J(z)&=\begin{bmatrix}0&-D_x&0\\-D_x&0&-\frac{\lambda m|m|}{2d\rho^2}-\frac{e}{\rho}D_x-\frac{1}{\rho}D_xp\\0&\frac{\lambda m|m|}{2d\rho^2}-D_x\frac{e}{\rho}-pD_x\frac{1}{\rho}&0
\end{bmatrix}.
\end{split}
\end{align}
The input is given as $u=T_\infty$. The Hamiltonian is $\Ham(z)=\int_\omega\frac{m^2}{2\rho}+e\,\D x$.
\end{msystem}
\noindent To derive this system we start from System \ref{Sys:pHDCv} and make use of the variable transformation $v=\frac{m}{\rho}$, see \cite{LilSM20}. Plugging this in and resolving the non-linearity $\partial_t\frac{m}{\rho}=-\frac{m}{\rho^2}\partial_t\rho+\frac{1}{\rho}\partial_tm$ yields System \ref{Sys:pHDCm} with Hamiltonian $\Ham(z)~=~\int_\omega\frac{m^2}{2\rho}+e\,\D x$. As only the operator $E(z)$ is changed in comparison to \ref{Sys:pHDCv}, the operators $J(z)$ and $R(z)$ are still skew-adjoint as well as semi-elliptic and self-adjoint, respectively. We are left with checking condition \eqref{Cond:Volker}, i.e.,
\begin{align*}
E^T(z)\eff(z)=\begin{bmatrix}1 & -\frac{m}{\rho^2}&0\\0&\frac{1}{\rho}&0\\0&0&1\end{bmatrix}\begin{bmatrix}
\frac{m^2}{2\rho^2}\\m\\1\end{bmatrix}=\begin{bmatrix}-\frac{m^2}{2\rho^2}\\\frac{m}{\rho}\\1\end{bmatrix}=\begin{bmatrix}
\frac{\delta\Ham}{\delta \rho}(z)\\\frac{\delta\Ham}{\delta m}(z)\\\frac{\delta\Ham}{\delta e}(z)
\end{bmatrix},
\end{align*}
which is also fulfilled.
\begin{remark}
The variable transformation $v=\frac{m}{\rho}$ has no influence on the boundary port, as
\begin{align}\label{eq:BoundaryPortM}
\begin{bmatrix}\BfloL\\\BeffL\end{bmatrix}=\begin{bmatrix}0 & -1 &0\\ 1&0& \frac{e+p}{\rho}|_L\end{bmatrix}\left.\begin{bmatrix}
\frac{m^2}{2\rho^2}\\m\\1\end{bmatrix}\right|_L, \quad \begin{bmatrix}\BfloZ\\\BeffZ\end{bmatrix}=\begin{bmatrix}0 &1 &0\\ 1&0& \frac{e+p}{\rho}|_0\end{bmatrix}\left.\begin{bmatrix}
\frac{m^2}{2\rho^2}\\m\\1\end{bmatrix}\right|_0.
\end{align}
\end{remark}

\subsubsection*{Structure-Preserving Incorporation of Boundary Conditions}\label{sec:BoundaryConditions}
\noindent The handling of the boundary conditions is especially important for model reduction, see Section \ref{Sec:SPApprox}. Stating the boundary conditions as additional algebraic equations makes sure that they are still fulfilled after model reduction. To preserve the structure these conditions are expressed through the boundary port variables, see, e.g., \cite{LeGZM05} and \cite{Vil17_Dis}. We consider boundary conditions for the variables $\rho$, $m$ and $e$. Although there is some freedom in choosing the boundary conditions, we demand that one of the inflow conditions is stated in terms of the internal energy density, i.e., $e(t,0)=\eZ$. This is important for our network coupling in Section \ref{sec:Coupling}. The other two conditions can be chosen in terms of $\rho$ and $m$. Similar to \cite{LilSM20} we introduce boundary operators.
\begin{definition}
Let $\phi\in\spa{H}^1(\omega)$ and let $\phiZ$ and $\phiL$ denote the evaluations of $\phi$ at the inflow and the outflow of the pipe. Furthermore, let the boundary condition $\eZ$ be given. Then  we define
\begin{align*}
T_\upm:\spa{H}^1(\omega)\rightarrow\mathbb{R}^2,\quad T_\upm\phi=\begin{bmatrix}
-\phiL\\\phiZ\end{bmatrix}\quad \text{and} \quad
T_\upe:\spa{H}^1(\omega)\rightarrow\mathbb{R},\quad T_\upe\phi=\eZ\phiZ.
\end{align*}
\end{definition}
\noindent With these operators and the boundary port variables \eqref{eq:BoundaryPortM} we are able to state boundary conditions for $m$ and $e$ in dependence of the effort function.
\begin{close}\label{Prop:BoundCond}
Assuming $\BfloZ\neq 0$. The boundary conditions $e(t,0)=\eZ$, $m(t,0)=~m|_0$ and $m(t,L)=m|_L$ can be stated as
\begin{align*}
\vect{0}&=-T_\upm m+\vect{\Bflo}\quad \text{and}\quad 0=-T_e 1+\frac{\eZ}{\BfloZ}\BfloZ.
\end{align*}
\end{close}
\noindent To preserve the pH structure, we need to make the equations stating the boundary conditions dependent on the respective effort variables, i.e., $\eff_2=m$ and $\eff_3=1$.
First, we use the equations for the boundary flows $\vect{\Bflo}$ as they can be used to state boundary conditions for the mass flow $m$. Recalling \eqref{eq:BoundaryPortM}, we have
\begin{align*}
m|_L=-\BfloL\quad \text{and} \quad m|_0=\BfloZ\quad \Rightarrow \vect{0}=-T_\upm m+\vect{\Bflo}.
\end{align*}
Using the assumption that $\BfloZ\neq 0$, we can state the boundary condition for the internal energy as
\begin{align*}
\eZ=\frac{\eZ}{\BfloZ}\BfloZ \quad \text{and} \quad \eZ=\eZ 1=T_e 1\quad \Rightarrow 0=-T_e 1+\frac{\eZ}{\BfloZ}\BfloZ.
\end{align*}
\noindent The reformulation of the boundary condition for the internal energy density might seem like unnecessarily blowing up the system, but by doing this the boundary effort variables act as an output in the pH sense.
\subsubsection*{Weak Formulation}
\noindent To use Galerkin approximation in space, we set up a weak pH formulation for \ref{Sys:pHDCm}.
\begin{weak}\label{eq:VarPrinc}
Let $\BfloZ\neq 0$ and $\eZ\neq 0$. Find $\rho\in\spa{C}^1([0,t_f],\spa{L}^2(\omega))$, $m,\,e\in\spa{C}^1([0,t_f],\spa{H}^1(\omega))$, $\blam_\upm\in\spa{C}^0([0,t_f],\mathbb{R}^2)$, $\lambda_\upe\in\spa{C}^0([0,t_f],\mathbb{R})$, $\vect{\Bflo}\in\spa{C}^1([0,t_f],\mathbb{R}^2)$ fulfilling
\begin{subequations}
\begin{align}
(\partial_t\rho,\psi)&=-(\partial_xm,\psi),\nonumber\\
(\partial_t\frac{m}{\rho},\varphi)&=(\frac{m^2}{2\rho^2},\partial_x\varphi)-(\frac{e}{\rho}\partial_x1,\varphi)-(\frac{1}{\rho}\partial_x(p1),\varphi)-(\frac{\lambda}{2d}\frac{|m|m}{\rho^2}1,\varphi)+\blam_\upm\cdot T_\upm \varphi ,\label{eq:mBound}\\
(\partial_te,\phi)&=(\frac{e}{\rho}m,\partial_x\phi)+(\frac{m}{\rho},\partial_x(\phi p))+
(\frac{\lambda}{2d}\frac{|m|m^2}{\rho^2},\phi)-(\frac{k_\omega}{d}T1,\phi)\nonumber\\&\quad+(\frac{k_\omega}{d}T_\infty,\phi)+\lambda_e\cdot T_\upe \phi+\frac{e+p}{\rho}|_L\phiL\BfloL,\label{eq:eBound}\\
0&=-T_\upm m +\vect{\Bflo},\label{eq:lamMBound}\\
0&=-T_\upe 1 +\frac{\eZ}{\BfloZ}\BfloZ\label{eq:lamEBound},
\end{align}
\end{subequations}
for all $\psi\in \spa{L}^2(\omega)$ and $\varphi,\,\phi\in \spa{H}^1(\omega)$, $\omega=(0,L)$ and $t>0$.
\end{weak}
\noindent The weak formulation can be deduced as follows. Testing the equations in System \ref{Sys:pHDCm} with $\psi\in\spa{L}^2(\omega)$ and $\varphi,\,\phi\in\spa{H}^1(\omega)$, integrating over the pipe $\omega$ yields,
\begin{subequations}
\begin{align}
(\partial_t\rho,\psi)&=-(\partial_xm,\psi),\nonumber\\
\label{eq:MassPI}
(-\frac{m}{\rho^2}\partial_t\rho+\frac{1}{\rho}\partial_tm,\varphi)&=-(\partial_x\frac{m^2}{2\rho^2},\varphi)-(\frac{e}{\rho}\partial_x1,\varphi)-(\frac{1}{\rho}\partial_x(p1),\varphi)-(\frac{\lambda}{2d}\frac{|m|m}{\rho^2}1,\varphi),\\
\label{eq:EnergyPI}
(\partial_te,\phi)&=-(\partial_x(\frac{e}{\rho}m),\phi)-(p\partial_x(\frac{m}{\rho}),\phi )+
(\frac{\lambda}{2d}\frac{|m|m^2}{\rho^2},\phi)-(\frac{k_\omega}{d}T1,\phi)+(\frac{k_\omega}{d}T_\infty,\phi).
\end{align}
\end{subequations}
Using partial integration on the first term in \eqref{eq:MassPI} and on the first two terms in \eqref{eq:EnergyPI} leads to, i.e.,
\begin{align*}
(\partial_t\rho,\psi)&=-(\partial_xm,\psi),\nonumber\\
(-\frac{m}{\rho^2}\partial_t\rho+\frac{1}{\rho}\partial_tm,\varphi)&=(\frac{m^2}{2\rho^2},\partial_x\varphi)-(\frac{e}{\rho}\partial_x1,\varphi)-(\frac{1}{\rho}\partial_x(p1),\varphi)-(\frac{\lambda}{2d}\frac{|m|m}{\rho^2}1,\varphi)-[\frac{m^2}{2\rho^2}\varphi]|_0^L,\\
(\partial_te,\phi)&=(\frac{e}{\rho}m,\partial_x\phi)+(\frac{m}{\rho},\partial_x(\phi p))+
(\frac{\lambda}{2d}\frac{|m|m^2}{\rho^2},\phi)-(\frac{k_\omega}{d}T1,\phi)+(\frac{k_\omega}{d}T_\infty,\phi)\\&\quad-[m\frac{(e+p)}{\rho}\phi]|_0^L.
\end{align*}
Furthermore, we need to couple the equations for the boundary conditions to the weak formulation. Starting with \eqref{eq:mBound} we define the associated Lagrange multiplier as $\blam_\upm=\begin{bmatrix}\frac{m^2}{2\rho^2}|_L&\frac{m^2}{2\rho^2}|_0\end{bmatrix}^T$.
Thus, the boundary term in the equation for the mass flux \eqref{eq:mBound} can be reformulated to
\begin{align*}
-[\frac{m^2}{2\rho^2}\varphi]|_0^L=\frac{m^2}{2\rho^2}|_0\varphi|_0-\frac{m^2}{2\rho^2}|_L\varphi|_L=\blam_\upm\cdot T_\upm(\varphi).
\end{align*}
Finally, we also have to couple \eqref{eq:eBound} to our system via a Lagrange multiplier. We define $\lambda_\upe=~m|_0\frac{1+p/e}{\rho}|_0$ using the assumption that $\eZ\neq 0$. Then the boundary term in the energy balance \eqref{eq:eBound} takes the form,
\begin{align*}
-\left.[m\frac{(e+p)}{\rho}\phi]\right|_0^L=\lambda_\upe\cdot T_\upe(\phi)+\frac{e+p}{\rho}|_L\phiL\BfloL.
\end{align*}
\noindent Similar to \cite{LilSM20} weak boundary conditions can be stated using the Lagrange multipliers $\blam_\upm$.
\begin{theorem}\label{Theo:MassEn}
Let $z=[\rho\quad m\quad e]^T$ be a solution of the Weak Formulation \ref{eq:VarPrinc}. Then the following energy dissipation and mass conservation properties hold,
\begin{align*}
\frac{\D}{\D t}\Ham(z)&\leq(\Pflo,\Peff) +\vect{\Bflo}^T \vect{\Beff}, \quad \frac{\D}{\D t}\int_\omega\rho\,\D x=\left.-[m]\right|_0^L.
\end{align*}
\end{theorem}
\begin{proof}
Mass conservation easily follows from testing the Weak Formulation \ref{eq:VarPrinc} with $[1\quad 0\quad 0]^T$, i.e.,
\begin{align*}
\frac{\D}{\D t}\int_\omega\rho\D x=(\partial_t\rho,1)=-(\partial_xm,1)=(m,\partial_x1)-\left.[m]\right|_0^L=\left.-[m]\right|_0^L.
\end{align*}
To show energy dissipation, we need to test the Weak Formulation \ref{eq:VarPrinc} with the effort function $\eff$, as
\begin{align*}
\frac{\D}{\D t}\Ham(z)=(\partial_tz,\frac{\delta \Ham}{\delta z})=(\partial_tz,E^T(z)\eff(z))=(E(z)\partial_tz,\eff(z)).
\end{align*}
The second equality follows from condition \eqref{Cond:Volker}. To finish the proof using the Weak Formulation \ref{eq:VarPrinc}, we need that $\frac{m^2}{2\rho^2}\in \spa{L}^2(\omega)$ and $m,\,1\in \spa{H}^1(\omega)$ for each $t\geq 0$. As the latter is fulfilled, we only need to take care of the first. Analogously to \cite{LilSM20}, we use the Riesz Representation Theorem \cite{Dob10}. As $\spa{L}^2(\omega)=\partial_x\spa{H}^1(\omega)$, there exists a $\varepsilon(\rho,m)\in \spa{L}^2(\omega)$, such that $(\frac{m^2}{2\rho^2},\psi)=(\varepsilon(\rho,m),\psi)$ for all $\psi\in \spa{L}^2(\omega)$.
\begin{align*}
\frac{\D}{\D t}\Ham(z)&=(E(z)\partial_tz,\eff)=(\partial_t\rho,\frac{m^2}{2\rho^2})+(-\frac{m}{\rho^2}\partial_t\rho+\frac{1}{\rho}\partial_tm,m)+(\partial_te,1)\\
&{\overset{(1)}{=}}(\partial_t\rho,\varepsilon(\rho,m))+(-\frac{m}{\rho^2}\partial_t\rho+\frac{1}{\rho}\partial_tm,m)+(\partial_te,1)\\
&{\overset{(2)}{= }}-(\partial_xm,\varepsilon(\rho,m))+(\frac{m^2}{2\rho^2},\partial_xm)-(\frac{e}{\rho}\partial_x1,m)-(\frac{1}{\rho}\partial_x(p1),m)-(\frac{\lambda}{2d}\frac{|m|m}{\rho^2}1,m)\\&\quad +\blam_\upm\cdot T_\upm m+(\frac{e}{\rho}m,\partial_x1)+(\frac{1}{\rho}m,\partial_x(1 p))+(\frac{\lambda}{2d}\frac{|m|m^2}{\rho^2},1)\\&\quad -(\frac{k_\omega}{d}T1,1)+(\frac{k_\omega}{d}T_\infty,1)+\lambda_\upe\cdot T_\upe 1+\frac{e+p}{\rho}|_L\BfloL1\\
& {\overset{(3)}{=}}-(\partial_xm,\varepsilon(\rho,m))+(\varepsilon(\rho,m),\partial_xm)+\blam_\upm\cdot\vect{\Bflo}\\&\quad -(\frac{k_\omega}{d}T1,1)+(\frac{k_\omega}{d}T_\infty,1)+\lambda_\upe\frac{\eZ}{\BfloZ}\BfloZ+\frac{e+p}{\rho}|_L\BfloL 1\\
&{\overset{(4)}{\leq}}(\frac{k_\omega}{d}T_\infty,1) +\vect{\Bflo}^T(\blam_\upm+[\frac{e+p}{\rho}|_L\quad\lambda_\upe\eZ]^T)\\
&{\overset{(5)}{=}}(\Pflo,\Peff)+\vect{\Bflo}^T\vect{\Beff},
\end{align*}
In the equalities (1) and (3) we use the existence of $\varepsilon(\rho,m)$ by the Riesz Representation Theorem. Furthermore, in (2) we inserted the system equations and in (3) the algebraic equations \eqref{eq:lamMBound} and \eqref{eq:lamEBound}. For the inequality (4) we utilized the positive definiteness of $(\frac{k_\omega}{d}T1,1)$, which is given as $k_\omega$, $T$ and $d$ are bigger zero. In the last equality (5) we used the definitions of the environment and boundary ports from Theorem \ref{Theo:DiracpHDC} and the definitions of the Lagrange multipliers from the derivation of Weak Formulation \ref{eq:VarPrinc}.
\end{proof}
\noindent From this it follows that mass conservation and energy dissipation are determined by the values of the state at the inflow and outflow of the pipe, i.e., the boundary port variables. Mass conservation is given, when the mass flow is equal at both pipe ends. Energy dissipation, furthermore, depends on the ambient temperature $T_\infty$, which acts as an input from the environment through the external port.

\subsection{Port-Hamiltonian Formulation for Pipe Networks}\label{Sec:NetPH}
As gas networks usually do not consist of only one pipe, this section deals with coupling the single pipe systems to a network. As we couple pH systems the coupling conditions not only should lead to a well-posed problem, but also preserve the structure. Thus, the next sections are concerned with incorporating energy-preserving coupling conditions in a structure-preserving way into the network system. 
\subsubsection{Network Toplogy and Function Spaces}\label{sec:NetTop}
In this section we follow closely the notation and wording of \cite{EggK18} and \cite{LilSM20}. We consider connected, finite and directed graphs $\mathcal{G}=(\mathcal{N},\mathcal{E},L)$ consisting of a set of nodes $\mathcal{N}=\{\nu_1,\dots,\nu_l\}$, $l\in\mathbb{N}$, and a set of edges $\mathcal{E}=\{\omega_1,\dots,\omega_k\}\subset\mathcal{N}\times\mathcal{N}$, $k\in\mathbb{N}$. Every edge $\omega$ has a length $L^\omega$, $0<L^\omega<\infty$, which we collect in the set $L=\{L^{\omega_1},\dots,L^{\omega_k}\}$. Furthermore, the structure of $\mathcal{G}$ is given through an incindence mapping, which describes the direction of the edges between nodes.
\begin{definition} \label{Def:IncMap}
The incidence mapping of node $\nu\in\mathcal{N}$ with respect to the edge $\omega\in\mathcal{E}$ is defined by
\begin{align*}
n^\omega[\nu]=\begin{cases}1 &\text{ for } \omega=(\nu,\bar{\nu}) \text{ for some } \bar{\nu}\in\mathcal{N},\\
						-1 &\text{ for } \omega=(\bar{\nu},\nu) \text{ for some } \bar{\nu}\in\mathcal{N},\\
						0 &\text{ else.} \end{cases}
\end{align*}
\end{definition}
\noindent Furthermore, we need the set of edges adjacent to node $\nu$, i.e., $\mathcal{E}(\nu)=\{\omega\in\mathcal{E}:\omega=(\nu,\bar{\nu})\text{ or } \omega=(\bar{\nu},\nu)\}$, the set of interior or coupling nodes $\mathcal{N}_0\subset\mathcal{N}$, and the set of boundary nodes $\mathcal{N}_\partial=\mathcal{N}\backslash\mathcal{N}_0$. An example graph is given in Figure \ref{fig:ExGraph}. Lastly, we make two assumptions on the network topology. 
\begin{assum}
\label{as:Graph}
Let the following two assumptions hold.
\begin{itemize}
\item The graph $\mathcal{G}=(\mathcal{N},\mathcal{E})$ is connected, finite and directed.
\item It holds that $|\mathcal{E}(\nu)|=1$ for $\nu\in\mathcal{N}_\partial$.
\end{itemize}
\end{assum}
\begin{figure}[t]
\centering
\begin{tikzpicture}
\draw[->,line width=1pt] (0,-5) -- (3,-5); 
\draw[->,line width=1pt] (3,-5) -- (5,-3.5); 
\draw[->,line width=1pt] (3,-5) -- (5,-6.5);
\fill (0,-5) circle (0.1);
\fill (3.05,-5) circle (0.1);
\fill (5.05,-3.45) circle (0.1);
\fill (5.05,-6.55) circle (0.1);
\node at (0,-5.5) {$\nu_1$};
\node at (3.05,-5.5) {$\nu_2$};
\node at (5.05,-2.95) {$\nu_3$};
\node at (5.05,-7.05) {$\nu_4$};
\node at (1.5,-5.5) {$\omega_1$};
\node at (4.5,-4.5) {$\omega_2$};
\node at (4.5,-5.5) {$\omega_3$};
\end{tikzpicture}
\caption{We have $\mathcal{N}=\{\nu_1,\nu_2,\nu_3,\nu_4\}$ and $\mathcal{E}=\{\omega_1,\omega_2,\omega_3\}$, respectively. The incidence mapping yields $n^{\omega_1}[\nu_1]=1$, $n^{\omega_1}[\nu_2]=-1$, $n^{\omega_2}[\nu_2]=1$,$n^{\omega_2}[\nu_3]=-1$, $n^{\omega_3}[\nu_2]=1$ and $n^{\omega_3}[\nu_4]=-1$. The set of boundary nodes is $\mathcal{N}_\partial=\{\nu_1,\nu_3,\nu_4\}$, such that $\mathcal{N}_0=\{\nu_2\}$.}
\label{fig:ExGraph}
\end{figure}
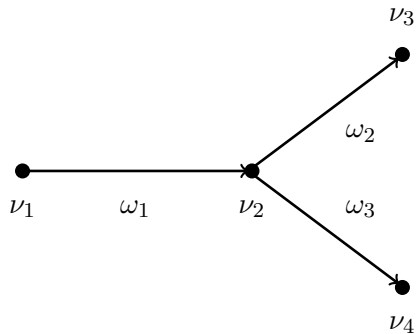
\noindent Let $\omega\in\mathcal{E}$ be identified with the interval $(0,L^\omega)$, then the spatial domain on the network is given by $\Omega=\{x:x\in\omega, \text{ for } \omega\in\mathcal{E}\}$. We define the space of square-integrable functions on the network as\linebreak $\spa{L}^2(\mathcal{E})=\{\psi:\Omega\to\mathbb{R}\text{ with } \psi|_\omega\in \spa{L}^2(\omega)\text{ for all }\omega\in\mathcal{E}\}$. For $\psi,\,\phi\in \spa{L}^2(\mathcal{E})$ the inner-product and the associated norm are given by
$(\psi,\phi)_\mathcal{E}=\sum_{\omega\in\mathcal{E}}(\psi|_\omega,\phi|_\omega)$ and $||\psi||_\mathcal{E}=\sqrt{(\psi,\psi)_\mathcal{E}},
$ respectively. Moreover, we introduce the broken operator $\partial_x'$, describing the edgewise weak derivative, i.e., $(\partial_x'\psi)|_\omega=\partial_x\psi|_\omega \text{ for all } \omega\in\mathcal{E}.$
For convenience, we rename $\partial_x'$ to $\partial_x$. From the context, it will be clear which partial derivative is used. Thus, the space of square-integrable functions with square-integrable broken derivative on $\mathcal{E}$ is given by
$\spa{H}^1_{pw}(\mathcal{E})=\{\psi\in \spa{L}^2(\mathcal{E}):\partial_x\psi\in~ \spa{L}^2(\mathcal{E})\}$. The function space of piecewise smooth functions is given by
$\spa{C}^k_{pw}(\mathcal{E})=\{\psi:\Omega\rightarrow\mathbb{R}\text{ with } \psi|_\omega\in \spa{C}^k(\omega)\text{ for all } \omega\in\mathcal{E}\}$ for $k\geq 0$. The spatial domains for the boundary or coupling conditions are $\mathcal{N}_\partial$ and $\mathcal{N}_0$, respectiveky. The functions spaces acting on these sets are $\mathbb{R}^{|\mathcal{N}_\partial|}$ and $\mathbb{R}^{|\mathcal{N}_0|}$, which are Euclidean vector spaces equipped with the standard scalar product and norm.
\subsubsection{Coupling Conditions and the Port-Hamiltonian Framework}\label{sec:Coupling}
We consider the following assumption.
\begin{assum}\label{as:NoFlowChange}
There will never be a change in flow direction. 
\end{assum}
\noindent This assumption is often used when working with the compressible non-isothermal Euler equations \ref{eq:DCv}, see, e.g., \cite{Her08} and \cite{LanM18}. In contrast to the coupling of the isothermal Euler-like equations, which has been thoroughly studied in literature, there are only few sources dealing with the coupling of \ref{eq:DCv}. The paper \cite{LanM18} gives a good overview on the coupling conditions at an interior node $\nu\in\mathcal{N}_0$, which are used in the literature, and sets up a well-defined set of coupling conditions for \ref{eq:DCv}, i.e.,
\begin{itemize}[leftmargin=2cm]
\item[$(\upM)$] conservation of mass: $\sum_{\omega\in\mathcal{E}(\nu)}n^{\omega}[\nu]\mnu=0$, 
\item[$(\upH)$] equality of specific total enthalpy: $\hnu=h^*$, $\omega\in\mathcal{E}(\nu)$,
\item [($\upS_{out}$)] equality of outgoing entropy: $\snu=s^*$, $\omega\in\mathcal{I}_{out}(\nu)$ with the entropy mix:\\ $s^*=\frac{1}{\sum_{\omega\in\mathcal{I}_{in}(\nu)}n^{\omega}[\nu]\mnu}\sum_{\omega\in\mathcal{I}_{in}(\nu)}n^{\omega}[\nu]\mnu\snu$.
\end{itemize}
The superscript $\omega$ denotes the respective quantity on pipe $\omega$. For the coupling condition $(\upS_{out})$ we additionally need the following sets. Let $\nu\in\mathcal{N}_0$ be an inner node, we then have
\begin{align*}
\delta_{in}(\nu)&\coloneqq\{\omega\in\mathcal{E}:\exists \bar{\nu}\in\mathcal{N} \text{ with } \omega=(\bar{\nu},\nu)\},\\
\delta_{out}(\nu)&\coloneqq\{\omega\in\mathcal{E}:\exists \bar{\nu}\in\mathcal{N} \text{ with } \omega=(\nu,\bar{\nu})\},
\end{align*}
i.e., the graph-topological directions of the pipes with respect to a certain node. From these sets and the flow direction, we can define the sets $\mathcal{I}_{in}(\nu)$ and $\mathcal{I}_{out}(\nu)$ through,
\begin{align*}
\mathcal{I}_{in}(\nu)&\coloneqq\{\omega\in\delta_{in}(\nu): m^\omega\geq 0\}\cup\{\omega\in\delta_{out}(\nu):m^\omega\leq 0\},\\
\mathcal{I}_{out}(\nu)&\coloneqq\{\omega\in\delta_{in}(\nu): m^\omega< 0\}\cup\{\omega\in\delta_{out}(\nu):m^\omega> 0\},
\end{align*}
which tell us the pipes $\omega$ on which the flow travels into or out of the node $\nu$, respectively.
Due to Assumption \ref{as:NoFlowChange}, we are able to fix these sets before the simulation, i.e., $\mathcal{I}_{in}(\nu)=\delta_{in}(\nu)$  and $\mathcal{I}_{out}(\nu)=\delta_{out}(\nu)$. Otherwise, these sets would change during the simulation, which evidently has an impact on the computations and the implementation of the resulting algorithms. This is a question of ongoing research.\\
\begin{theorem}\label{Theo:Coupling} \hfill
\begin{enumerate}
\item For the Euler equations \ref{eq:DCv} we have that $(\upM)$ and $(\upH)$ imply energy conservation, i.e., $$\sum_{\omega\in\mathcal{E}(\nu)}n^{\omega}[\nu](v^{\omega}(E^{\omega}+p^{\omega}))|_\nu=0.$$
\item From $(\upM)$ and $(\upS_{out})$ it follows that the entropy per unit volume in smooth flows is conserved, i.e.,
\begin{align*}
(\upS')\quad \sum_{\omega\in\mathcal{E}(\nu)}n^{\omega}[\nu]\mnu\snu=0,\quad t>0.
\end{align*}
\end{enumerate}
\end{theorem}
\begin{proof}
To prove $(1)$ we start by multiplying the mass conservation at node $\nu$ with the coupling enthalpy related coupling constant $h^*$, which yields
\begin{align}\label{eq:PropMass}
0=\sum_{\omega\in\mathcal{E}(\nu)}h^*n^{\omega}[\nu]\mnu.
\end{align}
As we are working with the Euler equations we can choose the total specific enthalpy at some pipe $\omega\in\mathcal{E}(\nu)$ as $h^*$, i.e., $h^*=\frac{E^\omega+p^\omega}{\rho^\omega}|_\nu$.
Plugging this into \eqref{eq:PropMass} and using $m=\rho v$ yields
\begin{align*}
0=\sum_{\omega\in\mathcal{E}(\nu)}n^{\omega}[\nu]\frac{E^\omega+p^\omega}{\rho^\omega}|_\nu\mnu=\sum_{\omega\in\mathcal{E}(\nu)}n^{\omega}[\nu](v^{\omega}(E^{\omega}+p^{\omega}))|_\nu,
\end{align*}
which proves energy conservation at node $\nu$.
To prove proposition $(2)$ we utilize that $\mathcal{I}_{in}(\nu)\cap\mathcal{I}_{out}(\nu)=\emptyset$ and $\mathcal{I}_{in}(\nu)\cup\mathcal{I}_{out}(\nu)=\Net(\nu)$. We reformulate the entropy mix in $(\upS_{out})$ as
\begin{align}\label{eq:PropIn}
s^*\sum_{\omega\in\mathcal{I}_{in}(\nu)}n^{\omega}[\nu]\mnu=\sum_{\omega\in\mathcal{I}_{in}(\nu)}n^{\omega}[\nu]\mnu\snu, 
\end{align}
for $\omega\in\mathcal{E}(\nu)$. To be able to use $(\upM)$, i.e., $\sum_{\omega\in\mathcal{E}(\nu)}n^{\omega}[\nu]\mnu=0$, we need to take the outgoing pipes in $\mathcal{E}(\nu)$ into account. Here, we have that 
\begin{align}\label{eq:PropOut}
s^*\sum_{\omega\in\mathcal{I}_{out}(\nu)}n^{\omega}[\nu]\mnu=\sum_{\omega\in\mathcal{I}_{out}(\nu)}s^*n^{\omega}[\nu]\mnu=\sum_{\omega\in\mathcal{I}_{out}(\nu)}n^{\omega}[\nu]\mnu\snu, 
\end{align}
since on each outgoing pipe $\snu=s^*$, $\omega\in\mathcal{I}_{out}(\nu)$ and $\omega\in\mathcal{E}(\nu)$. Adding \eqref{eq:PropIn} and \eqref{eq:PropOut} yields,
\begin{align*}
s^*\sum_{\omega\in\mathcal{E}(\nu)}n^{\omega}[\nu]\mnu=\sum_{\omega\in\mathcal{E}(\nu)}n^{\omega}[\nu]\mnu\snu,
\end{align*}
and finally using the mass conservation proves $(\upS')$.
\end{proof}
\noindent Having found coupling conditions, which couple the single pipe systems in an energy-preserving manner and close the system, they need to be incorporated into the pH structure. From the Weak Formulation \ref{eq:VarPrinc} and the boundary port variables \eqref{eq:BoundaryPort}, we see that $(\upM)$ is realizable via the boundary flows,  
\begin{align*}
\sum_{\omega\in\mathcal{E}(\nu)}\Bflonu=0,\quad \forall \nu\in\mathcal{N}_0,
\end{align*}
as $\Bflonu=n^{\omega}[\nu]\mnu$.
The enthalpy equality $(\upH)$ is realized via the boundary efforts,
\begin{align*}
\hnu=\Beffnu=h^*, \quad \omega\in\mathcal{E}(\nu),\quad \nu\in\mathcal{N}_0.
\end{align*}
This shows that $(\upM)$ and $(\upH)$ can be integrated into the pH formulation rather naturally. As seen in Theorem \ref{Theo:Coupling} these two conditions establish energy conservation at the coupling nodes, which makes the resulting coupled system pH. Up until now we have not considered $(\upS_{out})$, which is obviously not needed for energy conservation in the coupled system and thus, incorporating it into the system does not seem quite so obvious. Taking a closer look at the Weak Formulation \ref{eq:VarPrinc} reveals, that we still have $\enu$, i.e., the value of the internal energy at the inflow node of the outgoing pipe $\omega$, as a degree of freedom. Because of Assumption \ref{as:IdealGas}, we know that $s^{\omega}=c_v\operatorname{ln}(\frac{p^{\omega}}{(\rho^{\omega})^\gamma})$ and $p^\omega=\frac{R}{c_v}e^\omega$. We assume that $c_v$ and $c_p$ do not change over space, i.e., we only have one kind of gas in our network. We can then derive a condition for $\enu$,
\begin{align*}
\snu=s^*\quad &\Rightarrow c_v\operatorname{ln}(\frac{p^{\omega}}{(\rho^{\omega})^\gamma})|_\nu=s^*\quad \Rightarrow \enu=\frac{c_v}{R}(\rho^{\omega}|_\nu)^\gamma \mathrm{exp}(\frac{s^*}{c_v}).
\end{align*}
We can also compute $s^*$ from known quantities, as it depends only on the incoming pipes. This leads to an extensive non-linear formula for $\enu$, which can be inserted in our system, i.e., into
\begin{align*}
0&=-(\vect{t}^\omega)^T_\upe\vect{1}^{\omega}+\frac{\enu}{\Bflonu}\Bflonu,\quad \omega=(\nu,\bar{\nu}), \quad \nu\in\mathcal{N}_0, \, \bar{\nu}\in\mathcal{N}.
\end{align*}
Thus, we do not need to add a further equation to the system and preserve the pH structure.
\subsubsection{Port-Hamiltonian Network Model}\label{Sec:pHNet}
Having chosen our coupling conditions and having introduced the input-output-coupling in {Section \ref{sec:Coupling} we can write down the weak formulation for the pipe network analogously to Weak Formulation \ref{eq:VarPrinc} using the function spaces introduced in Section \ref{sec:NetTop}. 
\begin{weak}\label{Sys:wpHDCNet}
Let $\Bflonu\neq 0$, $\enu\neq 0$ for all $\omega\in\mathcal{E}$.
Find $\rho\in\spa{C}^1([0,t_f],\spa{L}^2(\mathcal{E}))$, $m,\,e\in\spa{C}^1([0,t_f],\spa{H}^1_{pw}(\mathcal{E}))$, $\blam_\upm^\omega\in\spa{C}^0([0,t_f],\mathbb{R}^2)$, $\lambda_\upe^\omega\in\spa{C}^0([0,t_f],\mathbb{R})$, $\vect{\Bflo}^\omega\in\spa{C}^1([0,t_f],\mathbb{R}^2$) and $ \lambda_\upH^\nu\in\spa{C}^0([0,t_f],\mathbb{R})$ for all $\omega\in\mathcal{E}$, $\nu\in\mathcal{N}_0$ fulfilling for each $\omega=(\nu,\bar{\nu})\in\mathcal{E}$
\begin{align*}
(\partial_t\rho,\psi)_\mathcal{E}&=-(\partial_xm,\psi)_\mathcal{E},\nonumber\\
(\partial_t\frac{m}{\rho},\varphi)_\mathcal{E}&=(\frac{m^2}{2\rho^2},\partial_x\varphi)_\mathcal{E}-(\frac{e}{\rho}\partial_x1,\varphi)_\mathcal{E} -(\frac{1}{\rho}\partial_x(p1),\varphi)_\mathcal{E}-(\frac{\lambda}{2d}\frac{|m|m}{\rho^2}1,\varphi)_\mathcal{E}\nonumber\\&\quad+\sum_{\omega\in\mathcal{E}}\blam^\omega_\upm\cdot T^\omega_\upm(\varphi^\omega),\nonumber\\
(\partial_te,\phi)_\mathcal{E}&=(\frac{e}{\rho}m,\partial_x\phi)_\mathcal{E}+(\frac{m}{\rho},\partial_x(\phi p))_\mathcal{E}+
(\frac{\lambda}{2d}\frac{|m|m^2}{\rho^2},\phi)_\mathcal{E}-(\frac{k_\omega}{d}T1,\phi)_\mathcal{E}+(\frac{k_\omega}{d}T_\infty,\phi)_\mathcal{E}\nonumber\\&\quad+\sum_{\omega\in\mathcal{E}}\lambda_e^\omega\cdot T_\upe^\omega(\phi^\omega)+\sum_{\omega\in\mathcal{E}}\frac{e^\omega+p^\omega}{\rho^\omega}\phi^\omega\Bflonub,\\
0&=-T_\upm^\omega(m^\omega)+\vect{\Bflo}^\omega, &\text{for all } \omega\in\mathcal{E},\\
0&=-T_\upe^\omega(1)+\frac{\enu}{\Bflonu}\Bflonu, &\text{for all } \omega\in\mathcal{E},
\end{align*}
for $\psi\in\spa{L}^2(\mathcal{E})$, $\varphi,\,\phi\in\spa{H}^1_{pw}(\mathcal{E})$ and $t\in[0,t_f]$
and the coupling conditions
\begin{subequations}
\begin{align}
\label{Coup:Mass}
0&=\sum_{\omega\in\mathcal{E}(\nu)}\Bflonu &\text{ for all } \nu\in\mathcal{N}_0,\\
\label{Coup:Enthalpy}
\Beffnu&=\lambda_\upH^\nu &\text{ for all } \nu\in\mathcal{N}_0, \, \omega\in\mathcal{E}(\nu),\\
\enu&=\frac{c_v}{R}(\rho^{\omega}|_\nu)^\gamma \mathrm{exp}(\frac{s^*}{c_v}) &\text{for all } \nu\in\mathcal{N}_0,\, \omega\in\mathcal{E}(\nu)\nonumber.
\end{align}
\end{subequations}
\end{weak}
\noindent The system has to be complemented with two boundary conditions per inflow boundary node and one boundary condition per outflow boundary node. 
\begin{theorem}
For any solution $z=[\rho\quad m \quad e]^T$ of Weak Formulation \ref{Sys:wpHDCNet} the energy dissipation inequality,
\begin{align*}
\frac{\D}{\D t}\Ham_\mathcal{E}(z)\leq (\Pflo,\Peff)_\mathcal{E}+\sum_{\nu\in\mathcal{N}_\partial,\omega\in\mathcal{E}(\nu)}\Bflonu\cdot\Beffnu,
\end{align*}
holds for $t\in[0,t_f]$ with the network Hamiltonian $\Ham_\mathcal{E}(z)=\sum_{\omega\in\mathcal{E}}\Ham_\omega(z^\omega)=\sum_{\omega\in\mathcal{E}}\int_\omega \frac{(m^\omega)^2}{2\rho^{\omega}}+e^\omega\,\D x$.
\end{theorem}
\begin{proof}
Since the total energy of the network is the sum of the total energies of each pipe the energy dissipation can be proved for each pipe separatley as in Theorem \ref{Theo:MassEn}. The definition of the inner-product for the network given in Section \ref{sec:NetTop} then yields the last equality. Thus, we have that,
\begin{align}\label{eq:b}
\frac{\D}{\D t}\Ham_\mathcal{E}(z)&=\frac{\D}{\D t}\sum_{\omega\in\mathcal{E}}\Ham_\omega(z^\omega)=\sum_{\omega\in\mathcal{E}}\frac{\D}{\D t}\Ham_\omega(z^\omega)
=\sum_{\omega\in\mathcal{E}}(\frac{\delta}{\delta z^\omega}\Ham_\omega(z^\omega),\partial_tz^\omega)\nonumber\\&\leq \sum_{\omega\in\mathcal{E}}(\Pflo^\omega,\Peff^\omega)+\vect{\Bflo}^\omega\cdot\vect{\Beff}^\omega=(\Pflo,\Peff)_\mathcal{E}+\sum_{\omega\in\mathcal{E}}\vect{\Bflo}^\omega\cdot\vect{\Beff}^\omega.
\end{align}
Finally, we make use of the coupling conditions \eqref{Coup:Mass} and \eqref{Coup:Enthalpy} and rewrite the second summand in the last equality of \eqref{eq:b}, i.e.,
\begin{align*}
\sum_{\omega\in\mathcal{E}}\vect{\Bflo}^\omega\cdot\vect{\Beff}^\omega&=\sum_{\nu\in\mathcal{N},\omega\in\mathcal{E}(\nu)}\Bflonu\Beffnu=\sum_{\nu\in\mathcal{N}_\partial,\omega\in\mathcal{E}(\nu)}\Bflonu\Beffnu+\sum_{\bar{\nu}\in\mathcal{N}_0,\omega\in\mathcal{E}(\bar{\nu})}\Bflonub\Beffnub\\
&=\sum_{\nu\in\mathcal{N}_\partial,\omega\in\mathcal{E}(\nu)}\Bflonu\Beffnu+\sum_{\bar{\nu}\in\mathcal{N}_0,\omega\in\mathcal{E}(\bar{\nu})}\Bflonub\lambda_\upH^{\bar{\nu}}\\
&=\sum_{\nu\in\mathcal{N}_\partial,\omega\in\mathcal{E}(\nu)}\Bflonu\Beffnu+\sum_{\bar{\nu}\in\mathcal{N}_0}\lambda_\upH^{\bar{\nu}}\sum_{\omega\in\mathcal{E}(\bar{\nu})}\Bflonub\\
&=\sum_{\nu\in\mathcal{N}_\partial,\omega\in\mathcal{E}(\nu)}\Bflonu\Beffnu,
\end{align*}
which proves the theorem.
\end{proof}
\noindent Furthermore, we also have global mass conservation.
\begin{theorem}
For any solution $z=[\rho\quad m \quad e]^T$ of Weak Formulation \ref{Sys:wpHDCNet} it holds for $t\in[0,t_f]$, that
\begin{align*}
\frac{\D}{\D t}\int_\Omega\rho\,\D x=\sum_{\nu\in\mathcal{N}_\partial,\omega\in\mathcal{E}(\nu)}\Bflonu.
\end{align*}
\end{theorem}
\begin{proof}
The conservation of total mass on the network is proved by using conservation of the total mass on each pipe, as shown in the proof of Theorem \ref{Theo:MassEn}, as
\begin{align*}
\frac{\D}{\D t}\int_\Omega\rho\,\D x=\sum_{\omega\in\mathcal{E}}\int_\omega \rho^\omega\,\D x=\sum_{\omega\in\Net}m^\omega|_0-m^\omega|_L=\sum_{\omega\in\mathcal{E}}\vect{\Bflo}^\omega\cdot[1\quad 1]^T.
\end{align*}
Rearranging this sum and using the coupling condition \eqref{Coup:Mass} yields,
\begin{align*}
\sum_{\omega\in\mathcal{E}}\vect{\Bflo}^\omega\cdot[1\quad 1]^T=\sum_{\nu\in\mathcal{N}_\partial,\omega\in\mathcal{E}(\nu)}\Bflonu+\sum_{\bar{\nu}\in\mathcal{N}_0,\omega\in\mathcal{E}(\bar{\nu})}\Bflonub=\sum_{\nu\in\mathcal{N}_\partial,\omega\in\mathcal{E}(\nu)}\Bflonu.
\end{align*}
\end{proof}
\begin{remark}\label{rk:Cross}
Pipelines in gas transport networks are usually modeled with cross-sec\-tionally averaged dynamics. Thus, we include the cross-sectional pipe area $A^\omega$ into our implementation. This affects all integral-related quantities, like the $\spa{L}^2(\mathcal{E})$ inner-product on the network, i.e.,
\begin{align*}
(\psi,\bar{\psi})_{\mathcal{E}}=\sum_{\omega\in\mathcal{E}}A^\omega(\psi|_\omega,\bar{\psi}|_\omega),
\end{align*}
and $\Ham_\mathcal{E}(z)=\sum_{\omega\in\mathcal{E}}A^\omega\Ham_\omega(z^\omega)$. Furthermore, the incidence mapping in Definition \ref{Def:IncMap} is changed to
\begin{align*}
n^\omega[\nu]=\begin{cases}A^\omega &\text{ for } \omega=(\nu,\bar{\nu}) \text{ for some } \bar{\nu}\in\mathcal{N},\\
						-A^\omega &\text{ for } \omega=(\bar{\nu},\nu) \text{ for some } \bar{\nu}\in\mathcal{N},\\
						0 &\text{ else.} \end{cases}
\end{align*}
Thus, also the boundary conditions and operators and coupling conditions are altered.
\end{remark}

\section{Structure-Preserving Approximations}\label{Sec:SPApprox}
\noindent The following approximations are carried out for the single pipe system. In Section \ref{Sec:Mod} we mention the modifications which are needed, when these methods are applied to the network system.
\subsection{Spatial Discretization}\label{Sec:SpaceDis}
\noindent In the context of pH systems structure-preservation is very important in order to keep energy dissipation and mass conservation during approximation. Thus, Galerkin approximation schemes with compatible finite dimensional approximation spaces need to be utilized. These schemes help to realize not only classical finite element discretizations, but also model order reduction, where compatible global bases are needed. In the case of the compressible non-isothermal Euler equations we can extend the ansatz introduced in \cite{LilSM20}. Thus, we introduce the following assumption.
\begin{assum}\label{As:Comp}
Let $\spa{V}=\spa{V}_\uprho\times\spa{V}_\upm\times\spa{V}_\upe$ and $\spa{V}_\uprho\subset \spa{L}^2(\omega)$, $\spa{V}_\upm,\spa{V}_\upe\subset \spa{H}^1(\omega)$ be finite dimensional subspaces which fulfill the following assumptions:
\begin{itemize}
\item[A1)] $\spa{V}_\uprho=\partial_x\spa{V}_\upm$ with $\partial_x\spa{V}_\upm=\{\xi: \text{ It exists } \zeta\text{ with } \partial_x\zeta=\xi\}$ 
\item[A2)] $\{b\in \spa{H}^1(\omega):\,\partial_xb=0\}\subset\spa{V}_\upm$
\item[A3)] $1\in\spa{V}_\upe$
\end{itemize}
\end{assum}
\noindent With this mass conservation and energy dissipation can be proved analogously to Theorem \ref{Theo:MassEn}. Let $T_\uph(\omega)=\{\omega_j\}$ be a uniform partition of the pipe $\omega=(0,L)$, such that we have
\begin{align*}
\omega=\bigcup_{j=1}^n\omega_j=\bigcup_{j=1}^{n}[x_{j},x_{j+1}],
\end{align*}
where the $x_i$, $i=1,\dots, n+1$, are the grid points of the partition.
Finite dimensional subspaces which fulfill Assumption \ref{As:Comp} are, e.g., the $\spa{P}_0$ finite elemet with a discontinuous piecewise constant basis, i.e., $\spa{V}_\rho=\operatorname{span}\{\psi_1,\dots,\psi_n\}$, with
\begin{align*}
\psi_j(x)=\begin{cases}1 \quad \text{if } x\in(x_j,x_{j+1})\\
                       0  \quad \text{else}
\end{cases}, j=1,\dots,n,
\end{align*}
and the $\spa{P}_1$ finite element with a nodal basis for $\spa{V}_\upm=\operatorname{span}\{\varphi_1,\dots,\varphi_{n+1}\}$, i.e.,
\begin{align*}
\varphi_i(x_\iota)&=\delta_{i\iota},\, i,\iota=1,\dots,n+1, \quad \text{and} \\
\varphi_1(x)&=\begin{cases} \frac{x_2-x}{x_2-x_1} \quad &\text{if } x\in \omega_1\\0 &\text{else} \end{cases}, \quad \varphi_{n+1}(x)=\begin{cases}\frac{x-x_{n}}{x_{n+1}-x_n} \quad &\text{if } x\in \omega_n\\
0\quad &\text{else} \end{cases},	\\			
\varphi_{i+1}(x)&=\begin{cases} \frac{x-x_i}{x_{i+1}-x_i} \quad &\text{if } x\in \omega_i\\
\frac{x_{i+2}-x}{x_{i+2}-x_{i+1}} \quad &\text{if } x\in \omega_{i+1}\\
0 \quad &\text{else}
\end{cases}, i=1,\dots,n-1.
\end{align*}  
In the implementation we use the $\spa{P}_1$ finite element basis also for $\spa{V}_\upe$. For a coordinate representation see Appendix \ref{App:FOMCoRep}.

\subsection{Snapshot-based Model Order Reduction}\label{Sec:MOR}
When finite element discretizations lead to very large systems, model order reduction can be employed to make simulation feasible. For non-linear systems proper orthogonal decomposition (POD) is a commonly used model reduction method and elaborate research on it has been done over the last years, see, e.g., \cite{ChaS12,GraHV21,KunV02} for basics, expansions and error bounds. Thus, the reduction procedure applied to Coordinate Representation \ref{Sys:CoRep} is purely projection based and relies on the snapshot matrix $\mat{S}$, i.e.,
\begin{align}\label{eq:Snapshot}
\mat{S}=\begin{bmatrix}
\mat{S}_\uprho\\\mat{S}_\upm\\\mat{S}_\upe
\end{bmatrix}, \quad \text{with} \quad \mat{S}_\uprho&=[\brho^0\quad\dots\quad\brho^{n_t}],\quad \mat{S}_\upm=[\bm^0\quad\dots\quad\bm^{n_t}],\quad \mat{S}_\upe&=[\be^0\quad\dots\quad\be^{n_t}].
\end{align} 
As the spaces of the Lagrange multipliers will not be reduced in order to preserve the differential-algebraic structure, the snapshots of $\blam_\upm$ and $\lambda_\upe$ are not needed for the computation of the projection matrices and we do not consider them in \eqref{eq:Snapshot}. The superscript in \eqref{eq:Snapshot} denotes the time step $t_l$, $l=0,\dots,n_t$, at which the states are evaluated. Applying POD to these snapshot matrices then yields projection matrices which could be used for model order reduction. But in the context of pH systems, the structure needs to be preserved and this is generally not the case when simply applying POD. Therefore, problem-dependent compatibility conditions need to be developed, which are then enforced onto the projection matrices. For the linear damped wave equations and the isothermal Euler equations this has been done in \cite{EggKLSMM18} and \cite{LilSM21}, respectively. The papers \cite{ AfkH16,AfkH19,PenM16} investigate model order reduction for canonical symplectic structures, whose Hamiltonian structure can be preserved under certain compatibility conditions.  In the following we adapt the approach of \cite{LilSM21}.
\subsubsection{Computation of Projection Matrices}\label{Sec:ProjMat}
\noindent As we reduce the space dimension of a pH system, we need to preserve the structure in order to keep properties like mass and energy conservation. For this, the reduced order spaces $\spa{V}_{\uprho,\upr}\subset\spa{V}_\uprho$, $\spa{V}_{\upm,\upr}\subset\spa{V}_\upm$ and $\spa{V}_{\upe,\upr}\subset\spa{V}_\upe$ need to fulfill the compatibility conditions stated in Assumption \ref{As:Comp}. The algebraic equivalent of Assumption \ref{As:Comp} is given as follows,
\begin{assum}\label{As:CompAlg}
Let the reduction basis $\mat{V}_\upr$ have the following structure
\begin{align*}
\mat{V}_\upr=\begin{bmatrix}
\mat{V}_\uprho & \sz &\sz &\sz\\
\sz & \mat{V}_\upm & \sz&\sz\\
\sz&\sz&\mat{V}_\upe&\sz\\
\sz&\sz&\sz&\so
\end{bmatrix}.
\end{align*}
Then $\mat{V}_\upr$ is assumed to fulfill
\begin{itemize}
\item[$A1_\uph$)] $\operatorname{image}(\mat{M}_\uprho \mat{V}_\uprho)=\operatorname{image}(\mat{J}_{\uprho,\upm}\mat{V}_\upm)$
\item[$A2_\uph$)] $\operatorname{kernel}(\mat{J}_{\uprho,\upm})\subset\operatorname{image}(\mat{V}_\upm)$
\item[$A3_\uph$)] $[1,\dots,1]^T\in\operatorname{image}(\mat{V}_\upe)$
\end{itemize}
\end{assum}
\noindent Here, $\mat{M}_\uprho=[(\psi_q,\psi_j)]_{j,q=1\dots n_1}$ and $\mat{J}_{\uprho,\upm}=[-(\partial_x\varphi_\iota,\psi_j)]_{j=1\dots n_1,\iota=1\dots n_2}$. To keep the pH structure, it is important to create a block-diagonal projection matrix $\mat{V}_\upr$, such that the blocks in the system of equations in the Coordinate Representation \ref{Sys:CoRep} are not mixed. The identity matrix in the last row and column of $\mat{V}_\upr$ makes sure that the algebraic equations are not reduced. To obtain the bases $\mat{V}_\uprho$, $\mat{V}_\upm$ and $\mat{V}_\upe$ of  $\spa{V}_{\uprho,\upr}$, $\spa{V}_{\upm,\upr}$ and $\spa{V}_{\upe,\upr}$, respectively, we start by computing the basis $\mat{V}_\uprho$ from the snapshots \eqref{eq:Snapshot} by Algorithm \ref{Algo:PODW1W3}, which uses POD \cite{KunV01}, see Algorithm \ref{Algo:POD}.
\begin{algo}[Computation of $\mat{V}_\uprho$]\label{Algo:PODW1W3}
Let $\mat{S}$, $\mat{M}_\uprho$, $\mat{J}_{\uprho,\upm}$ and $r_\uprho$ with $0<r_\uprho<n$ be given.
\begin{enumerate}
\item[(1)] Set up the snapshot matrices $\mat{S}_\uprho$, $\mat{S}_\upm$ and $\mat{S}_\upe$ from $\mat{S}$.
\item[(2)] Construct a reduced basis $\mat{V}_\uprho$ from $[\mat{S}_\uprho\quad\mat{M}_\uprho^{-1}(\mat{J}_{\uprho,\upm}[\mat{S}_\upm,\,\mat{S}_\upe])]$ by POD, i.e., Algorithm \ref{Algo:POD}, w.r.t. the scalar product induced by $\mat{M}_\uprho$ of dimension $r_\uprho$.
\end{enumerate}
\end{algo}
\noindent In the computation of $\mat{V}_\uprho$ we not only make use of the snapshots $\mat{S}_\uprho$, but also of $\mat{M}_\uprho^{-1}\mat{J}_{\uprho,\upm}\mat{S}_\upm$, as can be seen in step (2) of Algorithm \ref{Algo:PODW1W3}. This is motivated by Assumption \ref{As:CompAlg}-$A1_\uph)$. As $\operatorname{image}(\mat{M}_\uprho \mat{V}_\uprho)=\operatorname{image}(\mat{J}_{\uprho,\upm}\mat{V}_\upm)$ and $\mat{M}_\uprho$ invertible, it follows that $\operatorname{image}(\mat{V}_\uprho)=\operatorname{image}(\mat{M}_\uprho^{-1}\mat{J}_{\uprho,\upm}\mat{V}_\upm)$. Furthermore, it resembles including the time derivatives of $\rho$ into the snapshot matrix, i.e., $\partial_t\boldsymbol{\rho}=\mat{M}_\uprho^{-1}\mat{J}_{\uprho,\upm}\bm$. This has been proven useful in literature \cite{IliW13} and yields more information on the solution behavior. Additionally, numerical experiments showed that adding $\mat{M}_\uprho^{-1}\mat{J}_{\uprho,\upm}\mat{S}_\upe$ to the computation of $\mat{V}_\uprho$ enhances the approximation quality of the reduced bases and yields more robust reduced order models. If the spaces $\spa{V}_\upm$ and $\spa{V}_\upe$ are constructed by the same finite element bases, it follows that the functions $e(t_l,x)=\sum_{i=0}^{n}[\mat{S}_\upe]_{i,l}\varphi_i(x)\in\spa{V}_\upm$, such that Assumptions \ref{As:Comp}-$A1$ and \ref{As:CompAlg}-$A1_\uph)$ are not violated by including these functions or snapshots. 
\begin{algo}[Proper orthogonal decomposition]\label{Algo:POD}
Let $\mat{S}$, a symmetric positive definite $\mat{M}$ and $r>0$ be given.
\begin{enumerate}
\item[(1)] Set up $\mat{Y}=\sqrt{\mat{M}}\mat{S}$.
\item[(2)] Compute a singular value decomposition of $\mat{Y}$: Find matrices  $\mat{U}$, $\boldsymbol{\Sigma}$, $\mat{W}$, such that $\mat{Y}=\mat{U}\boldsymbol{\Sigma}\mat{W}^T$.
\item[(3)] The POD basis $\mat{V}$ associated to $\mat{S}$ is then given by the first $r$ columns of $\sqrt{\mat{M}}^{-1}\mat{U}$.
\end{enumerate}
\end{algo}
\noindent Finally, the compatibility conditions are enforced, see Algorithm \ref{Algo:CompBasis}, which has been adapted from \cite{LilSM21}. Here, $\mat{V}_\upm$ and $\mat{V}_\upe$ are deduced from $\mat{V}_\uprho$, which contains information of $\brho$, $\bm$ and $\be$.
\begin{algo}[Compatible basis]\label{Algo:CompBasis}
Let  $\mat{V}_\uprho$, $\mat{M}_\uprho$, $\mat{M}_\upm$, $\mat{J}_{\uprho,\upm}$ be given.
\begin{enumerate}
\item[(1)] Compute the kernel $\mat{N}$ of $\mat{J}_{\uprho,\upm}$.
\item[(2)]
\begin{enumerate}
\item Compute $\mat{W}_\upm=\mat{M}_\upm^{-1}\mat{J}_{\uprho,\upm}^T(\mat{J}_{\uprho,\upm}\mat{M}_\upm^{-1}\mat{J}_{\uprho,\upm}^T)^{-1}\mat{M}_\uprho \mat{V}_\uprho$.
\item Orthogonalize the columns of $[\mat{W}_\upm\quad\mat{N}]$ w.r.t. the inner product induced by $\mat{M}_\upm$. This yields $\mat{V}_\upm$.
\end{enumerate}
\item[(3)] Set $\mat{V}_\upe$ to be equal to $\mat{V}_\upm$.
\end{enumerate}
\end{algo}
\noindent Here, $\mat{M}_\upm=[(\varphi_i,\varphi_\iota)]_{i,\iota=1,\dots,n_2}$ denotes the mass matrix related to finite element space $\spa{V}_\upm$. Lines (2a) and (2b) in Algorithm \ref{Algo:CompBasis} show how the compatibility conditions \ref{As:CompAlg}-$A1_\uph)$ and \ref{As:CompAlg}-$A2_\uph)$ are enforced for the reduced space $\spa{V}_{\upm,\upr}$. Choosing the same basis for $\spa{V}_{\upm,\upr}$ and $\spa{V}_{\upe,\upr}$ might seem odd, but experiments showed that this choice enhances the approximation quality of the reduced models. This creates more symmetry in the reduced models, as $\tilde{\mat{J}}_{\upm,\upe}$ and $\tilde{\mat{J}}_{\upe,\upm}$ will still be symmetric after reduction. Furthermore, the compatibility condition \ref{As:CompAlg}-$A3_\uph)$ is fulfilled. The latter follows from the kernel of $\mat{J}_{\uprho,\upm}$ being one dimensional, since it is simply the constant function, i.e.,  $\operatorname{span}\{\mathbf{1}\}=\operatorname{span}\{\operatorname{kernel}(\mat{J}_{\uprho,\upm})\}$.
\subsection{Complexity Reduction}\label{Sec:CompRed}
Recently, empirical quadrature has become popular for structure-pre\-serving complexity reduction, e.g., for energy conservation, see \cite{FarACC14} and \cite{HerCF17}. Here, the underlying inner-product is rewritten into a weighted sum with only a few non-zero weights. The weights are learned in an optimization problem. In \cite{LilSM21} this ansatz was first used for complexity reduction of non-linear pH systems. We expand this approach from the barotropic to the non-isothermal Euler equations. Essentially, we are looking for an approximation of the $\spa{L}^2$ inner-product, as given in Definition \ref{Def:CompScal}.
\begin{definition}[\cite{LilSM20}]\label{Def:CompScal}
Let a partitioning of $\omega$ be given, i.e., $\omega=\bigcup_{j=1}^n\omega_j$. Furthermore, let an index set $J\subset~\{1,\dots,n\}$ and weights $w_j$, $j\in J$, be given. Then, we can define the complexity reduced bilinear form $(\cdot,\cdot)_\upc:\spa{L}^2(\omega)\times \spa{L}^2(\omega)\mapsto\mathbb{R}$ and $||\cdot||_\upc$ by
\begin{align*}
(b,\bar{b})_\upc=\sum_{j\in J}w_j\int_{\omega_j}b(x)\bar{b}(x)\D x \quad \text{and} \quad ||b||_\upc=\sqrt{(b,b)_\upc}.
\end{align*}
\end{definition}
\noindent Furthermore, we adapt the following assumption for the new inner-product from \cite{LilSM20}.
\begin{assum}\label{As:ComRed}
The bilinear form $(\cdot,\cdot)_\upc$ is given as in Definition \ref{Def:CompScal} with $w_j>0$ for $j\in J$. Moreover, there exists a constant $\tilde{C}$, such that
\begin{align*}
\frac{1}{\tilde{C}}||b||_\upc\leq||b||\leq\tilde{C}||b||_\upc \quad \text{for} \quad b\in \spa{V}_\uprho\cup\spa{V}_\upm\cup\spa{V}_\upe.
\end{align*}
\end{assum}
\noindent The positivity of the weights is needed to preserve the positive semi-definiteness of the dissipative part of the pH system. The second assumption makes sure that the original and the complexity reduced norm are equivalent. The weak formulation including the approximated inner product is then given as follows.
\begin{weak}\label{Sys:CompRed}
Let $\BfloZ\neq 0$ and $\eZ\neq 0$. Find $z_\uph=[\rho_h\quad m_\uph\quad e_\uph]^T\in \spa{C}^1([0,t_f],\spa{V}_\uprho\times\spa{V}_\upm\times\spa{V}_\upe)$ and $\blam_\upm\in\spa{C}^0([0,t_f],\mathbb{R}^2)$, $\lambda_\upe\in\spa{C}^0([0,t_f],\mathbb{R})$, $\vect{\Bflo}\in\spa{C}^1([0,t_f],\mathbb{R}^2)$ fulfilling
\begin{align}
(\partial_t\rho_\uph,\psi)&=-(\partial_xm_\uph,\psi),\nonumber\\
(-\frac{m_\uph}{\rho^2_\uph}\partial_t\rho_\uph+\frac{1}{\rho_\uph}\partial_tm_\uph,\varphi)_\upc&=(\frac{m_\uph^2}{2\rho_\uph^2},\partial_x\varphi)_\upc-(\frac{e_\uph}{\rho_\uph}\partial_x1_\uph,\varphi)_\upc-(\frac{1}{\rho_\uph}\partial_x(p_\uph 1_\uph),\varphi)_\upc\nonumber\\&\quad-(\frac{\lambda}{2d}\frac{m_\uph|m_\uph|}{\rho^2_\uph}1_\uph,\varphi)_\upc+\blam_\upm\cdot T_\upm(\varphi),\nonumber\\
(\partial_te_\uph,\phi)&=(\frac{e_\uph}{\rho_\uph}m_\uph,\partial_x\phi)_\upc+(\frac{1}{\rho_\uph}m_\uph,\partial_x(\phi_\iota p_\uph))_\upc+(\frac{\lambda}{2d}\frac{m^2_\uph|m_\uph|}{\rho^2_\uph},\phi)_\upc\nonumber\\&\quad-(\frac{k_\omega}{d}T_\uph 1_\uph,\phi)_\upc+(\frac{k_\omega}{d}T_\infty,\phi)+\lambda_e\cdot T_\upe(\phi)+\frac{e+p}{\rho}|_L\phiL\BfloL,\nonumber\\
0&=-T_\upm m_\uph +\vect{\Bflo},\nonumber\\
0&=-T_\upe 1_\uph+\frac{\eZ}{\BfloZ}\BfloZ\nonumber,
\end{align}
for all $\psi\in\spa{V}_\uprho$, $\varphi\in\spa{V}_\upm$ and $\phi\in\spa{V}_\upe$. Assumptions \ref{As:Comp} and \ref{As:ComRed} are supposed to hold.
\end{weak}
\noindent The new inner-product is only used in the terms which include non-linearities. This makes it possible to prove mass conservation analogously to Theorem \ref{Theo:MassEn}. As for the proof of energy dissipation, we need to make some small adjustments, e.g., introducing a complexity reduced Hamiltonian, see \cite{Hau24_Dis} and \cite{LilSM21} for the proof and an approach to learn the weights.
\begin{theorem}\label{Theo:CompRedEnDis}
For any solution $z_\uph$ of the Weak Formulation \ref{Sys:CompRed} the following inequality holds
\begin{align*}
\frac{\D}{\D t}\Ham_\upc(z_\uph)&=-(\frac{k_\omega}{d}T_\uph 1_\uph,1_\uph)_\upc+ \int_\omega\frac{k_\omega}{d}T_\infty1_\uph\D x -\left.[m_\uph(\frac{m_\uph^2}{2\rho_\uph^2}+\frac{1}{\rho_\uph}(e_\uph+p_\uph))]\right|_0^L\leq(\Pflo, \Peff) +\vect{\Bflo}^T\vect{\Beff}
\end{align*}
with complexity reduced Hamiltonian $\Ham_\upc(z_\uph)=(\frac{m_\uph^2}{2\rho_\uph},1)_\upc+(e_\uph,1).$
\end{theorem}
\subsection{Modification of Approximation Approach}\label{Sec:Mod}
As the pH formulation of the network system is derived using energy-preserving coupling between the single pipes, the Weak Formulation \ref{Sys:wpHDCNet} on the whole network, has almost the same structure as Weak Formulation \ref{eq:VarPrinc}. The only differences are the additional algebraic equations for mass conservation and enthalpy equality at the coupling nodes. Therefore, for structure-preserving space discretization we can utilize the same compatibility conditions as in Assumption \ref{As:Comp} using the function spaces on the network, as introduced in Section \ref{sec:NetTop}.
\begin{assum}\label{As:CompNet}
Let $\spa{V}^\Net=\spa{V}^\Net_\uprho\times\spa{V}^\Net_\upm\times\spa{V}^\Net_\upe$ and $\spa{V}^\Net_\uprho\subset \spa{L}^2(\Net)$, $\spa{V}^\Net_\upm,\spa{V}^\Net_\upe\subset \spa{H}^1_{pw}(\Net)$ be finite dimensional subspaces which fulfill the following assumptions:
\begin{itemize}
\item[A1)] $\spa{V}^\Net_\uprho=\partial_x\spa{V}^\Net_\upm$ with $\partial_x\spa{V}^\Net_\upm=\{\xi: \text{ It exists } \zeta\text{ with } \partial_x\zeta=\xi\}$ 
\item[A2)] $\{b\in \spa{H}_{pw}^1(\Net):\,\partial_xb=0\}\subset\spa{V}^\Net_\upm$
\item[A3)] $1\in\spa{V}^\Net_\upe$
\end{itemize}
\end{assum}
\noindent The semi-discretized network system can be formulated using the same pipe-wise spaces as for the single pipe system. See Appendix \ref{Sec:CoRepNet} for the coordinate formulation of the network system. As networks of pipes in real-world applications lead to highly non-linear semi-discretized systems of very high dimensions, we need to apply model order and complexity reduction to be able to simulate them. For this, the algorithms from Sections \ref{Sec:MOR} and \ref{Sec:CompRed} are applicable to the network system with minor alterations. The algebraic version of the compatibility conditions in Assumption \ref{As:CompNet} are established as follows.
\begin{assum}\label{As:CompAlgNet}
Let the reduction basis $\mat{V}^\Net_\upr$ have the following structure
\begin{align*}
\mat{V}^\Net_\upr=\begin{bmatrix}
\mat{V}^\Net_\uprho & \sz &\sz &\sz\\
\sz & \mat{V}^\Net_\upm & \sz&\sz\\
\sz&\sz&\mat{V}^\Net_\upe&\sz\\
\sz&\sz&\sz&\so
\end{bmatrix}.
\end{align*}
Then $\mat{V}^\Net_\upr$ is assumed to fulfill
\begin{itemize}
\item[$A1_\uph$)] $\operatorname{image}(\mat{M}^\Net_\uprho \mat{V}^\Net_\uprho)=\operatorname{image}(\mat{J}^\Net_{\uprho,\upm}\mat{V}^\Net_\upm)$
\item[$A2_\uph$)] $\operatorname{kernel}(\mat{J}^\Net_{\uprho,\upm})\subset\operatorname{image}(\mat{V}^\Net_\upm)$
\item[$A3_\uph$)] $[1,\dots,1]^T\in\operatorname{image}(\mat{V}^\Net_\upe)$
\end{itemize}
\end{assum}
\noindent Similar to model order reduction on the single pipe system, we do not reduce the space for the algebraic equations, as this might lead to instabilities and loss of system properties. For the computation of the reduction basis $\mat{V}_\upr^\Net$ there are two different approaches in the network case. 
\begin{itemize}[leftmargin=2cm]
\item[$A_{\Net}$] On the one hand, we can compute a reduction basis from snapshots of the whole network, i.e., $\mat{S}^\Net_\uprho$, $\mat{S}^\Net_\upm$ and $\mat{S}^\Net_\upe$, by the Algorithms \ref{Algo:PODW1W3}, \ref{Algo:POD}, \ref{Algo:CompBasis}, which leads to dense matrices $\mat{V}_\uprho^\Net$, $\mat{V}_\upm^\Net$ and $\mat{V}_\upe^\Net$ and thus, to a reduced order model without pipe-wise structure.
\item[$A_{\omega}$]On the other hand, the reduction basis can also be computed pipe-wise, such that the reduction matrices $\mat{V}_\uprho^\Net$, $\mat{V}_\upm^\Net$ and $\mat{V}_\upe^\Net$ are sparse and block-diagonal, e.g.,
\begin{align*}
\mat{V}_\uprho^\Net=\begin{bmatrix}
\mat{V}_\uprho^{\omega_1} & \sz & \cdots &\sz\\
\sz &\ddots&&\vdots\\
\vdots&&\ddots&\sz\\
\sz&\cdots&\sz&\mat{V}_\uprho^{\omega_k}
\end{bmatrix}.
\end{align*}
To set up these reduction matrices we need to run the Algorithms \ref{Algo:PODW1W3}, \ref{Algo:POD}, \ref{Algo:CompBasis} on the snapshot matrices for each pipe, i.e., $\mat{S}^{\omega_i}_\uprho$, $\mat{S}^{\omega_i}_\upm$ and $\mat{S}^{\omega_i}_\upe$, $i=1,\dots,k$.
This leads to a reduced order model where the single pipes can still be identified.
\end{itemize}
In both cases the compatibility conditions from Assumption \ref{As:CompAlgNet} are fulfilled. Analogously, the complexity reduction can be computed pipe-wise $A_\omega^\upc$ or for the whole network $A_\Net^\upc$.

\section{Numerical Results}\label{Sec:Num}
\noindent All computations are carried out in SI-Units. For a standard pipe set-up we have
\begin{equation}\label{Set-Up}
\begin{aligned}
L=1\,m,\quad d=0.1\,m,\quad \lambda=4,\quad k_\omega=\frac{1}{2}\,\frac{W}{m^2K},\quad T_\infty=1\,K, \quad R=1\,\frac{J}{K},\quad c_v=\frac{5}{2}\,\frac{J}{K},\quad c_p=\frac{7}{2}\,\frac{J}{K},
\end{aligned}
\end{equation}
such that the flow of an ideal two-atomic gas is modeled, see \cite{LeV92}. This set-up is similar to the one used in \cite{Egg16}. Time is measured in seconds, e.g., $t_f=30s$, and we have $\rho\,[\frac{kg}{m^3}]$, $m\,[\frac{kg}{m^2s}]$ and $e\,[\frac{J}{m^3}]$. In the following, we suppress all units for better readability. The dimension of a semi-discretized system is $3n^\Net+2|\Net|+n_\lambda$, where $n^\Net=\sum_{i=1}^{|\Net|}n^{\omega_i}$ with $n^{\omega_i}$ being the number of finite elements on pipe $\omega_i$, $|\Net|$ is the number of pipes in the network and $n_\lambda$ adds the number of the Lagrange multipliers. These systems are then solved using the implicit Euler scheme with a temporal step size of $\tau=0.1$ and Newton's method with the analytical Jacobian, a tolerance of $tol=10^{-10}$ and a maximum number of 20 iterations. Furthermore, consistent initial values for the Lagrange multipliers are computed or educated guesses have been made in the network case.
To measure the quality of the reduced order models (ROMs) with respect to the full order model (FOM) the maximal relative $\spa{L}^2$-error is considered, i.e., 
$$\mathfrak{E}_t=\max_{t\in[0,t_f]}\frac{||z(t,x)-z_\upr(t,x)||_{\spa{L}^2}}{||z(t,x)||_{\spa{L}^2}}.$$ 
This error is also called model reduction (MOR) error throughout the next sections. Here, $z_\upr(t,x)$  denotes the approximation of the full state $z(t,x)$. Additionally, the maximal relative projection error $\mathfrak{E}_{t,P}$ is computed, i.e., the error when $z_r(t,x)$ is set to be the orthogonal projection of the solution trajectory onto the respective reduced space. This projection is computed by $\mat{V}^\dag\vect{z}$ with $\mat{V}^\dag$ being the pseudo-inverse of $\mat{V}$. It thus shows the pure projection error made in MOR and serves as a lower bound for the expected model reduction error. Furthermore, the following error  
\begin{align*}
\mathfrak{E}_{\mathrm{pH}}=\max_{t\in[0,t_f]}||\mat{E}(\vect{z}(t))^T\vect{\eff}(\vect{z}(t))-\nabla_{\vect{z}(t)}\Ham(\vect{z}(t))||_2
\end{align*}
shows, if a system violates condition \eqref{Cond:Volker} and thus, is not a pH-system. Here, $||\cdot ||_2$ denotes the standard Euclidean norm. Additionally, we denote by $r_\uprho^\Net=\sum_{i=1}^{|\Net|}r_\uprho^{\omega_i}$ the reduced counter part to $n^\Net$, such that the reduced models are of dimension $3r_\uprho^\Net+2|\Net|+n_\lambda$.
The ROMs are initialized with
\begin{align*}
\boldsymbol{\rho}_\upr(0)=\mat{V}_\upr^\dag\boldsymbol{\rho}(0),\quad \vect{m}_\upr(0)=\mat{V}_\upm^\dag\vect{m}(0), \quad \vect{e}_\upr(0)=\mat{V}_\upe^\dag\vect{e}(0).\nonumber 
\end{align*}
All experiments have been carried out on MATLAB Version R2022a on a Linux 64-Bit machine with an Intel\textsuperscript{\textregistered} Core\textsuperscript{\texttrademark} i7-6700 processor using the Tensor Toolbox for MATLAB \cite{Toolbox}.
 \begin{figure}[!t]
\includegraphics[scale=1]{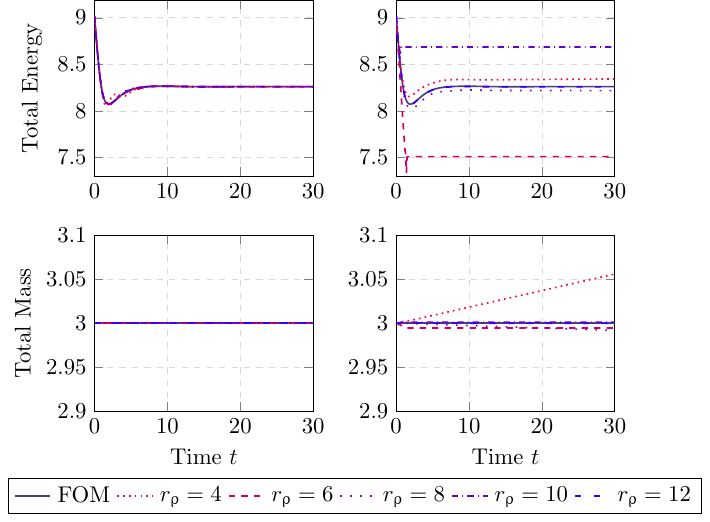}
\caption{Total energy and mass with (left) and without (right) compatibility conditions.} \label{fig:EMS_CoC}
\end{figure}

\subsection{Model Order Reduction for a Single Pipe}
Since the long term goal is to simulate real-world gas networks, which leads to semi-discretized systems of high dimensions, there is an immediate need for MOR. Therefore, we check the effectiveness and the error behavior of the MOR procedure from Section \ref{Sec:MOR}. For this, we compare our MOR procedure to standard MOR based on POD without compatibility conditions, see Section \ref{Sec:Num_CoC}. The parity of the FOM plays an important role in model reduction and is therefore studied in Section \ref{Sec:Num_Parity}. 

\subsubsection{Influence of Compatibility Conditions on Model Order Reduction}\label{Sec:Num_CoC}
We consider the flow of an ideal gas modeled by System \ref{Sys:pHDCm} with set-up \eqref{Set-Up}. The initial values and the boundary conditions are $\rho(0,x)=3, \, m(0,x)=0.3,\, e(0,x)=9$ and $m(t,0)=m(t,1)=0.3, \, e(t,0)=9$, respectively. The FOM is discretized with a spatial step size $\Delta x=0.01$, resulting in a dimension of $3n+2+3=305$ with $n=100$. The FOM is simulated until the stationary solution is reached at $t_{f}=30$. The ROMs are solved using the same boundary conditions and parameters as the FOM. For the MOR with compatibility conditions the projection matrices are computed with the snapshot matrix \eqref{eq:Snapshot}, set up from the solutions of the FOM, and the Algorithms \ref{Algo:PODW1W3} and \ref{Algo:CompBasis}. For the MOR without compatibility conditions the projection matrices are also computed from the snapshots \eqref{eq:Snapshot}, but only by Algorithm \ref{Algo:POD}. Here, $\mat{M}$ is set to be the mass matrix with respect to the state variable for which the projection matrix is computed, i.e., $\mat{M}_\uprho$, $\mat{M}_\upm$ and $\mat{M}_\upe$ for $\boldsymbol{\rho}$, $\bm$ and $\be$, respectively. Furthermore, $r=r_\uprho$ for $\boldsymbol{\rho}$ and $r=r_\uprho+1$ for $\bm$ and $\be$ to achieve comparability between the ROMs computed by MOR with and MOR without compatibility conditions. Figure \ref{fig:EMS_CoC} shows the energy dissipation and the mass conservation for the FOM and ROMs with various values of $r_\uprho$. The plots on the left show that the energy dissipation and mass conservation of the FOM is very well approximated when the ROMs are computed with the compatibility conditions. On the other hand, when the compatibility conditions are not used during MOR both properties can be lost and thus, the ROMs are bad approximations of the FOM, e.g., for $r_\uprho=6$ or $r_\uprho=10$. These behaviors also become visible in Figure \ref{fig:Err_CoC}, which shows $\mathfrak{E}_t$ and $\mathfrak{E}_{t,P}$ for different $r_\uprho$. The excellent error behavior in the left plot is only due to the compatibility conditions from Assumption \ref{As:CompAlg}. Here, the errors decline rapidly with increasing dimension. Whereas for MOR without compatibility conditions, only $\mathfrak{E}_{t,P}$ declines rapidly and the reduction errors show that even with increasing $r_\uprho$  a good approximation of the FOM is not guaranteed. Table \ref{tab:Cond_Err} shows $\mathfrak{E}_{\mathrm{pH}}$ for each ROM. Naturally, each ROM computed with the compatibility conditions from Assumption \ref{As:Comp} fulfills condition \eqref{Cond:Volker}. For the ROM with $r_\uprho=6$ computed without compatibility conditions, the pH-structure clearly is lost, as the error is $10^{-3}$. All others still fulfill condition \eqref{Cond:Volker}, even though the ROM with $r_\uprho=10$ does not yield a good approximation. These findings are in accordance with the results for the isothermal Euler equations, see \cite{LilSM21}.
\begin{figure}[!tb]
\includegraphics[scale=1]{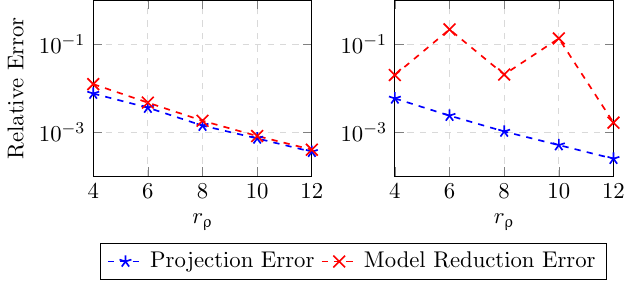}
\caption{$\mathfrak{E}_{t,P}$ and $\mathfrak{E}_t$ for ROMS with (left) and without (right) compatibility conditions.}\label{fig:Err_CoC}
\end{figure}
\begin{table}[t]
\centering
\setlength\extrarowheight{3pt}
\begin{tabular}{l||c|c|c|c|c|} 
$r_\uprho$	& 4 & 6 & 8 & 10 & 12 \\ \hline
With Comp.        &  $5.5\cdot 10^{-18}$ & $5.5\cdot 10^{-18}$  & $5.5\cdot 10^{-18}$  &  $5.5\cdot 10^{-18}$  & $5.5\cdot 10^{-18}$\\ \hline
Without Comp.     &  $5.5\cdot 10^{-18}$ & 0.001  & $5.5\cdot 10^{-18}$  & $5.5\cdot 10^{-18}$   & $5.5\cdot 10^{-18}$
\end{tabular}
\caption{Error with respect to Condition \eqref{Cond:Volker}, i.e., $\mathfrak{E}_\mathrm{pH}$.}
\label{tab:Cond_Err}
\end{table}

\subsubsection{Influence of the Parity of $r_\uprho$ on the Reduced Order Model}\label{Sec:Num_Parity}
\begin{figure}[!t]
\includegraphics[scale=1]{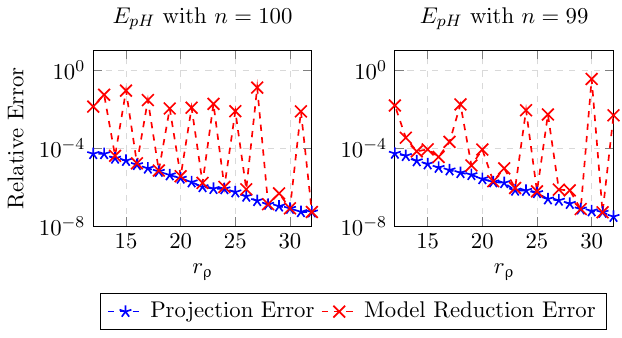}
\caption{$\mathfrak{E}_{t,P}$ and $\mathfrak{E}_{t}$ for different values of $r_\uprho$.}\label{fig:Err_Parity}
\end{figure}
In the following, we show that the parity of the FOM needs to be preserved when deducing a ROM. The set up is similar to the one in Section \ref{Sec:Num_CoC}, except for the boundary conditions, which are set as $\rho(t,0)=3,$ $m(t,L)=0.3,$ and $e(t,0)=9.$ The following numerical findings are similar for other settings of boundary conditions. The number of finite elements used for semi-discretization here is $n=100$, which is even. Furthermore, a second FOM based on System \ref{Sys:pHDCm} with an uneven number of finite element cells is studied, i.e., $n=99$. The ROMs are simulated using the same parameters, initial conditions and boundary values as the FOMs. From this and the findings from Section \ref{Sec:Num_CoC} we would expect that with increasing $r_\uprho$ the errors decrease. This unfortunately is not the case as Figure \ref{fig:Err_Parity} shows. For both scenarios the projection errors decline very rapidly, whereas the ROMs only lead to reduction errors smaller than $10^{-2}$ from $r_\uprho=12$. After this, $\mathcal{E}_{t,P}$ continues in a zig-zag fashion, meaning that the ROMs with $r_\uprho$ having the same parity as $n$ are very good approximations, whereas the ROMs with $r_\uprho$ not having the same parity as $n$ may lead to bad approximation errors. This zig-zag behavior stems from the part of the skew-adjoint interconnection operator $J(z)$ containing the derivative terms, i.e.,
\begin{align*}
J_\mathrm{div}(z)=\begin{bmatrix}0&-D_x&0\\-D_x&0&-\frac{e}{\rho}D_x-\frac{1}{\rho}D_xp\\0&-D_x\frac{e}{\rho}-pD_x\frac{1}{\rho}&0
\end{bmatrix}.
\end{align*}
The semi-discretization of $J_\mathrm{div}(z)$ is skew-symmetric and therefore its eigenvalues are purely imaginary pairs or zero. As \ref{eq:DCv} consist of three equations and three unknowns,
the system matrix $\mat{J}(z_\uph)$ is certainly singular. Therefore, the semi-discretized operator has an even or uneven number of zero eigenvalues, depending on the chosen number of finite element cells. Thus, ROMs with $r_\uprho$ not matching the parity of $n$ of the FOM do not reproduce the correct parity of zero eigenvalues and therefore, the MOR can lead to bad approximations. This is different to MOR applied to the isothermal Euler equations, as there the underlying operator is not singular.

\subsection{Model Order and Complexity reduction for a Network System}
\subsubsection{$A_{\Net}$ versus $A_{\omega}$}\label{Sec:NetVsPipe}
\begin{figure}[!t]
\centering
\begin{tikzpicture}
\draw[->,line width=1pt] (0,-5) -- (3,-5);
\draw[->,line width=1pt] (7,-5)--(10,-5);
\draw[->,line width=1pt] (3.05,-4.95) -- (5,-3.5);
\draw[->,line width=1pt] (3.05,-5.05) -- (5,-6.5);
\draw[->,line width=1pt] (5.05,-3.45) -- (7,-4.95);
\draw[->,line width=1pt] (5.05,-6.55)--(7,-5.05);
\fill (0,-5) circle (0.1);
\node at (0,-5.5) {$\nu_1$};
\fill (3.05,-5) circle (0.1);
\node at (3.05,-5.5) {$\nu_2$};
\fill (5.05,-3.45) circle (0.1);
\node at (5.05,-2.95) {$\nu_3$};
\fill (5.05,-6.55) circle (0.1);
\node at (5.05,-7.05) {$\nu_4$};
\fill (7.05,-5) circle (0.1);
\node at (7.05,-5.5) {$\nu_5$};
\fill (10.05,-5) circle (0.1);
\node at (10.05,-5.5) {$\nu_6$};
\node at (1.5,-5.5) {$\omega_1$};
\node at (3.5,-4) {$\omega_2$};
\node at (3.5,-6) {$\omega_3$};
\node at (8.5,-5.5) {$\omega_6$};
\node at (6.5,-4) {$\omega_4$};
\node at (6.5,-6) {$\omega_5$};
\end{tikzpicture}
\caption{Diamond network as used in Sections \ref{Sec:NetVsPipe} and \ref{Sec:RedNet}}
\label{fig:Diamond}
\end{figure}
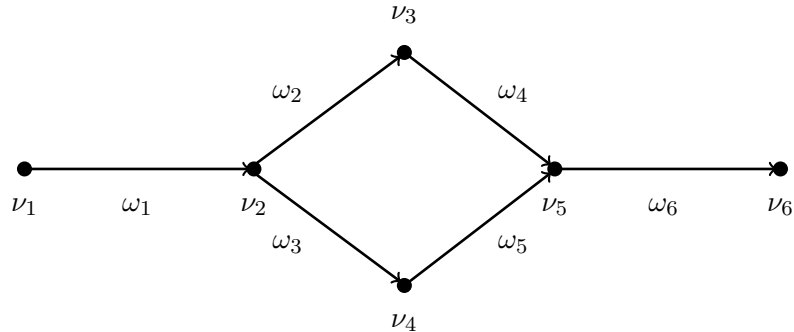
\begin{table}[t]
\centering
\setlength\extrarowheight{3pt}
\begin{tabular}{l||c|c|c|}
$\quad r_\uprho^\Net\quad$	& 28 & 29 & 30  \\ \hline
$\quad \mathfrak{E}_t\quad $& $4.5\cdot 10^{-06}$& $9.5\cdot 10^{-3}$ &  $4.1\cdot 10^{-06}$
\end{tabular}
\caption{$\mathfrak{E}_t$ for different $r_\uprho^\Net$ with $A_\Net$.}
\label{tab:Err_MOR_net}
\end{table}
In this section we take a closer look on the approaches $A_\omega$ and $A_\Net$ for model order and complexity reduction and their impact on the error behavior, see Section \ref{Sec:Mod} for an explanation of these abbreviations. For this, we consider a diamond network, see Figure \ref{fig:Diamond}, and the Set-up \ref{Set-Up}} with minor alterations, i.e., the lengths of the pipes are set to
\begin{align*}
L^{\omega_1}=0.55,\quad L^{\omega_2}=0.5,\quad L^{\omega_3}=0.5,\quad L^{\omega_4}=0.5,\quad L^{\omega_5}=0.5,\quad L^{\omega_6}=0.55.
\end{align*} 
and the cross-sectional area $A=1$ for each pipe. The inital values and boundary conditions given as
\begin{align*}
\rho(0,x)&=3,\quad &m(0,x)&=\begin{cases}0.3 \quad &\text{if } x\in\omega_1\cup\omega_6\\ 0.15\quad &\text{if } x\in\bigcup_{i=2}^5\omega_i\end{cases},\quad &e(0,x)&=9,\\
\rho(t,\nu_1)&=3,\quad &m(t,\nu_6)&=0.3,\quad &e(t,\nu_1)&=9.
\end{align*}
Each pipe is discretized in space with a spatial step size $\Delta x=0.01$, which means that pipe $\omega_1$ and $\omega_6$ have an uneven number of finite elments, i.e., 55, whereas all other pipes have an even number, i.e., 50. This results in $n^\Net=310$, which is even. The FOM is simulated with $\Delta t=0.1$ until the end time $t_f=30$ is reached. This yields a total of $301$ snapshots to compute the compatible projection bases.
\subsubsection*{Model Order Reduction}
The following ROMs are simulated with the same parameters, initial values and boundary conditions as the FOM. As stated in Section \ref{Sec:Mod} the network FOM can be reduced with two different approaches, i.e., $A_\Net$, which reduces the whole network at once, or $A_\omega$, which reduces each pipe separately. When using the approach $A_\Net$ the ROMs show the same error behavior as seen for the single pipe system in Section~\ref{Sec:Num_Parity}, which means that ROMs with an even $r_\uprho$ yield good approximations and ROMs with uneven $r_\uprho$ might yield bad approximations, see Table \ref{tab:Err_MOR_net}. Using the approach $A_\omega$ for MOR has two major drawbacks, as can be seen in Figure~\ref{fig:Err_552}, where $\mathfrak{E}_{t,P}$ and $\mathfrak{E}_t$ are plotted for different even $r_\uprho^\Net$. For $A_\Net$ there is almost no gap between the projection and reduction errors and at $r_\uprho^\Net=60$ the reduction error is already in $\mathcal{O}(10^{-9})$. The case is completely different when using $A_\omega$, here, $\mathfrak{E}_{t,P}$ for $r_\uprho^\Net=60$ is in $\mathcal{O}(10^{-6})$, which is three orders of magnitude higher than for $A_\Net$. Furthemore, $\mathfrak{E}_t$ here is at $\mathfrak{E}_t=0.12$. This is the case, as the $r_\uprho^{\omega_i}$ , $i=1,\dots,6$, for each pipe have been chosen with the wrong parity, i.e., 
$$r_\uprho^{\omega_1}=12,\,r_\uprho^{\omega_2}=9,\,r_\uprho^{\omega_3}=9,\,r_\uprho^{\omega_4}=9,\, r_\uprho^{\omega_5}=9 \text{ and } r_\uprho^{\omega_6}=12,$$ such that $r_\uprho^\Net=\sum_{i=1}^6 r_\uprho^{\omega_i}=60$. When increasing each $r_\uprho^{\omega_i}$ by 1 the parities fit again those of the number of finite elements of the FOM on each pipe, in total $r_\uprho^\Net=66$, and the reduction error drops to $\mathfrak{E}_t=5\cdot 10^{-6}$. Thus, we not only have to deal with higher errors, when using $A_\omega$, but we also have to have knowledge over the discretization of each individual pipe, in order to choose the reduction parameters right.

\begin{figure}[!t]
\includegraphics[scale=1]{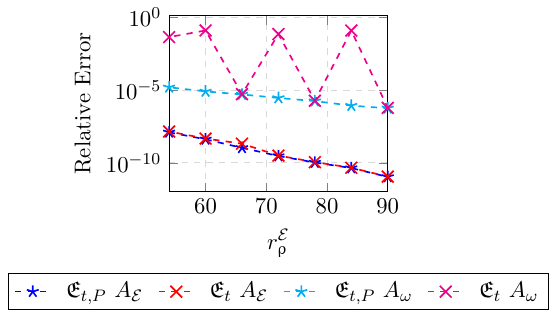}
\caption{Relative $\spa{L}^2$-error for different $r_\uprho^\Net$ and both MOR approaches $A_\Net$ and $A_\omega$.}\label{fig:Err_552}
\end{figure}

\subsubsection*{Complexity Reduction}
As a next step MOR can be combined with complexity reduction. This procedure can also be applied in a pipe-wise ($A_\omega^\upc$) or network-wise ($A_\Net^\upc$) manner. When $A_\omega^\upc$ is used $n_\upc^\Net=\sum_{i=1}^{|\Net|}n_\upc^{\omega_i}$. In this subsection we pair the MOR approaches with $A_\omega^\upc$ and $A_\Net^\upc$ using the empirical quadrature ansatz and DEIM for complexity reduction. Based on ROMs with $r_\uprho^\Net=66$, i.e., $r_\uprho^{\omega_1}=r_\uprho^{\omega_6}=13$ and $r_\uprho^{\omega_i}=10$ for $i=2,\dots,5$, the complexity reduced models are compared. Figure \ref{fig:Err_552_cROM} shows $\mathcal{E}_t$ for the ROMs without complexity reduction in red and magenta, this acts as a best-case-error for the complexity reduced systems. In the left plot of Figure \ref{fig:Err_552_cROM} we see that DEIM does not lead to acceptable complexity reduced systems. Even the errors produced by the empirical quadrature approach show a big gap towards the ROM without complexity reduction, even though almost $46\%$ of the finite elements of the FOM are preserved using $n_\upc^\Net=142$. In the left plot of Figure \ref{fig:Err_552_cROM} we omit the error plot for $\mathfrak{E}_{t,Quad}\,A_\Net^\upc$ for better readability as it does not lead to good approximations in the selected range for $n_\upc^\Net$. When equipping the pipe-wise reduced system with complexity reduction the picture changes. For $A_\Net^\upc$ DEIM leads only to bad approximations, but the empirical quadrature ansatz yields very good approximations from $n_\upc^\Net=136$, as here the error is in $\mathcal{O}(10^{-5})$, which is only one order of magnitude higher than the error of the ROM without complexity reduction. Furthermore, for $n_\upc^\Net=142$ the error is even one order of magnitude smaller than the error produced when $A_\Net^\upc$ using empirical quadrature is applied to the ROM computed by $A_\Net$, which is in $\mathcal{O}(10^{-4})$. As for $A_\omega$ paired with $A_\omega^\upc$ Figure \ref{fig:Err_552_cROM} shows that DEIM fails except for $n_\upc^\Net=136$, i.e., $n_\upc^{\omega_1}=n_\upc^{\omega_5}=25$, $n_\upc^{\omega_2}=n_\upc^{\omega_3}=19$, $n_\upc^{\omega_4}=n_\upc^{\omega_5}=24$, which is one of the few settings where pipe-wise empirical quadrature fails for this example. Otherwise, this approach leads from $n_\upc^\Net=112$ to declining errors, which are in $\mathcal{O}(10^{-5})$ for $n_\upc^\Net=130$ and $n_\upc^\Net=142$, which is one order of magnitude higher than the pure reduction error. Using the approach $A_\omega^\upc$ for complexity reduction gives the user more possibilities to fine tune the empirical quadrature ansatz. 
\begin{table}[t]
\centering
\setlength\extrarowheight{3pt}
\begin{tabular}{l||c|c|c|c|c|c|}
				& $n_\upc^{\omega_1}$&$n_\upc^{\omega_2}$&$n_\upc^{\omega_3}$&$n_\upc^{\omega_4}$&$n_\upc^{\omega_5}$&$n_\upc^{\omega_6}$\\ \hline
$A_\Net^\upc$	&27  & 21 & 21&18&18&25  \\ \hline
$A_\omega^\upc$& 24& 18 &  18&23&23&24
\end{tabular}
\caption{Pipe-wise $n_\upc$ for $A_\Net^\upc$ and $A_\omega^\upc$ with $n_\upc^\Net=130$.}
\label{tab:552_Comp}
\end{table}
For, e.g., $n_\upc^\Net=130$ the pipe-wise complexity reduction approach leads to an error in $\mathcal{O}(10^{-5})$ and the network approach leads to an error of $\mathfrak{E}_t=0.057$. For $A_\Net^\upc$ the distribution of the $n_\upc^\Net=130$ finite elements onto the pipes is done an algortithm \cite{LilSM21}, whereas for $A_\omega^\upc$ this is done by the user. Table \ref{tab:552_Comp} shows the pipe-wise number of preserved finite elements. Eventhough the distributions only differ slightly, the impact on the error of the complexity reduced system is enormous. Lastly, we give an example showing that $A_\Net$ paired with complexity reduction can also lead to better error behavior. Based on a ROM with $r_\uprho^\Net=28$ Figure \ref{fig:Err_552_cROM_2} shows that $A_\Net^\upc$ with empirical quadrature, both $A_\Net^\upc$ and $A_\omega^\upc$, leads to very good approximations. Even pipe-wise DEIM leads to good approximations, whereas $A_\Net^\upc$ using DEIM fails again. We cannot compare this to complexity reduction based on a pipe-wise reduced ROM with $r_\uprho^\Net=28$, as model reduction leads to an error in $\mathcal{O}(10^{-1})$.
\begin{figure}[!t]
\includegraphics[scale=1]{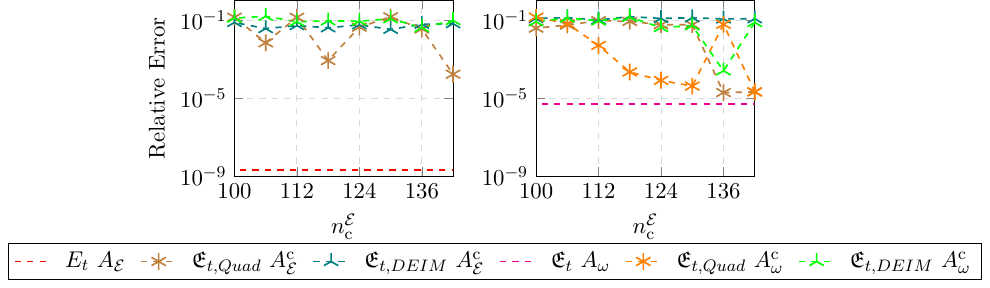}
\caption{Relative $\spa{L}^2$-error for different $n_\upc^\Net$ and both MOR approaches $A_\Net$ (left) and $A_\omega$ (right).}\label{fig:Err_552_cROM}
\end{figure}

\begin{figure}[!t]
\includegraphics[scale=1]{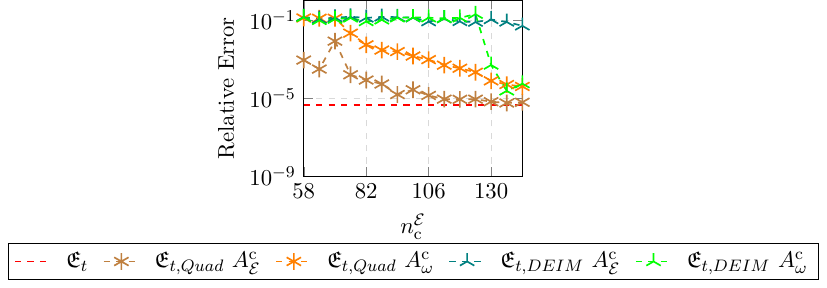}
\caption{Relative $\spa{L}^2$-error for different $n_\upc^\Net$ for $A_\Net$ and $r_\uprho^\Net=28$.}\label{fig:Err_552_cROM_2}
\end{figure}

\subsubsection{Reducibility of a Network System}\label{Sec:RedNet}
\begin{figure}[!t]
\includegraphics[scale=1]{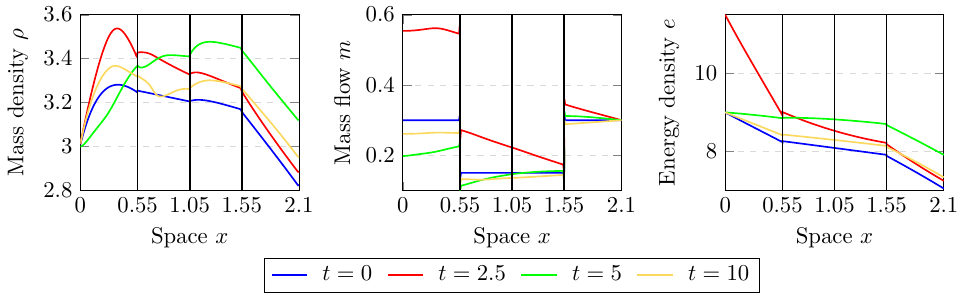}
\caption{State variables $\rho$, $m$ and $e$ along path at different time points.}\label{fig:States_553}
\end{figure}
In this section we examine MOR and complexity reduction for a coupled network system, which couples different pH-models of compressible fluid flow. We choose almost the same set-up as in the previous Section. The only differences are in the boundary condition for the energy density, i.e., a saw tooth,
\begin{align*}
e(t,0)=9+\begin{cases} t \quad &\text{for } 0\leq t<2.5,\\
                      5-t \quad &\text{for } 2.5\leq t <5,\\
                      0\quad &\text{else}, \end{cases}
\end{align*}
and in the different models on each pipe. On pipes $\omega_2$ and $\omega_3$ we set the damping coefficient $\lambda=0$ in the energy equation and on pipes $\omega_4$ and $\omega_5$ we additionally set $k_\omega=0$. For the pH-formulations of these models see \cite{Hau24_Dis}. The solution trajectories along the path $\mathrm{P}=\omega_1\cup\omega_2\cup\omega_4\cup\omega_6$ are plotted for selected time points in Figure \ref{fig:States_553}. Here, it becomes visible that the solutions exhibit discontinuities, whenever more than two pipes meet at a coupling node. 
\begin{figure}[!t]
\includegraphics[scale=1]{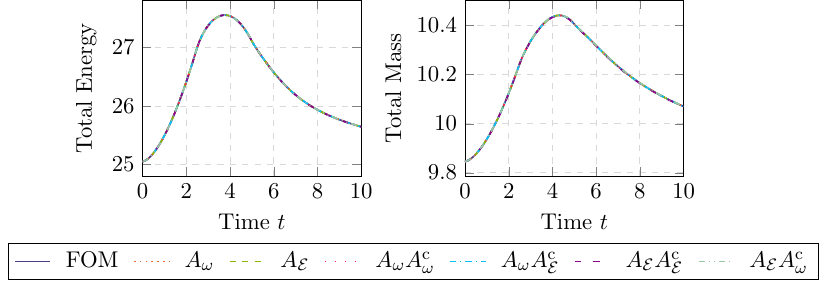}
\caption{Total energy and total mass for each FOM and ROM.} \label{fig:EM_553}
\end{figure}
In the following we apply different approaches for MOR, i.e., $A_\omega$ and $A_\Net$, and complexity reduction, i.e., $A_\omega^\upc$ and $A_\Net^\upc$, to this example and compare their error behavior and run times, see Table \ref{tab:553}. 
\begin{table}[t]
\centering
\setlength\extrarowheight{3pt}
\begin{tabular}{l||c|c|c|c|c|c|c|}
	& FOM &$A_\omega$&$A_\Net$&$A_\omega A_\omega^\upc$&$A_\omega A_\Net^\upc$&$A_\Net A_\Net^\upc$&$A_\Net A_\omega^\upc$\\ \hline
Time & 1,584s& 144s              & 526s             &117s              &126s             &300s  &319s \\ \hline
$\mathfrak{E}_t$& -     & $3.9\cdot10^{-4}$ & $1.4\cdot10^{-6}$&$4.7\cdot10^{-4}$&$3.9\cdot10^{-4}$&$5.3\cdot10^{-4}$&$1.5\cdot 10^{-4}$ \\ \hline
$\mathfrak{E}_{pH}$&$1.3\cdot 10^{-17}$&$2.3\cdot 10^{-16}$&$1.5\cdot 10^{-15}$&$2.3\cdot 10^{-16}$&$2.3\cdot 10^{-16}$&$1.5\cdot 10^{-15}$& $1.5\cdot 10^{-15}$
\end{tabular}
\caption{CPU times and errors for full and order/complexity reduced systems.}
\label{tab:553}
\end{table}
For complexity reduction we only consider the empirical quadrature approach. The FOM is set up analogously to Section \ref{Sec:NetVsPipe}. The following set-ups for model and complexity reduction are exemplary. We start by applying the pipe-wise model reduction approach, i.e., $A_\omega$, with the reduced dimensions $r_\uprho^{\omega_i}$, $i=1,\dots,6$, set to
\begin{align*}
r_\uprho^{\omega_1}=15,\quad r_\uprho^{\omega_2}=4,\quad r_\uprho^{\omega_3}=4,\quad r_\uprho^{\omega_4}=14,\quad r_\uprho^{\omega_5}=14 \quad r_\uprho^{\omega_6}=3,
\end{align*}
which sums up to $r_\uprho^\Net=54$. This MOR approach can be supplemented with with pipe-wise complexity reduction, i.e., $A_\omega A_\omega^\upc$, or network complexity reduction , i.e., $A_\omega A_\Net^\upc$.
\begin{figure}[!t]
\includegraphics[scale=1]{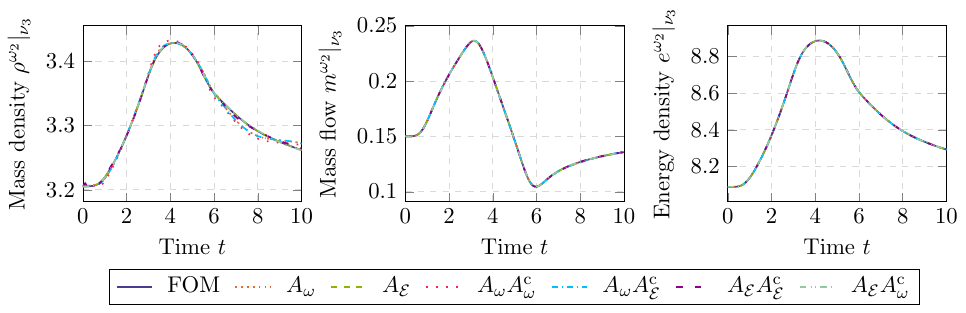}
\caption{State variables $\rho^{\omega_2}$, $m^{\omega_2}$ and $e^{\omega_2}$ at $\nu_3$ over time.}\label{fig:Time_553}
\end{figure}
For the first, we choose the pipe-wise number of preserved finite elements to be
\begin{align*}
n_\upc^{\omega_1}=25\quad n_\upc^{\omega_2}=6\quad n_\upc^{\omega_3}=6\quad n_\upc^{\omega_4}=30\quad n_\upc^{\omega_5}=30\quad n_\upc^{\omega_6}=6,
\end{align*}
which sums up to $n_\upc^\Net=103$. This configuration seems to be the smallest, which yields a complexity reduced system with an error smaller than $10^{-2}$, i.e., in $\mathcal{O}(10^{-4})$. Using $n_\upc^\Net=103$ for $A_\Net^\upc$ yields a relative $\spa{L}^2$-error of $\mathfrak{E}_t=0.0911$. Thus more degrees of freedom are needed, when mixing pipe-wise MOR with network complexity reduction for this example. The smallest number leading to a good approximation is $n_\upc^\Net=130$. Applying MOR for the whole network, instead of to each pipe, with $r_\uprho^\Net=54$ leads to an error of $\mathcal{O}(10^{-6})$, which is two degrees of magnitude smaller than when applying $A_\omega$. Nonetheless, the CPU run time of this reduced system is more than 3.5 times slower than when using $A_\omega$ and only three times faster than the FOM. Lastly, we can combine $A_\Net$ with $A_\Net^\upc$ and $A_\omega^\upc$. For the first, the smallest number yielding a good approximationj is $n_\upc^\Net=118$ and for $A_\omega^\upc$ we use
\begin{align*}
n_\upc^{\omega_1}=35\quad n_\upc^{\omega_2}=30\quad n_\upc^{\omega_3}=30\quad n_\upc^{\omega_4}=35\quad n_\upc^{\omega_5}=35\quad n_\upc^{\omega_6}=30.
\end{align*}
Thus, the approach $A_\Net A_\Net^\upc$ needs less degrees of freedom than $A_\omega A_\Net^\upc$ to generate systems of the same error magnitude, see Table \ref{tab:553}. On the other hand the approach $A_\Net A_\omega^\upc$ needs a lot more degrees of freedom than $A_\omega A_\omega^\upc$ to yield a comparable error and is $2.7$ times slower.
\begin{figure}[!t]
\includegraphics[scale=1]{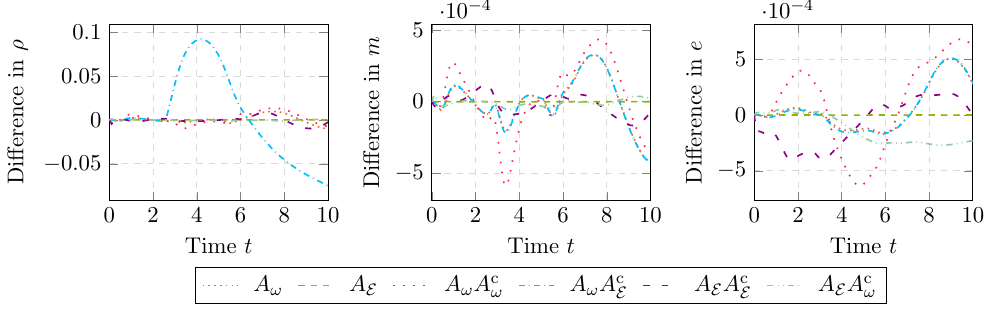}
\caption{Difference of full and reduced variables $\rho^{\omega_2}$, $m^{\omega_2}$ and $e^{\omega_2}$ at $\nu_3$ over time.}\label{fig:Diff_553}
\end{figure}
This table also shows, that using $A_\omega A_\omega^\upc$ has the biggest speed up, i.e., 14, compared to the FOM. This might follow from the carefully choosen configurations for the pipe-wise $r_\uprho$ and $n_\upc$. Furthermore, all order and complexity reduced systems preserve the total energy and mass of the FOM, as can be seen in Figure \ref{fig:EM_553}, and also the pH-structure, as the error $E_{pH}$ is close to machine precision, i.e., condition \eqref{Cond:Volker} is fulfilled, see Table \ref{tab:553}. Lastly, we examine the temporal evolution of the states $\rho^{\omega_2}$, $m^{\omega_2}$ and $e^{\omega_2}$ at the coupling node $\nu_3$, see Figure \ref{fig:Time_553}. As most of the reduced states approximate the FOM quite good to eye precision, the differences with respect to the FOM of all reduced systems is plotted in Figure \ref{fig:Diff_553}. These show that the biggest errors are made in $\rho^{\omega_2}[\nu_3]$, i.e., for $A_\omega A_\omega^\upc$ in $\mathcal{O}(10^{-2})$, for $A_\omega$, $A_\omega A_\omega^\upc$ and $A_\Net A_\Net^\upc$ in $\mathcal{O}(10^{-3})$. Whereas, the error in $\rho^{\omega_2}[\nu_3]$ for $A_\Net A_\omega^\upc$ is in $\mathcal{O}(10^{-4})$. The errors for $m^{\omega_2}[\nu_3]$ and $e^{\omega_2}[\nu_3]$ are in $\mathcal{O}(10^{-4})$. The best error behavior is achieved when $A_\Net$ without complexity reduction is applied, as we have the errors for $\rho^{\omega_2}[\nu_3]$ in $\mathcal{O}(10^{-6})$, for $m^{\omega_2}[\nu_3]$ and $e^{\omega_2}[\nu_2]$ in $\mathcal{O}(10^{-7})$.

\section{Conclusion}\label{Sec:Con}
\noindent Summarizing, we introduced a novel infinite dimensional port-Hamiltonian formulation for the compressible non-isothermal Euler equations. Here, we circumnavigated the difficulties of the GENERIC approach by working only with the Hamiltonian despite considering a thermodynamic system.
We closed the system by introducing the boundary port and backed it with the underlying Stokes-Dirac structure. The boundary port enabled us to incorporate boundary and coupling conditions in a structure-preserving way into our formulation. Furthermore, we preserved the port-Hamiltonian and thermodynamic structure during all stages of approximation. Finally, academic numerical examples supported our theoretical findings. First, we compared model reduction with and without compatibility conditions, which showed that these conditions are necessary to keep the energy dissipation and mass conservation of the full order model. Nevertheless, reduced models without compatibility conditions can also lead to acceptable approximations by chance. A peculiarity of model order reduction applied to the non-isothermal compressible Euler equations is that the parity of the dimension of the reduced order model needs to be the same as of the full dimension. Otherwise, model reduction might lead to bad approximations. This behavior can be explained by the eigenvalue structure of the skew-symmetric and singular interconnection matrix $\mat{J}(\vect{z})$.
Furthermore, we showed that pipe-wise and network-wise model order and complexity reduction have both advantages and disadvantages. Here, network-wise model reduction leads for example to smaller errors than pipe-wise MOR. Pipe-wise complexity reduction gives more freedom in fine tuning the pipe-wise number of preserved finite elements, but also needs more insight in the topology and discretization parameters of the network. We further showed that complexity reduction by empirical quadrature is superior to DEIM, as it preserves the structure and leads to good approximations with fewer degrees of freedom. 
This work acts as a basis for the next steps in research concerning the modeling, simulation and optimization of networks with compressible fluid flow. Besides the application of the procedures introduced here to real-world gas networks, the optimal control of the emerging systems is of importance in the light of a successful energy transition. Here, also the coupling with other port-Hamiltonian network systems, e.g., power and isothermal gas, comes to mind. Finally, the adaption of these results to a port-Hamiltonian formulation of incompressible non-isothermal fluid flow is a very interesting question of research.

\appendix
\section{Coordinate Representations}\label{Ap:CoRep}

\subsection{Full Order Model}\label{App:FOMCoRep}
Let us introduce $\{\psi_j\}_{j=1\dots n_1}$, $\{\varphi_i\}_{i=1\dots n_2}$, $\{\phi_\iota\}_{\iota=1\dots n_3}$ as finite dimensional bases of $\spa{V}_\uprho$, $\spa{V}_\upm$, $\spa{V}_\upe$, respectively. Then we have $\rho_\uph(t,x)=\sum_{j=1}^{n_1}\rho_j(t)\psi_j(x)$ and $m_\uph(t,x)=\sum_{i=1}^{n_2}m_i(t)\varphi_i(x)$ as well as $e_\uph(t,x)=\sum_{\iota=1}^{n_3}e_\iota(t)\phi_\iota(x)$ and $1_\uph=\sum_{\iota=1}^{n_3}\phi_\iota(x)$. We collect the time dependent coefficients in the vectors $\brho=[\rho_1\quad\dots\quad\rho_{n_1}]^T$, $\bm=[m_1\quad\dots\quad m_{n_2}]^T$, $\be=[e_1\quad \dots\quad e_{n_3}]^T$. 
\begin{coordinate}\label{Sys:CoRep}
Let $z_\uph=[\rho_h\quad m_\uph\quad e_\uph]^T$. Then the coordinate representation of the semi-discrete system under Assumption \ref{as:IdealGas} is given as,
\begin{align*}
\mat{M}_{\uprho}\dot{\brho}&=\mat{J}_{\uprho,\upm}\bm,\\
\mat{M}_{\upm,\uprho}(z_\uph)\dot{\brho}+\mat{M}_{\upm,\upm}(z_\uph)\dot{\bm}&=-\mat{J}_{\uprho,\upm}^T\boldsymbol{\varepsilon}(z_\uph)+(\mat{J}_{\upm,\upe}(z_\uph)+\JRM(z_\uph))\vect{1} +\mat{T}_\upm\blam_\upm,\\
\mat{M}_\upe\dot{\be}&=(-\mat{J}_{\upm,\upe}^T(z_\uph)-\JRM^T(z_\uph))\bm-\mat{R}_{\upe,\upe}(z_\uph)\vect{1}+\vect{t}_\upe(z_\uph)\lambda_\upe+\vect{b}_\upe\Bflo|_L+\vect{b}u,\\
\vect{0}&=-\mat{T}_\upm^T\bm+\vect{\Bflo},\\
0&=-\vect{t}_\upe^T(z_\uph)\vect{1}+\frac{\eZ}{\BfloZ}\BfloZ,
\end{align*}
with output equation $\vect{y}_{\mathrm{B}}=\vect{\Beff}$ and input $\vect{u}=[T_\infty\quad \BfloL\quad \BfloZ]^T$ and 
\begin{align}\label{eq:FEMSD}
\mat{M}_\uprho&=[(\psi_q,\psi_j)]_{j,q=1\dots n_1},&\mat{J}_{\uprho,\upm}&=[-(\partial_x\varphi_\iota,\psi_j)]_{\substack{j=1\dots n_1\\\iota=1\dots n_2}} ,&\mat{T}_\upm&=[T_\upm\varphi_1\quad\dots\quad T_\upm\varphi_{n_2}]^T,\nonumber\\
\mat{M}_{\upm,\uprho}&=[-(\frac{m_\uph}{\rho_\uph^2}\psi_j,\varphi_i)]_{\substack{i=1\dots n_2,\\ j=1\dots n_1}}, &\mat{J}_{\upm,\upe}&=[(\frac{-e_\uph\partial_x\phi_\iota-\partial_x(p_\uph\phi_\iota)}{\rho_\uph},\varphi_i)]_{\substack{i=1\dots n_2,\\\iota=1\dots n_3}}, &\vect{t}_\upe&=[T_\upe\phi_1\quad\dots\quad T_\upe\phi_{n_3}]^T,\\
\mat{M}_{\upm,\upm}&=[(\frac{1}{\rho_\uph}\varphi_\iota,\varphi_i)]_{i,\iota=1\dots n_2}, &\JRM&=[-(\frac{\lambda}{2d}\frac{m_\uph|m_\uph|}{\rho_\uph^2}\phi_\iota,\varphi_i)]_{\substack{i=1\dots n_2,\\\iota=1\dots n_3}}, &\vect{b}&=[(\frac{k_\omega}{d},\phi_\iota)]_{\iota=1\dots n_3}, \nonumber\\
\mat{M}_\upe&=[(\phi_\iota,\phi_i)]_{i,\iota=1\dots n_3},&\mat{R}_{\upe,\upe}&=[(\frac{k_\omega}{d}T_\uph \phi_i,\phi_\iota)]_{i,\iota=1\dots n_3},&\vect{b}_\upe&=[\frac{e+p}{\rho}|_L\phi_i|_L]_{i=1,\dots,n_3}\nonumber,
\end{align}
and $\boldsymbol{\varepsilon}(z_\uph)=\mat{M}_\uprho^{-1}\vect{f}, \vect{f}=[(\frac{m_\uph^2}{2\rho_\uph^2},\psi_j)]_{j=1\dots n_1}$. Here, we have that $\mat{R}_{\upe,\upe}=\mat{R}_{\upe,\upe}^T\geq 0$. The semi-discretized Hamiltonian is given by $\Ham_\mathrm{Semi}(\vect{z})=\frac{1}{2}\bm^T\mat{M}_{\upm,\upm}(z_\uph)\bm+\vect{1}^T\mat{M}_\upe\be.$
\end{coordinate}
\noindent The term $\vect{y}_\mathrm{B}$ solely denotes the output created by the boundary terms, originating from partial integration. The output $y_\mathrm{P}=\frac{k_\omega}{d}$ belongs to the external port, but has no deeper meaning. Coordinate Representation \ref{Sys:CoRep} is a differential-algebraic equation of differentiation index two, see \cite{Hau24_Dis}.

\subsection{Reduced Order Model}\label{Sec:ROM}
Having computed the projection matrices by Algorithms \ref{Algo:PODW1W3}, \ref{Algo:POD} and \ref{Algo:CompBasis}, we can set up the reduction basis $\mat{V}_\upr$ from Assumption \ref{As:CompAlg}. The reduced model is then given as follows.
\begin{coordinate}\label{eq:TensMOR}
The reduced order model of Coordinate Representation \ref{Sys:CoRep} can be constructed by Galerkin projection as follows,
\begin{align*}
\mat{M}_\uprho^\upr\dot{\brho}_\upr&=\mat{J}^\upr_{\uprho,\upm}\bm_\upr,\nonumber\\
\mat{M}^\upr_{\upm,\uprho}\dot{\brho}_\upr+\mat{M}^\upr_{\upm,\upm}\dot{\bm}_\upr&=-\mat{J}^{\upr\,T}_{\uprho,\upm}\boldsymbol{\varepsilon}_\upr+(\mat{J}^\upr_{\upm,\upe}+\tilde{\mat{J}}^\upr_{\upm,\upe})\vect{1}_\upr+\mat{T}^\upr_m\blam_\upm,\nonumber\\
\mat{M}_\upe^\upr\dot{\be}_\upr&=(-\mat{J}^{\upr\,T}_{\upm,\upe}-\tilde{\mat{J}}^{\upr\,T}_{\upm,\upe})\bm_\upr-\mat{R}_{\upe,\upe}^\upr\vect{1}_\upr+\vect{b}^\upr u+\vect{t}^\upr_\upe(\mat{V}_\upe\be_\upr)\lambda_\upe+\vect{b}^\upr_\upe(\mat{V}_\uprho\brho_\upr,\mat{V}_\upe\be_\upr)\BfloL,\\
\vect{0}&=-(\mat{T}^\upr_\upm)^T\bm_\upr+\vect{\Bflo},\nonumber\\
0&=-(\vect{t}^\upr_\upe(\mat{V}_\upe\be_\upr))^T\vect{1}_\upr+\frac{\eZ}{\BfloZ}\BfloZ,\nonumber
\end{align*}
with
\begin{align*}
\mat{M}_\uprho^\upr&=\V_\uprho^T\mat{M}_\uprho\V_\uprho,&\mat{J}_{\uprho,\upm}^\upr&=\V_\uprho^T\mat{J}_{\uprho,\upm}\V_\upm,&\mat{T}_\upm^\upr&=\V_\upm^T\mat{T}_\upm,\nonumber \\
\mat{M}^\upr_{\upm,\uprho}&=\V_\upm^T\mat{M}_{\upm,\uprho}\V_\uprho, &\mat{J}^\upr_{\upm,\upe}&=\V_\upm^T\mat{J}_{\upm,\upe}\V_\upe,&\vect{t}_\upe^\upr&=\V_\upe^T\vect{t}_\upe,\nonumber\\
\mat{M}^\upr_{\upm,\upm}&=\V_\upm^T\mat{M}_{\upm,\upm}\V_\upm, &\tilde{\mat{J}}^\upr_{\upm,\upe}&=\V_\upm^T\tilde{\mat{J}}_{\upm,\upe}\V_\upe  , &\vect{b}^\upr&=\V_\upe^T\vect{b},\nonumber\\
\mat{M}_\upe^\upr&=\V_\upe^T\mat{M}_\upe \V_\upe,&\mat{R}^\upr_{\upe,\upe}&=\V_\upe^T\mat{R}_{\upe,\upe}\V_\upe,,&\vect{b}_\upe^\upr&=\V_\upe^T\vect{b}_\upe,\nonumber\\
\vect{1}_\upr&=\V_\upe^\dag\vect{1}, &\boldsymbol{\varepsilon}_\upr&=\V_\uprho^\dag\boldsymbol{\varepsilon}  & \V_\mathrm{x}^\dag&=(\V_\mathrm{x}^T\mat{M}_\mathrm{x} \V_\mathrm{x})^{-1}\V_\mathrm{x}^T\mat{M}_\mathrm{x},\nonumber
\end{align*}
with $\mathrm{x}\in\{\uprho,\upe\}$ and the Hamiltonian is given as $\Ham_\mathrm{ROM}(\vect{z}_\upr)=\frac{1}{2}\bm_\upr^T\mat{M}^\upr_{\upm,\upm}\bm_\upr+\vect{1}^T\mat{M}_\upe\mat{V}_\upe\be_\upr.$
\end{coordinate}
\noindent By construction the symmetric positive definiteness of $\mat{R}_{\upe,\upe}^r=(\mat{R}_{\upe,\upe}^r)^T\geq 0$ are preserved. Furthermore, condition \eqref{Cond:Volker} is fulfilled, such that the reduced order model is still port-Hamiltonian, see \cite{Hau24_Dis} for the proof.
\subsection{Complexity Reduced Coordinate Representation}
\begin{coordinate}\label{Sys:EmpQuad}
The complexity reduced system can be deduced from Coordinate Representation \ref{eq:TensMOR},
\begin{align*}
\mat{M}_\uprho^\upr\dot{\brho}_\upr&=\mat{J}^\upr_{\uprho,\upm}\bm_\upr,\nonumber\\
\mat{M}^\upc_{\upm,\uprho}\dot{\brho}_\upr+\mat{M}^\upc_{\upm,\upm}\dot{\bm}_\upr&=\mat{J}^\upr_{\upm,\uprho}\boldsymbol{\varepsilon}_\upc+(\mat{J}^\upc_{\upm,\upe}+\tilde{\mat{J}}^\upc_{\upm,\upe})\vect{1}_\upr+\mat{T}^\upr_m\blam_\upm,\nonumber\\
\mat{M}_\upe^\upr\dot{\be}_\upr&=(\mat{J}^\upc_{\upe,\upm}+\tilde{\mat{J}}_{\upe,\upm}^\upc)\bm_\upr-\mat{R}^\upc_{\upe,\upe}\vect{1}_\upr+\vect{b}^\upr u+\vect{t}^\upr_\upe(\mat{V}_\upe\be_\upr)\lambda_\upe+\vect{b}^\upr_\upe(\mat{V}_\uprho\brho_\upr,\mat{V}_\upe\be_\upr)\BfloL,\nonumber\\
\vect{0}&=-(\mat{T}^\upr_\upm)^T\bm_\upr+\vect{\Bflo},\nonumber\\
0&=-(\vect{t}^\upr_\upe(\mat{V}_\upe\be_\upr))^T\vect{1}_\upr+\frac{\eZ}{\BfloZ}\BfloZ,\nonumber
\end{align*}
with complexity reduced matrices
\begin{align*}
\mat{M}^\upc_{\upm,\uprho}&=[-(\frac{m_\uph}{\rho_\uph^2}\psi_j,\varphi_i)_\upc]_{\substack{i=1\dots n_2,\\ j=1\dots n_1}},   &\mat{M}_{\upm,\upm}^\upc&=[(\frac{1}{\rho_\uph}\varphi_\iota,\varphi_i)_\upc]_{i,\iota=1\dots n_2},\\
\mat{J}_{\upm,\upe}^\upc&=[(\frac{-e_\uph\partial_x\phi_\iota-\partial_x(p_\uph\phi_\iota)_\upc}{\rho_\uph},\varphi_i)]_{\substack{i=1\dots n_2,\\\iota=1\dots n_3}},
& \JRM^\upc&=[-(\frac{\lambda}{2d}\frac{m_\uph|m_\uph|}{\rho_\uph^2}\phi_\iota,\varphi_i)_\upc]_{\substack{i=1\dots n_2,\\\iota=1\dots n_3}},\\ 
\mat{R}_{\upe,\upe}^\upc&=[(\frac{k_\omega}{d}T_\uph \phi_i,\phi_\iota)_\upc]_{i,\iota=1\dots n_3},
\end{align*}
the other matrices are given as Coordinate Representation \ref{eq:TensMOR}.
The complexity reduced Hamiltonian is given by $\Ham_\mathrm{Quad}(\vect{z}_\upr)=\frac{1}{2}\bm_\upr^T\mat{M}^\upc_{\upm,\upm}\bm_\upr+\vect{1}^T\mat{M}_\upe\mat{V}_\upe\be_\upr.$
\end{coordinate}
\subsection{Coordinate Formulation on Network}\label{Sec:CoRepNet}
To set up the coordinate representation of the network system we use the structure-preserving space discretization on each pipe introduced in Chapter \ref{Sec:SpaceDis}and introduce the coupling matrix $\mat{C}$. Without loss of generality we assume that the interior nodes are numbered starting with $l=1$, such that $\mathcal{N}_0=\{\nu_1,\dots,\nu_j\}$. Furthermore, let the pipes adjacent to inner node $\nu\in\mathcal{N}_0$ be numbered in the following way $\mathcal{E}(\nu)=\{\omega_1,\dots,\omega_{k_{\nu}}\}$ with $k_\nu=|\mathcal{E}(\nu)|$.
\begin{definition}\label{Def:CoupMat}
Let $\vect{u}_0=[\Bflo^{\omega_1}|_{\nu_1}\quad \dots\Bflo^{\omega_{k_{\nu_1}}}|_{\nu_1}\quad \dots \quad\Bflo^{\omega_1}|_{\nu_j}\quad \dots\quad\Bflo^{\omega_{k_{\nu_j}}}|_{\nu_j}]^T$, i.e., the boundary flows at all coupling nodes. Then the coupling matrix $\mat{C}\in\mathbb{R}^{|\mathcal{N}_0|\times |\vect{u}_0|}$ is defined by
\begin{align*}
\mat{C}_{ik}=\begin{cases}1 \text{ if }[\vect{u}_0]_k \text{ is defined on } \omega\in\mathcal{E}(\nu_i)\\ 0 \text{ otherwise}\end{cases},
\end{align*}
with $\nu_i\in \mathcal{N}_0$.
\end{definition} 
\begin{coordinate}\label{eq:coupledNetwork}
The finite dimensional port-Hamiltonian system for a coupled network of pipes for $\vect{z}^\mathcal{E}=[(\brho^\mathcal{E})^T\quad (\bm^\mathcal{E})^T\quad (\be^\mathcal{E})^T\quad (\blam_\upm^\mathcal{E})^T\quad (\blam_\upe^\mathcal{E})^T]^T$ is given by
\begin{align*}
\begin{bmatrix}\mat{E}^{\mathcal{E}}(\vect{z}^{\mathcal{E}}) &\sz&\sz\\\sz&\sz&\sz\\\sz&\sz&\sz
\end{bmatrix}\begin{bmatrix}\dot{\vect{z}}^{\mathcal{E}}\\\dot{\vect{u}}_0\\\dot{\blam}_\upH\end{bmatrix}&=\begin{bmatrix}
\mat{J}^{\mathcal{E}}(\vect{z}^{\mathcal{E}})-\mat{R}(\vect{z}^{\mathcal{E}})& \mat{B}_0(\vect{z}^{\mathcal{E}})&\sz\\-\mat{B}_0^T(\vect{z}^{\mathcal{E}}) &\sz & \mat{C}^T\\\sz&-\mat{C}&\sz\end{bmatrix}\begin{bmatrix}\vect{\eff}^\mathcal{E}(\vect{z}^{\mathcal{E}})\\\vect{u}_0\\\blam_\upH\end{bmatrix}\nonumber
+\begin{bmatrix}
\mat{B}_{T_\infty} & \mat{B}_{\partial}(\vect{z}^\mathcal{E})\\\vect{0}&\vect{0}\\\vect{0}&\vect{0}\end{bmatrix}\begin{bmatrix}
\vect{u}_{T_\infty}\\\vect{u}_{\partial}\end{bmatrix},\\
\begin{bmatrix}\vect{y}_{T_\infty}\\\vect{y}_{\partial}\end{bmatrix}&=\begin{bmatrix}
\mat{B}_{T_\infty} & \mat{B}_{\partial}(\vect{z}^\mathcal{E})\\\vect{0}&\vect{0}\\\vect{0}&\vect{0}\end{bmatrix}^T\begin{bmatrix}
\vect{\eff}^{\mathcal{E}}(\vect{z}^{\mathcal{E}})\\\vect{u}_0\\\blam_\upH\end{bmatrix}\nonumber.
\end{align*}
\end{coordinate}
\noindent Given a graph $\mathcal{G}=(\mathcal{N},\mathcal{E},L)$ with $\mathcal{N}=\{\nu_1,\dots,\nu_l\}$, $l\in\mathbb{N}$,  $\mathcal{E}=\{\omega_1,\dots,\omega_k\}\subset\mathcal{N}\times\mathcal{N}$ and \linebreak $L=\{L^{\omega_1},\dots,L^{\omega_k}\}$, the model on each pipe $\omega_i=(\nu,\bar{\nu})$, $i=1,\cdots,k$, $\nu,\bar{\nu}\in\mathcal{N}$, takes the following form,
\begin{align}\label{eq:CoRepPipe}
\begin{split}
\mat{E}^{\omega_i}(\vect{z}^{\omega_i})\dot{\vect{z}}^{\omega_i}&=((\mat{J}^{\omega_i}-\mat{R}^{\omega_i})(\vect{z}^{\omega_i}))\vect{\eff}^{\omega_i}(\vect{z}^{\omega_i})+\mat{B}^{\omega_i}(\vect{z}^{\omega_i},\vect{\Bflo}^{\omega_i})\vect{u}^{\omega_i},\\
\vect{y}^{\omega_i}&=\mat{B}^{\omega_i}(\vect{z}^{\omega_i},\vect{\Bflo}^{\omega_i})^T\vect{\eff}^{\omega_i}(\vect{z}^{\omega_i}),
\end{split}
\end{align}
with
\begin{align}
\label{eq:CoRepPipeMat}
&\vect{z}^{\omega_i}\coloneqq\begin{bmatrix}
\brho^{\omega_i}\\\bm^{\omega_i}\\\be^{\omega_i}\\\blam_\upm^{\omega_i}\\\lambda_\upe^{\omega_i}\end{bmatrix},
\vect{\eff}^{\omega_i}(\vect{z}^{\omega_i})\coloneqq\begin{bmatrix}
\boldsymbol{\varepsilon}^{\omega_i}\\\bm^{\omega_i}\\\be^{\omega_i}\\\blam_\upm^{\omega_i}\\\lambda_\upe^{\omega_i}\end{bmatrix},
\vect{u}^{\omega_i}\coloneqq\begin{bmatrix}
T_\infty^{\omega_i}\\\Bflo^{\omega_i}|_{\bar{\nu}}\\\Bflo^{\omega_i}|_{\nu}
\end{bmatrix},
\mat{E}^{\omega_i}(\vect{z}^{\omega_i})\coloneqq\begin{bmatrix}
\mat{M}^{\omega_i}_\uprho&\sz&\sz&\sz&\vect{0}\\\mat{M}^{\omega_i}_{\upm,\uprho} &\mat{M}_{\upm,\upm}^{\omega_i}&\sz&\sz&\vect{0}\\\sz&\sz&\mat{M}_\upe^{\omega_i}&\sz&\vect{0}\\\sz&\sz&\sz&\sz&\vect{0}\\\vect{0}&\vect{0}&\vect{0}&\vect{0}&0
\end{bmatrix},\\
&(\mat{J}^{\omega_i}-\mat{R}^{\omega_i})(\vect{z}^{\omega_i})\coloneqq\begin{bmatrix}
\sz&\mat{J}^{\omega_i}_{\uprho,\upm}&\sz&\sz&\vect{0}\\\mat{J}^{\omega_i}_{\upm,\uprho}&\sz&\mat{J}^{\omega_i}_{\upm,\upe}+\tilde{\mat{J}}^{\omega_i}_{\upm,\upe}&\mat{T}^{\omega_i}_\upm&\vect{0}\\\sz&\mat{J}^{\omega_i}_{\upe,\upm}+\tilde{\mat{J}}^{\omega_i}_{\upe,\upm}&-\mat{R}^{\omega_i}_{\upe,\upe}&\sz&\vect{t}^{\omega_i}_\upe\\
\sz&-(\mat{T}^{\omega_i}_\upm)^T&\sz&\sz&\vect{0}\\\vect{0}&\vect{0}&-(\vect{t}_\upe^{\omega_i})^T&\vect{0}&0\end{bmatrix},
\mat{B}^{\omega_i}(\vect{z}^{\omega_i},\vect{\Bflo}^{\omega_i})\coloneqq \begin{bmatrix}
\vect{0}&\vect{0}&\vect{0}\\\vect{0}&\vect{0}&\vect{0}\\\vect{b}^{\omega_i}&\vect{0}&\vect{b}^{\omega_i}_\upe\\0&1&0\\0&0&1\\0&0&\frac{e^{\omega_i}|_{\nu}}{\Bflo^{\omega_i}|_{\nu}}
\end{bmatrix},\nonumber
\end{align}
\noindent The state-dependent matrices, boundary-operators and $\boldsymbol{\varepsilon}^{\omega_i}$  are given as in \eqref{eq:FEMSD}. The following coupling procedure is inspired by ideas from \cite{DomHLMMT21}. With respect to the $\spa{L}^2(\mathcal{E})$ inner-product, the matrices and tensors of the same kind are collected in a block-diagonal manner. Therefore, we have for example,
\begin{align*}
\mat{M}_\uprho^\mathcal{E}=\begin{bmatrix}\mat{M}_\uprho^{\omega_1}&\sz&\cdots&\sz\\
\sz&\mat{M}_\uprho^{\omega_2}&\sz&\vdots\\\vdots&&\ddots&\sz\\\sz&\cdots&\sz&\mat{M}_\uprho^{\omega_k}\end{bmatrix}.
\end{align*}
Thus, the state and effort vectors are created by stacking the respective pipe-vectors on top of each other. Doing this leads to a big system of the form of \eqref{eq:CoRepPipe} with \eqref{eq:CoRepPipeMat}, where the superscript $\omega_i$ is exchanged for $\mathcal{E}$. However, the created system still lacks the coupling conditions, which need to be incorporated next. As $(\upS_{out})$ can be transformed into a condition for $e^{\omega_i}|_\nu$, $\nu\in~\mathcal{N}_0$, if $\omega_i=(\nu,\bar{\nu})$, it is simply plugged into $\mat{B}^{\omega_i}(\vect{z}^{\omega_i},\vect{\Bflo}^{\omega_i})$. For implementing $(\upM)$ and $(\upH)$, i.e., \eqref{Coup:Mass} and \eqref{Coup:Enthalpy}, we rearrange the network input $\vect{u}^\mathcal{E}$, such that we have
$\vect{u}^\mathcal{E}=[\vect{u}_{T_\infty}^T\quad\vect{u}_0^T\quad\vect{u}_\partial^T]^T$. Here, $\vect{u}_{T_\infty}$ collects the ambient temperatures of each pipe, $\vect{u}_0$ the boundary flow variables at all coupling nodes $\nu_0\in\mathcal{N}_0$ and $\vect{u}_\partial$ the boundary flow variables at all boundary nodes $\nu_\partial\in\mathcal{N}_\partial$. Accordingly, we rearrange $\mat{B}^\mathcal{E}$ into $\mat{B}^\mathcal{E}=[\mat{B}_{T_\infty}\quad\mat{B}_0\quad\mat{B}_\partial]$. 
We can then express the mass conservation with the help of Definition \ref{Def:CoupMat} as $\mat{C}\vect{u}_0=\vect{0}$, as the multiplication of each row of $\mat{C}$ with $\vect{u}_0$, i.e., $\mat{C}_{i,:}\vect{u}_0$, sums up the boundary flow variables at the coupling node $\nu_i$, $i=1,\dots, |\mathcal{N}_0|$.
For $(\upH)$, i.e., the enthalpy equality at the coupling nodes, we need to make use of the boundary effort variables at the interior nodes, as shown in \eqref{Coup:Enthalpy}. From Section \ref{Sec:Ports} we know that these are given by the output equation. For the network output we have,
\begin{align*}
\vect{y}^\mathcal{E}=\begin{bmatrix}\vect{y}_{T_\infty}^T&\vect{y}^T_0&\vect{y}^T_\partial\end{bmatrix}^T=\begin{bmatrix}\mat{B}_{T_\infty}&\mat{B}_0&\mat{B}_\partial\end{bmatrix}^T\vect{\eff}^{\mathcal{E}}(\vect{z}^\mathcal{E}).
\end{align*}
Extracting the output equation for the interior nodes yields $\vect{y}_0=\mat{B}_0^T\vect{\eff}^{\mathcal{E}}(\vect{z}^\mathcal{E})$, which gives us the enthalpy at each interior node $\nu\in\mathcal{N}_0$. We collect the Lagrange multipliers from \eqref{Coup:Enthalpy} into a vector $\blam_\upH\in\mathbb{R}^{|\mathcal{N}_0|}$. To make sure that we have enthalpy equality at every coupling node, we set $\vect{y}_0=\mat{C}^T\blam_\upH$, i.e., we have that $\mat{C}^T\blam_\upH=\mat{B}_0^T\vect{\eff}^{\mathcal{E}}(\vect{z}^\mathcal{E})$. So, mass conservation and enthalpy equality on the interior nodes are given by 
\begin{align*}
\mat{C}\vect{u}_0=\vect{0},\quad \text{and}\quad \mat{C}^T\blam_\upH-\mat{B}_{0}^T\vect{\eff}^{\mathcal{E}}(\vect{z}^\mathcal{E})=\vect{0}.
\end{align*}
Adding these equations to the block-diagonally ordered system, yields the coordinate system for the coupled network system. By construction, the system is still port-Hamiltonian. The system matrix on the right side is skew-symmetric disregarding the symmetric positive semi-definite matrix $\mat{R}(\vect{z}^\mathcal{E})$. Furthermore, the condition $(\mat{E}^\mathcal{E}(\vect{z}^\mathcal{E}))^T\vect{\eff}^{\mathcal{E}}(\vect{z}^\mathcal{E})=\nabla_{\vect{z}^\mathcal{E}}\Ham^\mathcal{E}(\vect{z}^\mathcal{E})$ is fulfilled, see \cite{Hau24_Dis}. The Coordinate Representation \ref{eq:coupledNetwork} is a port-Hamiltonian differential-algebraic equation of differentiation index 2, see \cite{Hau24_Dis} for the proof.

\section{Derivations and Proofs}\label{Appendix2}
\subsection{Derivation of Boundary Port Variables \eqref{eq:BoundaryPortM}}
For the port-Hamiltonian formulation of System \ref{Sys:pHDCv} the boundary effort and flow variables can be stated in the following way
\begin{align*}
\vect{\Bflo}=[\BfloL\quad \BfloZ]^T, \quad  \vect{\Beff}=[\BeffL\quad \BeffZ]^T
\end{align*}
with
\begin{align*}
\begin{bmatrix}\BfloL\\\BeffL\end{bmatrix}=\begin{bmatrix}0 & -1 &0\\ 1&0& \frac{e+p}{\rho}|_L\end{bmatrix}\left.\begin{bmatrix}
\frac{v^2}{2}\\\rho v\\1\end{bmatrix}\right|_L, \quad \begin{bmatrix}\BfloZ\\\BeffZ\end{bmatrix}=\begin{bmatrix}0 &1 &0\\ 1&0& \frac{e+p}{\rho}|_0\end{bmatrix}\left.\begin{bmatrix}
\frac{v^2}{2}\\\rho v\\1\end{bmatrix}\right|_0.
\end{align*}
The papers \cite{LeGZM05} and \cite{Vil17_Dis} deduce the boundary port variables elegantly from the skew-adjoint operator $J$ by building a matrix $Q$ from which they derive different parametrizations of $\vect{\Beff}$ and $\vect{\Bflo}$. The first paper only considers non-singular state-independent $Q$ matrices, but \cite{Vil17_Dis} expands this approach to singular systems, which are still state-independent. We adapt the ansatz from \cite{Vil17_Dis} and, as our systems are highly state-dependent, we tackle this difficulty. In the following, we focus on the operator $J(z)$, defined in System \ref{Sys:pHDCv}, as this operator contains the differential terms, which are responsible for the boundary terms. Let $\mathbf{e}_1,\mathbf{e}_2\in\spa{H}^1(\omega)^3$. Imitating the steps in \cite{Vil17_Dis}, we get that,
\begin{align*}
(\mathbf{e}_1,J(z)\mathbf{e}_2)+( J(z)\mathbf{e}_1,\mathbf{e}_2)
&=\left.\left[\mathbf{e}_1^T\underbrace{\begin{bmatrix}0 &-1&0\\-1&0&-\frac{e+p}{\rho}\\ 0 & -\frac{e+p}{\rho}&0\end{bmatrix}}_{Q(z)}\mathbf{e}_2\right]\right|_0^L=\begin{bmatrix}\mathbf{e}_1|_L\\\mathbf{e}_1|_0\end{bmatrix}^T\begin{bmatrix}Q(z)|_L&0\\0&-Q(z)|_0\end{bmatrix}\begin{bmatrix}
\mathbf{e}_2|_L\\\mathbf{e}_2|_0\end{bmatrix}=(*).
\end{align*}
The aim of \cite{Vil17_Dis} is to find a splitting, such that the bilinear form above is no longer dependent on the state in the boundary terms, i.e., the state-dependency is shifted into the boundary flow and effort variables, and can be written as,
\begin{align*}
(*)&=\left(\begin{bmatrix}M_Q|_L&0\\0&M_Q|_0\end{bmatrix}\begin{bmatrix}\mathbf{e}_1|_L\\\mathbf{e}_1|_0\end{bmatrix}\right)^T\begin{bmatrix}
\tilde{Q}(z)|_L&0\\0&-\tilde{Q}(z)|_0\end{bmatrix}\begin{bmatrix}M_Q|_L&0\\0&M_Q|_0\end{bmatrix}\begin{bmatrix}
\mathbf{e}_2|_L\\\mathbf{e}_2|_0\end{bmatrix}\stackrel{!}{=}\begin{bmatrix}f_{\partial,\mathbf{e}_1}\\e_{\partial,\mathbf{e}_1}\end{bmatrix}^T\boldsymbol{\Sigma}\begin{bmatrix}f_{\partial,\mathbf{e}_2}\\e_{\partial,\mathbf{e}_2}\end{bmatrix}.
\end{align*}
with $\boldsymbol{\Sigma}=\begin{bmatrix} \sz & \so\\\so &\sz\end{bmatrix}$. Here, $\so$ denotes the identity matrix and $\sz$ the zero matrix of suitable dimension.
The matrices $M_Q$ and $\tilde{Q}$ need to be deduced from $Q(z)$, i.e.,
\begin{align*}
Q(z)=\begin{bmatrix}0 &-1&0\\-1&0&-\frac{e+p}{\rho}\\ 0 & -\frac{e+p}{\rho}&0\end{bmatrix}
\end{align*}
and the boundary flow and efforts are defined as, see \cite{Vil17_Dis},
\begin{align*}
\begin{bmatrix}f_{\partial,\mathbf{e}}\\e_{\partial,\mathbf{e}}\end{bmatrix}=R_{ext}\begin{bmatrix}M_Q|_L&0\\0&M_Q|_0\end{bmatrix}\begin{bmatrix}
\SeffL\\\SeffZ\end{bmatrix},
\end{align*}
i.e., they are deduced from the storage effort variables evaluated at the boundary. $R_{ext}$ has to fulfill,
\begin{align*}
R_{ext}^T\boldsymbol{\Sigma} R_{ext}=\begin{bmatrix}\tilde{Q}(z)|_L&0\\0&-\tilde{Q}(z)|_0
\end{bmatrix}.
\end{align*}
To compute the matrix $M_Q$, we need to choose a basis for $\operatorname{image}(Q(z))$, e.g, $$\operatorname{image}(Q(z))=\mathrm{span}\left\{ \begin{bmatrix}
0\\1\\0\end{bmatrix}, \begin{bmatrix}1\\0\\\frac{e+p}{\rho}\end{bmatrix}\right\}.$$
\noindent This yields the matrices $M$, $\tilde{Q}$ and $M_Q$,
\begin{align*}
M&=\begin{bmatrix}0&1\\1&0\\0&\frac{e+p}{\rho}\end{bmatrix},\quad \tilde{Q}=M^TQM=\begin{bmatrix}
0 & -\frac{\rho^2+(e+p)^2}{\rho^2}\\ -\frac{\rho^2+(e+p)^2}{\rho^2}&0\end{bmatrix},\\
M_Q&=(M^TM)^{-1}M^T=\begin{bmatrix}
0&1&0\\\frac{\rho^2}{\rho^2+(e+p)^2}&0&\frac{(e+p)\rho}{\rho^2+(e+p)^2}
\end{bmatrix}.
\end{align*}
Of course, these matrices are all state-dependent, but in favor of a shorter notation we suppress the dependency on $z$. We want to create the boundary port variables \eqref{eq:BoundaryPort}. For this we need that
\begin{align*}
\begin{bmatrix}f_{\partial,\mathbf{e}}\\e_{\partial,\mathbf{e}}
\end{bmatrix}\stackrel{!}{=}\begin{bmatrix}\BfloL\\\BfloZ\\\BeffL\\\BeffZ\end{bmatrix}=\underbrace{\begin{bmatrix}
0&-1&0&0&0&0\\0&0&0&0&1&0\\1&0&\left.\frac{e+p}{\rho}\right|_L&0&0&0\\0&0&0&1&0&\left.\frac{e+p}{\rho}\right|_0\end{bmatrix}}_{=R_{ext}\begin{bmatrix}M_Q|_L&0\\0&M_Q|_0\end{bmatrix}}\begin{bmatrix}
\SeffaL\\\SeffbL\\\SeffcL\\\SeffaZ\\\SeffbZ\\\SeffcZ
\end{bmatrix},
\end{align*}
Thus, the matrix $R_{ext}$ is given as,
\begin{align*}
R_{ext}=\begin{bmatrix}
-1 & 0&0&0\\0&0&1&0\\0&\left.\frac{\rho^2+(e+p)^2}{\rho^2}\right|_L&0&0\\0&0&0&\left.\frac{\rho^2+(e+p)^2}{\rho^2}\right|_0
\end{bmatrix}.
\end{align*}
This matrix also fulfills
\begin{align*}
R_{ext}^T\boldsymbol{\Sigma} R_{ext}&=\begin{bmatrix}
-1 & 0&0&0\\0&0& \left.\frac{\rho^2+(e+p)^2}{\rho^2}\right|_L    &0\\0&1&0&0\\0&0&0&\left.\frac{\rho^2+(e+p)^2}{\rho^2}\right|_0
\end{bmatrix}\begin{bmatrix}0&0&1&0\\0&0&0&1\\1&0&0&0\\0&1&0&0\end{bmatrix}\begin{bmatrix}
-1 & 0&0&0\\0&0&1&0\\0&\left.\frac{\rho^2+(e+p)^2}{\rho^2}\right|_L&0&0\\0&0&0&\left.\frac{\rho^2+(e+p)^2}{\rho^2}\right|_0
\end{bmatrix}\\
&=\begin{bmatrix}
0&-\left.\frac{\rho^2+(e+p)^2}{\rho^2}\right|_L&0&0\\-\left.\frac{\rho^2+(e+p)^2}{\rho^2}\right|_L&0&0&0\\0&0&0&\left.\frac{\rho^2+(e+p)^2}{\rho^2}\right|_0\\0&0&\left.\frac{\rho^2+(e+p)^2}{\rho^2}\right|_0&0\end{bmatrix}=\begin{bmatrix}\tilde{Q}(z)|_L&0\\0&-\tilde{Q}(z)|_0
\end{bmatrix}.
\end{align*}
Thus, the boundary ports \eqref{eq:BoundaryPort} are in line with \cite{LeGZM05} and \cite{Vil17_Dis}, considering the additional state-dependency. 

\subsection{Proof of Theorem \ref{Theo:DiracpHDC}}
Let $\spa{F}=\spa{H}^1(\omega)^3\times\spa{L}^2(\omega)^3\times\spa{L}^2(\omega) \times\mathbb{R}^2$ and $\spa{E}=\spa{H}^1(\omega)^3\times\spa{L}^2(\omega)^3\times\spa{L}^2(\omega)\times\mathbb{R}^2$. Then the underlying Stokes-Dirac structure is given by the linear subset $\spa{D}\subset\spa{F}\times\spa{E}$, 
\begin{align*}
\begin{split}
\spa{D}=&\Bigl\{[[f_\omega,\vect{\Bflo}],[e_\omega,\vect{\Beff}]]\in\spa{F}\times\spa{E}|\,
[p\Seffc,\,\frac{e}{\rho}\Seffb,\,\frac{1}{\rho}\Seffb]\in\spa{H}^1(\omega)^3\\
&\quad\begin{bmatrix}\Sflo\\\Rflo\\\Pflo\end{bmatrix}+\begin{bmatrix}
J(z)&\so&B\\-\so&\sz&\vect{0}\\-B^T&\vect{0}&0
\end{bmatrix}\begin{bmatrix}\Seff\\\Reff\\\Peff\end{bmatrix}=\begin{bmatrix}
\vect{0}\\\vect{0}\\0\end{bmatrix},\\
&\quad\begin{bmatrix}\BfloL\\\BeffL\end{bmatrix}=\begin{bmatrix}0 & -1 &0\\ 1&0& \frac{e+p}{\rho}|_L\end{bmatrix}\SeffL,\quad  \begin{bmatrix}\BfloZ\\\BeffZ\end{bmatrix}=\begin{bmatrix}0 &1 &0\\ 1&0& \frac{e+p}{\rho}|_0\end{bmatrix}\SeffZ\Bigr\}.
\end{split}
\end{align*}
Furthermore, the system of equations
\begin{align*}
\begin{split}
\Sflo&=-E(z)\partial_tz, \quad \Seff=\frac{\delta\Ham}{\delta z}(z),\quad \Reff=-R(z)\Rflo, \quad \Peff=T_\infty,\\&\quad [[f_\omega,\vect{\Bflo}],[e_\omega,\vect{\Beff}]]\in\spa{D},
\end{split}
\end{align*}
is equivalent to the original System \ref{Sys:pHDCv}, and 
$\langle[f_\omega,\vect{\Bflo}],[e_\omega,\vect{\Beff}]\rangle=0$ represents the power-balance equation.
\begin{proof}
\noindent To prove that \eqref{eq:StokesDirac} is a Stokes-Dirac structure as defined in Definition \ref{def:StokesDirac}, we have to show that $\spa{D}=\spa{D}^\perp$. We do this analogously to \cite{BanSAZISW,BanZISW21}, by showing that (i) $\spa{D}\subset\spa{D}^\perp$ and (ii) $\spa{D}^\perp\subset\spa{D}$ using the bilinear form \eqref{eq:Bilin}. For part (i) of the proof, we begin with considering two pairs of flow and effort variables belonging to the Dirac structure, i.e., $[[f_\omega,\vect{\Bflo}],[e_\omega,\vect{\Beff}]]\in\spa{D}$ and $[[\tilde{f}_\omega,\vect{\tBflo}],[\tilde{e}_\omega,\vect{\tBeff}]]\in\spa{D}$. Using \eqref{eq:PowerProduct} the bilinear form \eqref{eq:Bilin} yields,
\begin{align*}
&\langle\langle[f_\omega,\vect{\Bflo},e_\omega,\vect{\Beff}],[\tilde{f}_\omega,\vect{\tBflo},\tilde{e}_\omega,\vect{\tBeff}]\rangle\rangle=\langle [f_\omega,\vect{\Bflo}],[\tilde{e}_\omega,\vect{\tBeff}]\rangle+\langle [e_\omega,\vect{\Beff}],[\tilde{f}_\omega,\vect{\tBflo}]\rangle\\
&\quad=\int_\omega \Sflo\cdot\tSeff+\Rflo\cdot\tReff+\Pflo\cdot\tPeff\D x +\vect{\Bflo}\cdot\vect{\tBeff} +\int_\omega \Seff\cdot\tSflo+\Reff\cdot\tRflo+\Peff\cdot\tPflo\D x +\vect{\Beff}\cdot\vect{\tBflo}=(*).
\end{align*}
Using \eqref{eq:StokesDirac}, we get
\begin{align*}
(*)=&\int_\omega -(J(z)\Seff)^T\tSeff-(B\Peff)^T\tSeff-(\so\Reff)^T\tSeff+(\so\Seff)^T\tReff
+(B^T\Seff)^T\tPeff\D x+\vect{\Bflo}\cdot\vect{\tBeff}\\
&\quad + \int_\omega -\Seff^TJ(z)\tSeff-\Seff^TB\tPeff-\Seff^T\so\tReff+\Reff
^T\so\tSeff+\Peff^TB^T\tSeff\D x+\vect{\Beff}\cdot\vect{\tBflo}\\
&=\int_\omega -\Seff^TJ(z)^T\tSeff\D x +\vect{\Bflo}\cdot\vect{\tBeff} +\int_\omega -\Seff^TJ(z)\tSeff\D x+\vect{\Beff}\cdot\vect{\tBflo}.
\end{align*}
Be aware that to show that the above is equal to zero, Corollary \ref{Cor:JRProp} cannot be applied, since the boundary terms do not vanish here. Thus, plugging in the definition of the system operator $J(z)$ from System \ref{Sys:pHDCv} and performing integration by parts in some terms shows that 
\begin{align*}
\langle\langle[f_\omega,\vect{\Bflo},e_\omega,\vect{\Beff}],[\tilde{f}_\omega,\vect{\tBflo},\tilde{e}_\omega,\vect{\tBeff}]\rangle\rangle=0,
\end{align*}
i.e., we utilize the skew-adjointness of $J(z)$ in the $\spa{L}^2$ inner-product with respect to the boundary terms, and hence we have that $\spa{D}\subset\spa{D}^\perp$, which concludes part (i) of the proof.\\
For the second part (ii) of the proof, we choose $[[\tilde{f}_\omega,\vect{\tBflo}],[\tilde{e}_\omega,\vect{\tBeff}]]\in\spa{D}^\perp$ arbitrarily, such that 
\begin{align}\label{eq:ProofSD1}
\langle\langle[f_\omega,\vect{\Bflo},e_\omega,\vect{\Beff}],[\tilde{f}_\omega,\vect{\tBflo},\tilde{e}_\omega,\vect{\tBeff}]\rangle\rangle=0,
\end{align} 
for all $[[f_\omega,\vect{\Bflo}],[e_\omega,\vect{\Beff}]]\in\spa{D}$. We show that $[[\tilde{f}_\omega,\vect{\tBflo}],[\tilde{e}_\omega,\vect{\tBeff}]]\in\spa{D}^\perp$ fulfills the definitions of the flow and effort variables in \eqref{eq:StokesDirac} and thus, $[[\tilde{f}_\omega,\vect{\tBflo}],[\tilde{e}_\omega,\vect{\tBeff}]]\in\spa{D}$. For this, we take the freedom to choose the effort variables $e_\omega$ for $[[f_\omega,\vect{\Bflo}],[e_\omega,\vect{\Beff}]]\in\spa{D}$. Since for an arbitrary choice of $[[\tilde{f}_\omega,\vect{\tBflo}],[\tilde{e}_\omega,\vect{\tBeff}]]\in\spa{D}^\perp$, \eqref{eq:ProofSD1} is fulfilled for all $[[f_\omega,\vect{\Bflo}],[e_\omega,\vect{\Beff}]]\in\spa{D}$, we can make a certain choice for $[[f_\omega,\vect{\Bflo}],[e_\omega,\vect{\Beff}]]\in\spa{D}$ and \eqref{eq:ProofSD1} is still fulfilled. As $[[f_\omega,\vect{\Bflo}],[e_\omega,\vect{\Beff}]]\in\spa{D}$ the flow variable $f_\omega$ and the boundary flow and effort variables $\vect{\Bflo}$ and $\vect{\Beff}$ are defined through \eqref{eq:StokesDirac} and can be inserted into the bilinear form, i.e.,
\begin{align*}
(*)&=\int_\omega -(J(z)\Seff)^T\tSeff-(B\Peff)^T\tSeff-(\so\Reff)^T\tSeff+(\so\Seff)^T\tReff
+(B^T\Seff)^T\tPeff\D x+\vect{\Bflo}\cdot\vect{\tBeff}\\
&\quad +\int_\omega \Seff\cdot\tSflo+\Reff\cdot\tRflo+\Peff\cdot\tPflo\D x +\vect{\Beff}\cdot\vect{\tBflo}.
\end{align*}
\noindent In the following we use that $\Seff=[\Seffa\quad\Seffb\quad\Seffc]^T$. This splitting can be analogously done for the $\tilde{}$ and all other three dimensional quantities. Using the definition of the $J(z)$ and $B$ operators from System \ref{Sys:pHDCv} in the above equation we get,
\begin{align*}\label{eq:ProofSD2}
(*)&=\int_\omega \tSeffb\partial_x\Seffa+\tSeffa\partial_x\Seffb+\tSeffc\left(-\frac{\lambda}{2d}v|v|\Seffb+\partial_x(\frac{e}{\rho}\Seffb)+p\partial_x(\frac{1}{\rho}\Seffb)\right)\\
&\quad +\tSeffb\left(\frac{\lambda}{2d}v|v|\Seffc+\frac{e}{\rho}\partial_x\Seffc+\frac{1}{\rho}\partial_x(p\Seffc)\right)-\frac{k_\omega}{d}\Peff\tSeffc-(\so_l\Reff)^T\tSeff\\
&\quad+(\so_l\Seff)^T\tReff
+\frac{k_\omega}{d}\Seffc\tPeff\D x+\vect{\Bflo}\cdot\vect{\tBeff} +\int_\omega \Seff\cdot\tSflo+\Reff\cdot\tRflo+\Peff\cdot\tPflo\D x +\vect{\Beff}\cdot\vect{\tBflo}.
\end{align*}
The proof now consists of the following nine similar steps and uses partial integration.
\begin{enumerate}
\item  Let $[[f_\omega,\vect{\Bflo}],[e_\omega,\vect{\Beff}]]\in\spa{D}$ with $\Seffb=\Seffc=\Peff=0$, $\Reff=\vect{0}$ and $\SeffaZ=\SeffaL=0$. Plugging this into the bilinear form, we are left with
\begin{align*}
0=\int_\omega \tSeffb\partial_x\Seffa +\Seffa\tReffa\D x +\int_\omega \Seffa \tSfloa\D x=\int_\omega -\Seffa\partial_x\tSeffb+\Seffa\tReffa\D x +\int_\omega \Seffa \tSfloa\D x
\end{align*}
and thus, by using partial integration in the first term we have that $\tSfloa=\partial_x\tSeffb-~\tReffa$ with $\tSeffb\in\spa{H}^1(\omega)$. 
\item Let $[[f_\omega,\vect{\Bflo}],[e_\omega,\vect{\Beff}]]\in\spa{D}$ with $\Seffa=\Seffc=\Peff=0$, $\Reff=\vect{0}$ and $\SeffbZ=\SeffbL=0$. Plugging this into the bilinear form  we are left with
\begin{align*}
0&=\int_\omega\tSeffa\partial_x\Seffb + \tSeffc (-\frac{\lambda}{2d}v|v|\Seffb+\partial_x(\frac{e}{\rho}\Seffb)+p\partial_x(\frac{1}{\rho}\Seffb))+\Seffb\tReffb\D x+\int_\omega \Seffb\tSflob \D x.
\end{align*}
By using partial integration in the terms where $\Seffb$ is in a partial derivative we get that
$\tSflob=\partial_x\tSeffa+\frac{\lambda}{2d}v|v|\tSeffc+\frac{e}{\rho}\partial_x\tSeffc+\frac{1}{\rho}\partial_x(p\tSeffc)-\tReffb$ with $\tSeffa,\,\tSeffc,\,p\tSeffc\in\spa{H}^1(\omega)$.
\item Let $[[f_\omega,\vect{\Bflo}],[e_\omega,\vect{\Beff}]]\in\spa{D}$ with $\Seffa=\Seffb=\Peff=0$, $\Reff=\vect{0}$ and $\SeffcZ=\SeffcL=0$. Plugging this into the bilinear form  we are left with
\begin{align*}
0&=\int_\omega \tSeffb (\frac{\lambda}{2d}v|v|\Seffc+\frac{e}{\rho}\partial_x\Seffc+\frac{1}{\rho}\partial_x(p\Seffc))+\Seffc\tReffc+\Seffc\frac{k_\omega}{d}\tPeff\D x+\int_\omega \Seffc\tSfloc\D x.
\end{align*}
By using partial integration in the terms where $\Seffc$ is in a partial derivative we get that
$\tSfloc=-\frac{\lambda}{2d}v|v|\tSeffb+\partial_x(\frac{e}{\rho}\tSeffb)+p\partial_x(\frac{1}{\rho}\tSeffb)-\tReffc-\frac{k_\omega}{d}\tPeff$ with $\frac{e}{\rho}\tSeffb,\,\frac{1}{\rho}\tSeffb\in\spa{H}^1(\omega)$.
\item Let $[[f_\omega,\vect{\Bflo}],[e_\omega,\vect{\Beff}]]\in\spa{D}$ with $\Seff=\vect{0}$, $\Peff=0$. Plugging this into the bilinear form  we are left with
\begin{align*}
0&=\int_\omega \Reff\cdot\tRflo\D x+\int_\omega -(\so\Reff)^T\tSeff\D x,
\end{align*}
i.e., we get $\tRflo=\so\tSeff$.
\item Let $[[f_\omega,\vect{\Bflo}],[e_\omega,\vect{\Beff}]]\in\spa{D}$ with $\Seff=\Reff=\vect{0}$. Plugging this into the bilinear form  we are left with
\begin{align*}
0=\int_\omega \Peff\cdot\tPflo\D x+\int_\omega -\frac{k_\omega}{d}\Peff\tSeffc\D x
\end{align*}
and therefore, it is $\tPflo=\frac{k_\omega}{d}\tSeffc=B^T\tSeff$.
\item Let $[[f_\omega,\vect{\Bflo}],[e_\omega,\vect{\Beff}]]\in\spa{D}$ with $\Seffb=\Seffc=\Peff=0$, $\Reff=\vect{0}$ and $\SeffaZ=0$. Plugging this into the bilinear form and using step 1.  we are left with
\begin{align*}
0&=\int_\omega \tSeffb\partial_x\Seffa \D x +\int_\omega \Seffa\partial_x\tSeffb\D x+\tBflo|_L\SeffaL=\tBflo|_L\SeffaL+\tSeffb|_L\SeffaL,
\end{align*}
and from this it follows that $\tBflo|_L=-\tSeffb|_L$.
\item Let $[[f_\omega,\vect{\Bflo}],[e_\omega,\vect{\Beff}]]\in\spa{D}$ with $\Seffb=\Seffc=\Peff=0$, $\Reff=\vect{0}$ and $\SeffaL=0$. Plugging this into the bilinear form and using step 1.  we are left with
\begin{align*}
0&=\int_\omega \tSeffb\partial_x\Seffa \D x +\int_\omega \Seffa\partial_x\tSeffb\D x+\tBflo|_0\SeffaZ=-\tBflo|_0\SeffaZ+\tSeffb|_0\SeffaZ,
\end{align*}
and from this it follows that $\tBflo|_0=\tSeffb|_0$.
\item Let $[[f_\omega,\vect{\Bflo}],[e_\omega,\vect{\Beff}]]\in\spa{D}$ with $\Seffb=\Seffc=\Peff=0$, $\Reff=\vect{0}$ and $\SeffaL=0$. Plugging this into the bilinear form and using step 2.  we are left with
\begin{align*}
0&=\int_\omega \tSeffa\partial_x\Seffb+\tSeffc(-\frac{\lambda}{2d}v|v|\Seffb+\partial_x(\frac{e}{\rho}\Seffb)+p\partial_x(\frac{1}{\rho}\Seffb))\\
&\quad +\Seffb(\partial_x\tSeffa+\frac{\lambda}{2d}v|v|\tSeffc+\frac{e}{\rho}\partial_x\tSeffc+\frac{1}{\rho}\partial_x(p\tSeffc))\D x+\BfloZ\tBeff|_0\\
&=-\SeffbZ(\tSeffa+\frac{e+p}{\rho}\tSeffc)|_0+\SeffbZ\tBeff|_0,
\end{align*}
i.e., $\tBeff|_0=\tSeffa|_0+\frac{e+p}{\rho}|_0\tSeffc|_0$.
\item Let $[[f_\omega,\vect{\Bflo}],[e_\omega,\vect{\Beff}]]\in\spa{D}$ with $\Seffb=\Seffc=\Peff=0$, $\Reff=\vect{0}$ and $\SeffaZ=0$. Plugging this into the bilinear form and using step 2.  we are left with
\begin{align*}
0&=\int_\omega \tSeffa\partial_x\Seffb+\tSeffc(-\frac{\lambda}{2d}v|v|\Seffb+\partial_x(\frac{e}{\rho}\Seffb)+p\partial_x(\frac{1}{\rho}\Seffb))\\
&\quad +\Seffb(\partial_x\tSeffa+\frac{\lambda}{2d}v|v|\tSeffc+\frac{e}{\rho}\partial_x\tSeffc+\frac{1}{\rho}\partial_x(p\tSeffc))\D x+\BfloL\tBeff|_L\\
&=\SeffbL(\tSeffa+\frac{e+p}{\rho}\tSeffc)|_L-\SeffbL\tBeff|_L,
\end{align*}
i.e., $\tBeff|_L=\tSeffa|_L+\frac{e+p}{\rho}|_L\tSeffc|_L$.
\end{enumerate}
This shows that $[[\tilde{f}_\omega,\vect{\tBflo}],[\tilde{e}_\omega,\vect{\tBeff}]]\in\spa{D}$ if $[[\tilde{f}_\omega,\vect{\tBflo}],[\tilde{e}_\omega,\vect{\tBeff}]]\in\spa{D}^\perp$ and therefore, $\spa{D}^\perp\subset\spa{D}$. This finally proves that the set $\spa{D}$ in \eqref{eq:StokesDirac} describes a Stokes-Dirac structure as defined in Definition \ref{def:StokesDirac}.\\
To end the proof of Theorem \ref{Theo:DiracpHDC} we need to show that \eqref{eq:StokesDirac} together with \eqref{eq:ports} is equivalent to the original system \ref{Sys:pHDCv}. For this, we need to plug \eqref{eq:ports} into 
\begin{align*}
\begin{bmatrix}\Sflo\\\Rflo\\\Pflo\end{bmatrix}+\begin{bmatrix}
J(z)&\so&B\\-\so&\sz&\vect{0}\\-B^T&\vect{0}&0
\end{bmatrix}\begin{bmatrix}\Seff\\\Reff\\\Peff\end{bmatrix}=\begin{bmatrix}
\vect{0}\\\vect{0}\\0\end{bmatrix}
\end{align*}
and rearrange the equations, i.e.,
\begin{align*}
-E(z)\partial_tz&=-J(z)\frac{\delta \Ham}{\delta z}(z)-BT_\infty+R(z)\Rflo,\\
\Rflo&=\frac{\delta \Ham}{\delta z}(z),\\
\Pflo&=B^T\frac{\delta \Ham}{\delta z}(z).
\end{align*}
Multiplying the first equation by $-1$ and inserting $\Rflo$ yields,
\begin{align*}
E(z)\partial_tz&=(J(z)-R(z))\frac{\delta \Ham}{\delta z}(z)+BT_\infty,\\
\Pflo&=B^T\frac{\delta \Ham}{\delta z}(z),
\end{align*}
which is equivalent to System \ref{Sys:pHDCv} with $\Pflo=y$, which acts as an output in the port-Hamiltonian framework.
\end{proof}

\section*{Statements and Declarations}
\subsection*{Competing interests}
The authors have no competing interests to declare that are relevant to the content of this article.
\subsection*{Data availability statement}
Data sets generated during the current study are available from the corresponding author on reasonable request.

\bibliographystyle{abbrv}
\bibliography{Bib}


\end{document}